\documentclass[11pt]{amsart}
\usepackage[babel]{csquotes}
\usepackage{enumitem}
\usepackage{amsmath,amsthm,amssymb,mathrsfs,amsfonts,verbatim,enumitem,color,leftidx,mathtools}
\usepackage{mathabx}
\usepackage{etoolbox} 
\usepackage{bbm}
\usepackage[all,tips]{xy}
\usepackage{graphicx,ifpdf}
\usepackage{stmaryrd}
\usepackage{amssymb}
\usepackage{dsfont}

\ifpdf
\fi
\usepackage[colorlinks]{hyperref}
\hypersetup{
linkcolor=blue,          
citecolor=green,        
}

\usepackage{tikz}

\newtheorem{thm}{Theorem}[section]
\newtheorem{lem}[thm]{Lemma}

\newtheorem{prop}[thm]{Proposition}

\theoremstyle{definition}

\theoremstyle{remark}

\numberwithin{equation}{section}
\definecolor{esperance}{rgb}{0.0,0.5,0.0}

\newcommand{\bd}{\mathbf{d}}

\newcommand{\bm}{\mathbf{m}}

\newcommand{\bp}{\mathbf{p}}
\newcommand{\bq}{\mathbf{q}}

\newcommand{\bt}{\mathbf{t}}

\newcommand{\bx}{\mathbf{x}}
\newcommand{\by}{\mathbf{y}}

\newcommand{\supp}{\mathrm{supp\,}}

\newcommand{\del}{\delta}

\newcommand{\eps}{\epsilon}

\newcommand{\sig}{\sigma}

\newcommand{\om}{\omega}

\newcommand{\cB}{\mathcal{B}}
\newcommand{\cC}{\mathcal{C}}
\newcommand{\cD}{\mathcal{D}}
\newcommand{\cE}{\mathcal{E}}
\newcommand{\cF}{\mathcal{F}}

\newcommand{\cI}{\mathcal{I}}
\newcommand{\cJ}{\mathcal{J}}

\newcommand{\cL}{\mathcal{L}}
\newcommand{\cM}{\mathcal{M}}
\newcommand{\cN}{\mathcal{N}}

\newcommand{\cP}{\mathcal{P}}

\newcommand{\cS}{\mathcal{S}}
\newcommand{\cT}{\mathcal{T}}

\newcommand{\cZ}{\mathcal{Z}}

\newcommand{\bR}{\mathbb{R}}
\newcommand{\bZ}{\mathbb{Z}}
\newcommand{\bQ}{\mathbb{Q}}

\newcommand{\bN}{\mathbb{N}}

\newcommand{\bT}{\mathbb{T}}
\newcommand{\bS}{\mathbb{S}}

\newcommand{\sB}{\mathsf{B}}

\newcommand{\SL}{\operatorname{SL}}

\newcommand{\ASL}{\operatorname{ASL}}

\newcommand{\Id}{\operatorname{Id}}

\newcommand\set[1]{\left\{#1\right\}}

\newcommand\on[1]{\operatorname{#1}}
\newcommand\diag[1]{\operatorname{diag}\left(#1\right)}
\newcommand\mb[1]{\mathbf{#1}}

\newcommand\tb[1]{\textbf{#1}}

\newcommand\crly[1]{\mathscr{#1}}

\newcommand{\wstar}{\overset{\on{w}^*}{\lra}}
\newcommand{\Supp}{\on{Supp}}

\newcommand{\defn}{\overset{\on{def}}{=}}

\newcommand{\lra}{\longrightarrow}

\newcommand{\onto}{\xymatrix{\ar@{>>}[r]&}}

\newcommand{\eq}[1]
{
\begin{equation*}
{#1}
\end{equation*}
}
\newcommand{\eqlabel}[2]
{
\begin{equation}
{#2}\label{#1}
\end{equation}
}

\makeatletter
\newcommand*{\rom}[1]{\expandafter\@slowromancap\romannumeral #1@}
\makeatother

\begin{document}

\title[Effective equidistribution in $\SL_d(\bR)\ltimes \bR^d/\SL_d(\bZ)\ltimes \bZ^d$]{Effective equidistribution of expanding translates in the space of affine lattices}
\author{Wooyeon Kim}
\address{Wooyeon Kim. Department of Mathematics, ETH Z\"{u}rich, 
{\it wooyeon.kim@math.ethz.ch}}

\thanks{}



\def\thefootnote{}
\def\thefootnote{\arabic{footnote}}
\setcounter{footnote}{0}

\begin{abstract}
We prove a polynomially effective equidistribution result for expanding translates in the space of $d$-dimensional affine lattices for any $d\ge 2$.
\end{abstract}

\maketitle
\section{Introduction}
\subsection{Backgrounds}
In the theory of unipotent dynamics on homogeneous spaces, a fundamental result is Ratner's theorems \cite{Rat91a, Rat91b} on measure rigidity, topological rigidity, and orbit equidistribution. They have numerous and diverse applications ranging from number theory to mathematical physics, see e.g. \cite{EMS98, EM04, EO06, Sha09, Sha10, Mrk10, MS10}. We refer the reader to \cite{Mor05} for expositions and references.

A limitation of Ratner's results is that they do not tell us an explicit rate of density or equidistribution for unipotent orbits. As pointed out in \cite[Problem 7]{Mrg00}, establishing \textit{effective} versions of Ratner's results has been a central topic in homogeneous dynamics.

We say that a subgroup $H\subset G$ is (unstable) horospherical if there exists an element $a\in G$ such that
\eqlabel{horo1}{H=\set{g\in G: a^kga^{-k}\rightarrow e \quad\textrm{as}\quad k\to -\infty}.}
For horospherical subgroups, effective equidistribution results have been established with a \textit{polynomially} strong error term \cite{KM96, KM12, Ven10, DKL16, KSW17, Shi19} using the mixing property of $a$ on $G/\Gamma$ and the so-called ``thickening" method, originating in Margulis' thesis \cite{Mrg04}. We also refer the reader to the effective equidistribution results for horocycle orbits in $\SL_2(\bR)/\Gamma$ \cite{Sar81,Bur90,FF03,Str04,Str13} via direct representation-theoretic approaches.  

Here we state an effective equidistribution result of the horospherical orbits in a concrete setting as Theorem \ref{equX} below; it is closely related to the main results of the present paper. Let $G=\SL_d(\bR)$, $\Gamma=\SL_d(\bZ)$, and $X=G/\Gamma$ for fixed $d\ge 2$. Let $\{a_t\}< G$ be a 1-parameter diagonal subgroup $\diag{e^{nt}\Id_m,e^{-mt}\Id_n}$, where $\Id_m$ and $\Id_n$ are $m\times m$ and $n\times n$ identity matrices, respectively, and $m+n=d$. Then the horospherical subgroup for $a:=a_1$ in $G$ is indeed expressed as
\eqlabel{horo2}{H=\set{\left(\begin{matrix} \Id_m& A\\
0&  \Id_n
\end{matrix}\right)\in G: A\in \operatorname{Mat}_{m,n}(\bR)}.}
We denote by $m_H$ the Haar measure of $H$ and $m_X$ the $G$-invariant probability measure on $X$. In \cite{KM96}, an effective equidistribution theorem of expanding horospherical translates in $X$ was established as follows.

\begin{thm}\label{equX}\cite[Proposition 2.4.8]{KM96}
Let $V\subset H$ be a fixed neighborhood of the identity in $H$ with smooth boundary and compact closure. 
Then there exists a constant $\del_0>0$ only depending on $m$ and $n$ so that the following holds. For any compact set $K\subset X$, there exists a constant $T(K)\ge 0$ such that
\begin{equation}\label{eeqx}
  \frac{1}{m_H(V)}\int_{V}f(a_tux)dm_H(u)=\int_X fdm_X+O(\cS(f)e^{-\del_0 t})  
\end{equation}
for any $t\ge T(K)$, $f\in C_c^{\infty}(X)$, and $x\in K$. Here, $\cS$ is a suitable Sobolev norm of $C_c^{\infty}(X)$ and the implied constant depends only on $m$,$n$, and $V$.
\end{thm}

For non-horospherical subgroups, several effective equidistribution results have been proved in the last decade. When the ambient group $G$ is semisimple, an effective equidistribution theorem was established for closed orbits of a semisimple group in general homogeneous spaces \cite{EMV09} (see also \cite{AELM20}), and was recently extended to the adelic setting \cite{EMMV20}. Also, effective results for non-horospherical unipotent subgroups are known in some concrete settings \cite{Mo12,LM14,SU15,Ubi16,CY19}. When $G$ is unipotent, \cite{GT12} settled effective equidistribution of polynomial orbits in nilmanifolds.

In general, let $\hat{G}=G\ltimes W$ be the ambient group where $G$ is a semisimple group and $W$ is the unipotent radical of $\hat{G}$. If $\hat{G}$ is neither semisimple nor unipotent, the methods of the aforementioned results do not seem directly applicable to general homogeneous spaces. However, in special settings effective equidistribution results have been obtained e.g. \cite{Str15,BV16} for $\hat{G}=\SL_2(\bR)\ltimes\bR^2$, \cite{SV20} for $\hat{G}=\SL_2(\bR)\ltimes(\bR^2)^{\oplus k}$, and \cite{Pri18} for $\hat{G}=\SL_3(\bR)\ltimes\bR^3$. The purpose of the present paper is to establish an effective equidistribution result for $\hat{G}=\SL_d(\bR)\ltimes\bR^d$.

\subsection{Main results}
For fixed $d\ge 2$, denote
\begin{align*}
    G&=\SL_d(\bR),\quad \hat{G}=\SL_d(\bR)\ltimes \bR^d,\\
    \Gamma&=\SL_d(\bZ),\quad \hat{\Gamma}=\SL_d(\bZ)\ltimes \bZ^d=Stab_{\hat{G}}(\bZ^d).
\end{align*}
We define $X=G/\Gamma$ and $Y=\hat{G}/\hat{\Gamma}$. It is sometimes convenient to view $\hat{G}$ as a subgroup of $\SL_{d+1}(\bR)$ by $\hat{G}=\set{\left(\begin{matrix} g& b\\
0&  1
\end{matrix}\right): g\in \SL_d(\bR), b\in \bR^d}$. Then there exists a natural projection $\pi: Y\rightarrow X$ sending $\left(\begin{matrix} g& b\\
0&  1
\end{matrix}\right)\hat{\Gamma}\in Y$ to $\left(\begin{matrix} g& 0\\
0&  1
\end{matrix}\right)\Gamma\in X$. We denote by $W\simeq\bR^d$ the unipotent radical of $\hat{G}$, which consists of all translations on $\bR^d$. Since $W=\set{w(b):= \left(\begin{matrix} \Id_d& b\\
0&  1
\end{matrix}\right): b\in\bR^d}$ acts simply transitively on $\bR^d$, we can identify $\hat{G}/W\simeq G$.  We take a lift of the element $g\in G$ to $\hat{G}\subset \SL_{d+1}(\bR)$ given by $\left(\begin{matrix} 
g& 0\\
0& 1
\end{matrix}\right)$ and by abuse of notation we denote it again by $g$. Following this notation, we have a relation $gw(b)=w(gb)g$ for $g\in G$ and $b\in\bR^d$. For $g\Gamma\in X$ and $b\in\bR^d$, $gw(b)\hat{\Gamma}$ is invariant under the integer translations of $b$. Thus, the fiber $\pi^{-1}(g\Gamma)$ can be seen as $\set{gw(b)\hat{\Gamma}\in Y: b\in\bT^d}$.

Let $\{a_t\}$ be a $1$-parameter diagonal subgroup of $G$ defined by $$a_t=\diag{e^{nt}\Id_m,e^{-mt}\Id_n},$$ $H$ be the unstable horospherical subgroup in $G$ for $a=a_1$, and $m_H$ be the Haar measure of $H$ as before. We remark that $H$ is a full horospherical subgroup in $G$, \emph{but not in} $\hat{G}$, which renders the methods in the proof of Theorem \ref{equX} irrelevant. We denote by $m_X$ the $G$-invariant probability measure on $X$ and $m_Y$ the $\hat{G}$-invariant probability measure on $Y$.

Let $V\subset H$ be a fixed neighborhood of the identity in $H$ with smooth boundary and compact closure. For $y\in Y$ and $t\ge 0$, $a_tVy\subset Y$ is a lift of a piece of a full horospherical orbit $a_tVx\subset X$, where $x=\pi(y)\in X$. In this paper, we study the equidistribution of the expanding translates $\set{a_tVy}_{t\ge 0}$ in $Y$ as $t\to\infty$.

For any positive integer $q$, set
\eq{X_q:=\set{gw(b)\hat{\Gamma}\in Y: g\in G, b\in\bQ^d, \bq(b)=q},}
where for any $b\in\bQ^d$, $\bq(b)$ denote its common denominator, i.e. the smallest $q\in\bN$ such that $b\in q^{-1}\bZ^d$. Clearly, $X_q$ is a closed $G$-invariant subset of $Y$ for any $q$. We denote by $m_{X_q}$ the $G$-invariant probability measure on $X_q$.

We start by discussing the \textit{ineffective} equidistribution of $\set{a_tVy}_{t\ge 0}$ in $Y$. For $y\in Y$ and $t\ge 0$, we denote by $\mu_{y,t}\in\crly{P}(Y)$ the normalized probability measure on the orbit $a_tVy$, i.e. $\mu_{y,t}(f):=\frac{1}{m_H(V)}\int_V f(a_tuy)dm_H(u)$ for any $f\in C_b(Y)$. The following equidistribution result is a special case of \cite[Thm. 1.4]{Sha96}:
\eqlabel{irreq}{\mu_{y,t}\wstar m_Y \quad \mbox{for all } y\notin\bigcup_{q\in\bN}X_q,}
\eqlabel{rateq}{\mu_{y,t}\wstar m_{X_q}\quad \mbox{for all } y\in X_q,}
as $t\to\infty$. We refer to \cite[proof of Thm. 5.2]{MS10} and \cite[§1.1]{Str15} for the proof and more detailed discussions. Here we just make a remark that the proof of \cite[Thm. 1.4]{Sha96} crucially relies on Ratner's measure classification theorem. 

The main result of the present paper is Theorem \ref{mainthm} below, which is an \textit{effective} refinement of \eqref{irreq}. Observe that the weak$^*$ limit of $\set{\mu_{y,t}}_{t\ge 0}$ is determined by the Diophantine properties of the initial point $y$. Also, the condition $y\notin\bigcup_{q\in\bN}X_q$ in \eqref{irreq} is equivalent to $b\in\bT^d\setminus\bQ^d$, where $y=gw(b)\hat{\Gamma}$. Based on these observations, we can expect that the rate of convergence in \eqref{irreq} might also depend on the Diophantine properties of $y$, or equivalently that of $b$. 

We say that a vector $b\in\bT^d$ is \textit{of (Diophantine) type} $M\ge 1$ if there exists $c>0$ such that $|b-\frac{\bp}{q}|>cq^{-M}$ for all $q\in\bN$ and $\bp\in\bZ^d$. We first state a simplified version of the main result.

\begin{thm}\label{DioCor}
Let $V\subset H$ be a fixed neighborhood of the identity in $H$ with smooth boundary and compact closure. 
Then there exists a constant $\del>0$ only depending on $d$ such that
\begin{equation}\label{eeqDio}
  \frac{1}{m_H(V)}\int_{V}f(a_tuy)dm_H(u)=\int_Y fdm_Y+O\left(\cS(f)e^{-\frac{\del t}{M+1}}\right)  
\end{equation}
for any $t\ge 0$, $f\in C_c^{\infty}(Y)$, and $y=gw(b)\hat{\Gamma}$, where $b\in\bT^d$ is of Diophantine type $M$. Here, $\cS$ is a suitable Sobolev norm of $C_c^{\infty}(Y)$ and the implied constant depends only on $d$,$V$, and $y$.
\end{thm}

More generally, in order to express the Diophantine properties of $b$ in a quantitative sense, let us define a function $\zeta:(\bT^d\setminus\bQ^d)\times\bR^+\to\bN$ by
$$\zeta(b,T):=\min\set{N\in\bN: \displaystyle\min_{1\leq q\leq N}\|qb\|_\bZ\leq \frac{N^2}{T}},$$
where $\|\cdot\|_{\bZ}$ denotes the supremum distance from $0\in\bT^d$.
We note that $\zeta(b,\cdot)$ is non-decreasing, unbounded, and has the following properties:
\eqlabel{rescaling}{\zeta(b,cT)\leq \lceil c^{\frac{1}{2}}\zeta(b,T)\rceil,}
\eqlabel{gammab}{\zeta(b,\|\gamma^{-1}\|_{\textrm{op}}^{-1}T)\leq\zeta(\gamma b,T)\leq\zeta(b,\|\gamma\|_{\textrm{op}}T),}
\eqlabel{Dirichlet}{\zeta(b,T)\leq \lceil T^{\frac{d}{2d+1}}\rceil.}
for any $\gamma\in\Gamma$ and $c>1$. Here $\|\cdot\|_{op}$ denotes the operator norm of $G$ w.r.t. the supremum norm on $\bR^d$.
The inequality \eqref{Dirichlet} is equivalent to Dirichlet's approximation theorem.

\begin{thm}\label{mainthm}
Let $V$ and $\cS$ be as in Theorem \ref{DioCor}. Then there exists a constant $\del'>0$ only depending on $d$ such that
\begin{equation}\label{eeqy}
  \frac{1}{m_H(V)}\int_{V}f(a_tuy)dm_H(u)=\int_Y fdm_Y+O\left(\cS(f)\zeta(b,e^{\frac{nt}{2}})^{-\del'}\right)  
\end{equation}
for any $t\ge 0$, $f\in C_c^{\infty}(Y)$, and $y=gw(b)\hat{\Gamma}$ with $\|g\|\leq \zeta(b,e^{\frac{nt}{2}})^{\del'}$ and $b\in\bT^d$. Here, the implied constant depends only on $d$ and $V$.
\end{thm}

By definition of $\zeta$, we have 
\eqlabel{MDio}{\zeta(b,T)\gg T^{\frac{1}{M+1}}}
for any $T>0$ and $b$ of Diophantine type $M$. Hence, Theorem \ref{mainthm} directly implies Theorem \ref{DioCor} with $\del=\frac{n\del'}{2}$. Moreover, Theorem \ref{mainthm} implies \eqref{irreq} directly as $\zeta(b,T)\to\infty$ as $T\to\infty$ for any irrational $b\in\bT^d$.

We give more general versions of Theorem \ref{DioCor} and Theorem \ref{mainthm} for other diagonal subgroups in Section \ref{gendiag}.

Str\"ombergsson \cite{Str15} proved Theorem \ref{mainthm} for $\SL_2(\bR)\ltimes\bR^2/\SL_2(\bZ)\ltimes\bZ^2$ using Fourier analysis and methods of from analytic number theory, and later Prinyasart \cite{Pri18} obtained an effective equidistribution result in $\SL_3(\bR)\ltimes\bR^3/\SL_3(\bZ)\ltimes\bZ^3$ building on similar methods. In particular, their proofs rely on delicate analysis on the explicit expressions of unitary representations of $\SL_2(\bR)$ or $\SL_3(\bR)$ and cancellations from Kloosterman sums. In order to prove Theorem \ref{mainthm} for arbitrary dimension $d\ge 2$, we will use a different approach from the previous works. The method in the present paper is \textit{geometric} in the sense that we do not rely on representation-theoretic properties of the semisimple component $G=\SL_d(\bR)$. It proceeds via a geometric decomposition of $X=G/\Gamma$ and analysis of the Fourier coefficients of probability measures on $\bT^d$. We will use the effective equidistribution result in $X$ (Theorem \ref{equX} above), and a strategy that is similar in spirit to the idea given in \cite{BFLM11}.

\subsection{Sketch of the proof of Theorem \ref{mainthm}}
We fix an initial point $y_0=g_0w(b_0)\hat{\Gamma}$ and let $\mu_t\in\crly{P}(Y)$ be the normalized probability measure on the orbit $a_tVy_0$, for any $t\ge 0$. Then Theorem \ref{mainthm} is equivalent to obtaining an upper bound of $|\mu_t(f)-m_Y(f)|$ for $f\in C_c^\infty(Y)$. 

For any $x\in X$, $\pi^{-1}(x)\subset Y$ can be identified with $\bT^d$, so $Y$ is a torus bundle of $X$. Moreover, if we fix a fundamental domain $\cF\subset G$ of $X$, any $y\in Y$ can be uniquely parametrized as $y=gw(b)\widehat{\Gamma}$, where $g\in\cF$ and $b\in\bT^d$. Hence, we have a bijective measurable map $y\mapsto(g\Gamma,b)$ between $Y$ and $X\times \bT^d$. The projection to $X$ of this map is the canonical projection $\pi: Y\to X$. We denote by $\sigma:Y\to\bT^d$ the projection to $\bT^d$ of this map, which depends on the choice of $\cF$.

We first decompose the orbit $a_tVy_0\subset Y$ to small pieces so that each piece of orbit is approximately on the same fiber torus. To simplify notations assume that $b_0\in\bT^d$ is of Diophantine type $M$. Let $\set{\om_i}_{i\in I}$ be a partition of unity of $X$ such that each $\om_i$ is supported on a ball $B^X(z_i,r)$, where $z_i\in X$, $r\asymp e^{-\frac{\kappa t}{M+1}}$, and $\kappa>0$ will be an appropriately chosen small constant. Then the measure $\mu_t$ is decomposed by $\sum_i \pi_*\mu_t(\om_i)\mu_{t,\om_i}$, where $\mu_{t,\om_i}(f)=\pi_*\mu_t(\om_i)^{-1}\int f\om_i\circ\pi d\mu_t$. Considering $\mu_{t,\om_i}$ as a measure on $B^X(z_i,r)\times \bT^d$, we may also look at $\nu_{t,\om_i}:=\sigma_*\mu_{t,\om_i}$ which is a probability measure on $\bT^d$ (approximating $\mu_{t,\om_i}$ up to error $\ll r$).

Hence to prove the effective equidistribution in $Y$, it is enough to show that for each $\om$, the measure $\nu_{t,\om}$ effectively equidistributes toward the Lebesgue measure on the torus. In other words, it suffices to prove that any nontrivial Fourier coefficient $\widehat{\nu_{t,\om}}(\bm_0)$ is small for $\bm_0\in\bZ^d\setminus\set{0}$. Most of this paper will be devoted to proving the decay of nontrivial Fourier coefficients, Proposition \ref{Fourdecay}.

Our basic approach to prove Proposition \ref{Fourdecay} is viewing the horospherical action in $Y$ as the combination of the horospherical action in $X$ and a discrete $\Gamma=\SL_d(\bZ)$-action on $\bT^d$. For example, for $s\ge 0$, $u\in V$, and $x=g\Gamma\in X$ such that $g\in\cF$, there uniquely exist $\xi_s(x,u)\in\cF$ and $\gamma_s(x,u)\in\Gamma$ such that $a_sug=\xi_s(x,u)\gamma_s(x,u)$. Then for $y=gw(b)\hat{\Gamma}$, a point on the horospherical orbit $a_sVy$ is described by $a_suy=\xi_s(x,u)w(\gamma_s(x,u)b)\hat{\Gamma}$. Observe that the dynamics on the torus part is given by $\SL_d(\bZ)$-action.

We already have an effective equidistribution result in $X$ thanks to \cite{KM96}: expanding horospherical translates in $X$ are polynomially effectively equidistributed by Theorem \ref{equX}. Our main argument to deal with the dynamics on the torus is inspired by \cite{BFLM11}, which proved a quantitative equidistribution result for $\SL_d(\bZ)$-random walk on the torus. Following the strategy of \cite{BFLM11} (see also \cite{HdS19,HLL20}), we will prove:
\begin{enumerate}
    \item (Proposition \ref{mainprop}) If $|\widehat{\nu_{t,\om}}(\bm_0)|$ is large for $\bm_0\in\bZ^d\setminus\set{0}$, then for appropriately chosen $s<t$, we can find a rich family of probability measures on the torus $\set{\nu_{t-s,j}}_{j\in\cJ'}$ such that $\supp\nu_{t-s,j}\subseteq\supp\nu_{t-s}$ for all $j\in\cJ'$, $\supp\nu_{t-s,j}$'s are disjoint, and each $\nu_{t-s,j}$ has a rich set of Fourier coefficients which are larger than a polynomial in $|\widehat{\nu_{t,\om}}(\bm_0)|$. 
    \item (Proposition \ref{lowcon}) If the torus-coordinate $b_0\in\bT^d$ of the initial point $y_0\in Y$ is not well-approximated by rationals of low denominator, then $\nu_{t-2s}$ cannot be highly concentrated on a $\rho$-radius ball, i.e. $\nu_{t-2s}(B^{\bT^d}(p,\rho))\ll \rho^{2\kappa'}$ for any $p\in\bT^d$, where $\rho\asymp e^{-\frac{\kappa''(t-2s)}{M+1}}$ for an appropriately chosen constant $\kappa', \kappa''>0$.
\end{enumerate}
Then Proposition \ref{Fourdecay} is obtained from (1) and (2) as follows. Making use of \cite[Proposition 7.5.]{BFLM11} which is a quantitative analog of Wiener's lemma, we deduce from Proposition \ref{mainprop} that if $|\widehat{\nu_{t,\om}}(\bm_0)|$ is large, then $\nu_{t-s,j}$'s are highly concentrated, namely for each $j\in\cJ'$ there exists $p_j\in\bT^d$ such that $\nu_{t-s,j}(B^{\bT^d}(p_j,\rho_0))\gg \rho_0^{0.1\kappa'}$ (Proposition \ref{highconcentrationj}) for $\rho_0\asymp \rho^n$. This information is not yet sufficient to derive a contradiction. Indeed, $p_j$'s might vary over $j\in\cJ'$ so it does not give high-concentration of $\nu_{t-s}$ directly. Thus we shall consider $\nu_{t-2s}=\sig_*\mu_{t-2s}=\sig_*\big((a_{-s})_*\mu_{t-s}\big)$, and deduce from Proposition \ref{highconcentrationj} that $\nu_{t-2s}(B^{\bT^d}(p,\rho))\gg \rho^{\kappa'}$ for some $p\in\bT^d$ (Proposition \ref{highcon}). Combining this with Proposition \ref{lowcon}, we conclude that $|\widehat{\nu_{t,\om}}(\bm_0)|$ must be small if $b_0$ is not well-approximated by rationals of low denominator, i.e. Proposition \ref{Fourdecay}.

We now briefly explain how to prove (1) and (2) above. In the setting of random walks as \cite{BFLM11}, the dynamics on $\bT^d$ is determined by the law of the random walk. In the case of horospherical translates, the function $\gamma_s(x,u)$ defined above plays a similar role as the law of the random walk. Proposition \ref{key} describes the dynamics on $\bZ^d$ induced by $\gamma_s(x,u)$, and is a crucial ingredient to prove both (1) and (2). It asserts that for any $\bm\in\bZ^d\setminus\set{0}$ and $x\in X$, the orbit $\gamma_s(x,u)^{tr}\bm$ with respect to $u\in V$ is well-separated on $\bZ^d$. 

For (1), the Fourier coefficient $\widehat{\nu_{t,\om}}(\bm_0)$ can be expressed as the average of Fourier coefficients at this $\Gamma^{tr}$-orbit, so well-separateness of such $\Gamma^{tr}$-orbit implies a rich set of large Fourier coefficients. Unlike \cite{BFLM11} which studied the random walk on a fixed torus, a technical difficulty in our setting is that the function $\gamma_s(x,u)$ and the fiber tori $\pi^{-1}(x)$ vary with the base point $x\in X$. For this reason, we decompose the measure $\mu_{t-s}$ again with finer boxes in $X$ so that $\gamma_s(x,u)$ is stable with respect to $x$ in each box.

For (2), we have the well-separateness of the orbit $\gamma_{t-s}(x,u)^{tr}e_1$ from Proposition \ref{key}, and it implies the well-separateness of the first coordinate of the orbit $\gamma_{t-s}(x,u)b_0$ under some Diophantine condition of $b_0$. It implies that the measure $\nu_{t-s}$ cannot be highly concentrated, so we obtain (2). An effective version of Weyl's well-known equidistribution criterion for an irrational rotation on the one-dimensional torus will be needed, and here is the place where we use the Diophantine properties of $b_0$.

\vspace{5mm}
\tb{Acknowledgments}. 
I would like to thank Manfred Einsiedler for numerous fruitful suggestions on earlier drafts of this paper. I am grateful to Taehyeong Kim and Seonhee Lim for many valuable conversations and their illuminative comments. I am also deeply indebted to Andreas Str\"ombergsson and the referees for pointing out that Lemma 3.7 in a previous version of this manuscript was wrong and for a lot of helpful comments on the revised version.

\section{Preliminaries}

Throughout this paper, we fix the dimension $d=m+n$, the diagonal group $a_t=\diag{e^{nt}\operatorname{Id}_m,e^{-mt}\operatorname{Id}_n}$, and the neighborhood $V\subset H$. We use the notation $\kappa_1,\kappa_2, ...$ to denote positive small exponents that depend only on the dimensions $m,n$. We shall use the standard notation $A\ll B$ or $A=O(B)$ to mean that $A\leq CB$ for some constant $C>0$ that depends only on $m,n$ and $V\subset H$. We shall also write $A\asymp B$ to mean both $A\ll B$ and $B\ll A$. The notation $C_1,C_2, ...$ ($c_1,c_2, ...$) will denote large (small) positive constants that depend only on $m,n$, and $V$.

\subsection{Metrics and norms}\label{subsec:Metricnorm}
Let $\bd^{\hat{G}}(\cdot,\cdot)$ be a right invariant Riemannian metric on $\hat{G}$. Then this metric induces metrics on $G$, $Y$, and $X$. We denote by $B^{\hat{G}}(\hat{g},r)$, $B^{G}(g,r)$, $B^{Y}(y,r)$, and $B^{X}(x,r)$ the balls with radius $r$ and centers $\hat{g}\in \hat{G}$, $g\in G$, $y\in Y$, and $x\in X$ with respect to the metric on $\hat{G}$, $G$, $Y$, $X$, respectively. Also, we set $\|g\|:=\displaystyle\max_{1\leq i,j\leq d}(|g_{ij}|,|(g^{-1})_{ij}|)$ for $g\in G$. Then for some constants $C_1,C_2>1$, we have (See \cite[§3.3]{EMV09}):
\eqlabel{norm}{\|g^{-1}\|=\|g\|=\|g^{tr}\|,\quad\|g_1g_2\|\leq C_1\|g_1\|\|g_2\|,}
\eqlabel{leftmetric}{\begin{aligned}
    &\bd^G(gg_1,gg_2)\leq C_2\|g\|^2\bd^G(g_1,g_2),\\ \bd^X(gx_1,gx_2)\leq C_2&\|g\|^2\bd^X(x_1,x_2),\quad \bd^Y(gy_1,gy_2)\leq C_2\|g\|^2\bd^Y(y_1,y_2)
\end{aligned}}
for any $g,g_1,g_2\in G$, $x_1,x_2\in X$, $y_1,y_2\in Y$.
For $v=(v_1,\cdots,v_d)\in\bR^d$, we use the supremum norm $\|v\|=\displaystyle\max_{1\leq i\leq d}|v_i|$. We denote by $\|g\|_{\textrm{op}}$ the operator norm of $g\in G$ with respect to the supremum norm of $\bR^d$. Then for some $C_3>1$,
\eqlabel{opnorm}{\|gv\|\leq \|g\|_{\textrm{op}}\|v\|, \quad \|g\|_{\textrm{op}}\leq C_3\displaystyle\max_{1\leq i,j\leq d}(|g_{ij}|)\leq C_3\|g\|}
for any $g\in G$ and $v\in\bR^d$.

\subsection{Compactness, height, and injectivity}
 For $x\in X$ we set:
\begin{align*}
    \textrm{ht}(x)&\defn\sup\set{\|gv\|^{-1}: x=g\Gamma, v\in\bZ^d\setminus\set{0}}\\
    K(\eps)&\defn\set{x\in X: \textrm{ht}(x)\leq\eps^{-1}}.
\end{align*}
Note that $\textrm{ht}(x)\ge 1$ for all $x\in X$. By Mahler's compact criterion, $K(\eps)$ is a compact set of $X$ for all $\eps>0$. Moreover, one can show that 
\eqlabel{cspmsr}{m_X(X\setminus K(\eps))\asymp \eps^d} for $\eps>0$
by using the following Siegel's integral formula \cite{Sie45}:
$$\int_X \tilde{f}dm_X=\int_{\bR^d}fdm_{\bR^d}$$
for any bounded any compactly supported function $f:\bR^d\to\bR_{\ge 0}$, where $\tilde{f}: X\to\bR_{\ge 0}$ is the $\textit{Siegel transform}$ of $f$ defined by $\tilde{f}(g\Gamma)=\displaystyle\sum_{v\in g\bZ^d\setminus\set{0}}f(gv)$. Applying the Siegel's integral formula with the characteristic function of $\eps$-radius ball centered at $0\in\bR^d$, \eqref{cspmsr} follows.

We also have the following estimate of the injectivity radius of $K(\eps)$ from \cite[Proposition 3.5]{KM12}. There exists a constant $C_4>1$ such that for any $0<r<\frac{1}{2}$ and $x\in K(C_4r^{\frac{1}{d}})$, the map $g\mapsto gx$ is injective on the ball $B^G(\operatorname{id},r)$ . In other words, there exists a bijective isometry between $B^G(\operatorname{id},r)x$ and $B^X(x,r)$.

\subsection{Sobolev norms}
Fix a basis $\cB$ for the Lie algebra $\mathfrak{g}$ of $G$, and a basis $\hat{\cB}$ for the Lie algebra $\hat{\mathfrak{g}}$ of $\hat{G}$ extended from $\cB$. The basis $\cB$ defines differentiation action of $\mathfrak{g}$ on $C_c^\infty(X)$ by $Zf(x)=\frac{d}{dt}f(\textrm{exp}(tZ)x)|_{t=0}$ for $f\in C_c^\infty(X)$ and $Z\in \cB$. The differentiation action of $\hat{\mathfrak{g}}$ on $C_c^\infty(Y)$ is defined similarly. We denote by $\triangledown f$ the gradient vector field on $G$ with respect to the basis $\cB$. Following \cite[§3.7]{EMV09}, we define $L^2$-Sobolev norms on $C_c^\infty(X)$ and $C_c^\infty(Y)$ for $k\in\bN$ as
$$\cS^X_k(f)^2:=\displaystyle\sum_{\cD}\|\textrm{ht}(x)^k\cD f\|^2_{L^2},$$
$$\cS^Y_k(f)^2:=\displaystyle\sum_{\cD}\|\textrm{ht}(\pi(y))^k\cD f\|^2_{L^2},$$
where $\cD$ ranges over all monomials in $\cB$, $\hat{\cB}$ of degree $\leq k$, respectively.
Let $l_0$ be an integer such that Theorem $\ref{equX}$ holds with the Sobolev norm $\cS^X=\cS^X_{l_0}$. Throughout this paper, we fix a sufficiently large integer $l\in\bN$ and the corresponding $l$-th degree Sobolev norm $\cS=\cS^Y_l$ on $C^{\infty}_c(Y)$ with the following properties: for $f\in C^{\infty}_c(Y)$,
\eqlabel{Sobol}{\cS^X(\overline{f})\leq\cS(f), \quad\|\cD f\|_{L^{\infty}(Y)}\leq\cS(f)}
for $\cD$ of degree $\leq d+2$. The function $\overline{f}\in C_c^{\infty}(X)$ is the average function over a fiber defined by $\overline{f}(x)=\int_{\pi^{-1}(x)}f(y)dm_{\pi^{-1}(x)}(y)$ for $x\in X$, where $m_{\pi^{-1}(x)}$ is the normalized Haar measure of $\pi^{-1}(x)$. Note that one can view $\pi^{-1}(x)$ and $m_{\pi^{-1}(x)}$ as the torus $\bT^d$ and the Lebesgue measure on the torus.

\subsection{Effective equidistribution in $X$ and measure estimates}
Considering the dependence on the initial point $x_0\in X$, Theorem \ref{equX} indeed can be formulated as follows: there exists a constant $\kappa_1>0$ such that
\eqlabel{equX'}{|\frac{1}{m_H(V)}\int_V f(a_tux_0)dm_H(u)-\int_X fdm_X|\ll \textrm{ht}(x_0)^{\kappa_1}\cS^X(f)e^{-\del_0 t}}
for any $t>0$, $x_0\in X$ and $f\in C_c^{\infty}(X)$. Keeping track of the dependence on $x_0\in X$ in the proof of \cite{KM96}, one can check that the error bound polynomially depends on the injectivity radius of $x_0$, which also polynomially depends on $\textrm{ht}(x_0)$.

Recall that $\mu_{y,t}$ denotes the normalized orbit measure on $a_tVy\subset Y$, and $\pi_*\mu_{y,t}$ is the normalized horospherical orbit measure on $a_tVx\subset X$, where $x=\pi(y)$. Thus, \eqref{equX'} enables us to estimate the measure of a small ball in $X$ with respect to $\pi_*\mu_{y,t}$: there exist $0<\kappa_2\leq\frac{1}{2}$ such that for any $e^{-\kappa_2 t}<r<\frac{1}{2}$ and $x\in X(C_4r^{\frac{1}{d}})$,
\eqlabel{ballest}{\pi_*\mu_{y,t}(B^X(x,r))\asymp r^{d^2-1}}
if $\textrm{ht}(\pi(y))<e^{\kappa_2 t}$. This estimate is from the fact $m_X(B^X(x,r))\asymp r^{d^2-1}$ for $x\in X(C_4r^{\frac{1}{d}})$ and the use of \eqref{equX'} with an approximating function of $B^X(x,r)$. 

Also, we can estimate how much measure appears near the cusp with respect to $\pi_*\mu_{y,t}$. We may assume that $\kappa_2>0$ was taken sufficiently small so that for any $e^{-\kappa_2 t}<\eps<\frac{1}{2}$,
\eqlabel{cuspest}{\pi_*\mu_{y,t}(X\setminus K(\eps))\asymp m_X(X\setminus K(\eps))\asymp \eps^d}
if $\textrm{ht}(\pi(y))<e^{\kappa_2 t}$. For this estimate, we used \eqref{cspmsr} and \eqref{equX'} with an approximating function of $K(\eps)$.

\subsection{The Horospherical subgroup}
For the unstable horospherical subgroup $H<G$ and the fixed neighborhood of identity $V\subset H$, $a_tVa_{-t}$ can be considered as the $e^{(m+n)t}$-dilated set of $V$. Thus, we have $m_H(a_tVa_{-t})=e^{(m+n)mnt}m_H(V)$ for any $t\in \bR$, since $\dim H=mn$. Also, this expanding property provides us an inductive structure of the measures of expanding translates $\set{\mu_{y,t}}_{t\ge 0}$ as follows, where $y$ is fixed.

\begin{prop}\label{inductive}
For any $0\leq s<t$, $u_0\in V$, $y\in Y$ and $f\in L^{\infty}(Y)$, we have
\eq{\mu_{y,t}(f)=(a_su_0)_*\mu_{y,t-s}(f)+O(\|f\|_{L^\infty} e^{-d(t-s)}),}
\end{prop}

\begin{proof}
First, we have
\begin{align*}
    (a_su_0)_*\mu_{y,t-s}(f)&=\frac{1}{m_H(V)}\int_V f(a_su_0a_{t-s}uy)dm_H(u)\\
    &=\frac{1}{m_H(V)}\int_V f(a_t(a_{-(t-s)}u_0a_{(t-s)})uy)dm_H(u),
\end{align*}
Denote by $\triangle$ the symmetric difference $A\triangle B=(A\setminus B)\cup (B\setminus A)$. Let $u'=a_{-(t-s)}u_0a_{(t-s)}$. Note that $a_{-(t-s)}Va_{(t-s)}$ is the $e^{d(t-s)}$-contracted set of $V$ and $u'\in a_{-(t-s)}Va_{(t-s)}$. Since $V$ has a smooth boundary, it follows that $m_H(u'V\triangle V)\ll e^{-d(t-s)}$. By a change of variable, we get
\begin{align*}
    (a_su_0)_*\mu_{y,t-s}(f)&=\frac{1}{m_H(V)}\int_{u'V} f(a_tuy)dm_H(u)\\
    &=\frac{1}{m_H(V)}\left(\int_{V} f(a_tuy)dm_H(u)+O(\|f\|_{L^\infty} m_H(u'V\triangle V))\right)\\
    &=\mu_{y,t}(f)+O(\|f\|_{L^\infty} e^{-d(t-s)}).
\end{align*}

\end{proof}

\subsection{$HH^0H^-$ decomposition of $G$}\label{HH0H-}
We denote by $H^0$ the centralizer of $a$ and $H^-$ be the stable horospherical group for $a$, i.e.
\eq{H^0=\set{\left(\begin{matrix} Z_{11}& 0\\
0&  Z_{22}
\end{matrix}\right)\in G: \operatorname{det}(Z_{11})\operatorname{det}(Z_{22})=1,\; Z_{11}\in\operatorname{GL}_m(\bR),\; Z_{22}\in\operatorname{GL}_n(\bR)},}
\eq{H^-=\set{\left(\begin{matrix} \Id_m& 0\\
A&  \Id_n
\end{matrix}\right)\in G: A\in \operatorname{Mat}_{n,m}(\bR)}.}
We also denote $\widetilde{H}:=HH^0$ and $\widetilde{H}':=H^0H^-$. Let us denote by $\cE\subset G$ the locus of the $n\times n$ lower-right minor, i.e.
$$\cE=\set{g=\left(\begin{matrix} Z_{11}& Z_{12}\\
Z_{21}&  Z_{22}\end{matrix}\right)\in G: \operatorname{det}(Z_{22})=0}.$$
Note that $\cE$ is a submanifold of $G$ with $\dim \cE<\dim G$, and for any element $g\in G\setminus\cE$ there uniquely exist $h_+\in H$, $h_0\in H^0$, and $h_-\in H^-$ such that $g=h_+h_0h_-$. Furthermore, the Haar measure $m_G$ can be also expressed in terms of the Haar measures of $H,H^0$, and $H^-$ as follows: 
\eqlabel{HaarmeasureChange}{\int fdm_G=\int_{H\times H^0\times H^-}f(h_+h_0h_-)\Delta_{\widetilde{H}}(h_0) dm_H(h_+)dm_{H^0}(h_0)dm_{H^-}(h_-)}
for any $f
\in L^1(G)$, where $dm_{H^0}$ and $dm_{H^-}$ are the Haar measures of $H^0$ and $H^-$, respectively, and $\Delta_{\widetilde{H}}(h_0)$ is the modular function of the group $\widetilde{H}$. 

Note that there exists a constant $C_5>1$ such that for any $0<r <\frac{1}{2}$
\eqlabel{Htildeball}{B^{\widetilde{H}}(\operatorname{id},r)\subseteq B^{H}(\operatorname{id},C_5r)B^{H^0}(\operatorname{id},C_5r),}
\eqlabel{Htilde'ball}{B^{\widetilde{H}'}(\operatorname{id},r)\subseteq B^{H^0}(\operatorname{id},C_5r)B^{H^-}(\operatorname{id},C_5r),}
\eqlabel{Gball}{B^G(\operatorname{id},r)\subseteq B^{\widetilde{H}}(\operatorname{id},C_5r)B^{H^-}(\operatorname{id},C_5r),}
\eqlabel{Gball'}{B^G(\operatorname{id},r)\subseteq B^{H}(\operatorname{id},C_5r)B^{\widetilde{H}'}(\operatorname{id},C_5r),}
where the metrics on $H^0$, $H^-$, $\widetilde{H}$, and $\widetilde{H}'$ are induced by the metric of $G$.

We now give an estimate of the size of $h_+,h_0$, and $h_-$ in the decomposition $g=h_+h_0h_-$. For any $g\in G\setminus \cE$ we have
\eqlabel{hdecompositionbound}{\|h_+\|,\|h_0\|,\|h_-\|\ll \max\set{\bd^G(g,\cE)^{-(d-1)},1}\|g\|^{d-1}.}
The proof of this estimate will be given in Appendix \ref{appendixA}.

\subsection{A fundamental domain of $X=\SL_d(\bR)/\SL_d(\bZ)$}\label{subfund} We will fix a fundamental domain of $X$ in order to parametrize elements in $Y$. The standard Siegel domain has been commonly used as a fundamental domain of $\SL_d(\bR)/\SL_d(\bZ)$. For $d\ge 3$, the standard Siegel domains are a weak fundamental domain that allows finite numbers of $g$ satisfying $x=g\Gamma$ for fixed $x\in X$, but it is not sufficient for our purpose. To avoid technical complexity, we construct a strong fundamental domain $\cF$ as follows in order to allow unique $g\in\cF$ satisfying $x=g\Gamma$. We define a continuous function $F:G\to\bR_{>0}$ by
$$F(g)^2=\frac{(\sum_{i,j}|g_{ij}|^2)(\sum_{i,j}|(g^{-1})_{ij}|^2)}{\sum_{i,j}|g_{ij}|^2+\sum_{i,j}|(g^{-1})_{ij}|^2},$$
where $g_{ij}$'s are matrix components of $g\in \SL_d(\bR)$. 
Hereafter, we fix a connected fundamental domain $\cF\subset G$ of $X$ which consists of every element in $\set{g\in G: F(g)< F(g\gamma) \mbox{  for all } \gamma\in\Gamma\setminus\set{\operatorname{id}}}$ and properly chosen elements from its boundary. Denote by $\phi: G\to X$ the canonical projection.
Then $\cF$ satisfies the following properties. See Appendix \ref{appendixA} for the concrete construction and the proofs.
\begin{enumerate}
    \item For any $x\in X$, there uniquely exists $g\in \cF$ such that $x=g\Gamma$. Moreover, we can define a measure-preserving map $\iota: X\to \cF$ such that $\phi\circ \iota=\operatorname{id}_{X\to X}$ and $\iota|_{\phi(\cF^\circ)}$ is continuous, where $\cF^\circ$ denotes the interior of $\cF$. We also have $m_G(\cF)=m_G(\cF^\circ)=m_X(\phi(\cF^\circ))=m_X(X)=1$.
    \item There exists $C_6>1$ such that for any $x\in X$, 
    \eqlabel{iotaht}{\|\iota(x)\|\leq C_6\textrm{ht}(x)^{d-1}.} Hence, for any $\eps>0$,
\eqlabel{fundest}{m_G(\set{g\in \cF: \|g\|> \eps^{-1}})\leq m_X\left(\set{x\in X: \textrm{ht}(x)> C_6^{-\frac{1}{d-1}}\eps^{-\frac{1}{d-1}}}\right)\leq C_7\eps^{\frac{d}{d-1}}} for some constant $C_7>1$ by \eqref{cspmsr}. 
    \item For $r>0$ and $\eps>0$, let 
    $$\cF(r,\eps):=\set{g\in\cF: \textrm{ht}(g\Gamma)\leq \eps^{-1}, \bd^G(g,\partial\cF)\ge r, \bd^G(g,\cE^{-1})\ge r^{\frac{1}{20d}}},$$ where $\partial\cF$ denotes the boundary of $\cF$ and $\cE^{-1}:=\set{g\in G: g^{-1}\in\cE}$. Let $\kappa_3:=\frac{1}{100d^3}$. Then there exists a constant $C_8>1$ satisfying 
\eqlabel{Fest}{m_G(\cF\setminus \cF(r,\eps))\leq C_8\max\set{r^{\kappa_3},\eps^d}} 
for any $r>0$ and $\eps>0$.
\end{enumerate}

\subsection{Parametrization of $Y$}
For a point $y\in Y$, it is often convenient to view $y$ as a point on the fiber $\pi^{-1}(x)$, where $x=\pi(y)$. The fiber $\pi^{-1}(x)$ can be identified with the torus $\bT^d$ by a map $b\mapsto gw(b)\hat{\Gamma}$, where $b\in\bT^d$ and $g\in G$ with $x=g\Gamma$, but this mapping depends on the choice of $g\in G$. For example, $gw(b)\hat{\Gamma}=(g\gamma^{-1})w(\gamma b)\hat{\Gamma}$ for any $\gamma\in \Gamma$.

However, if we fix a fundamental domain, then we have a well-defined parametrization of $Y$. Throughout this paper, we fix $\cF\subset G$ and $\iota$ as in §\ref{subfund}. For any $y\in Y$, there exists unique $b\in\bT^d$ such that $y=\iota(\pi(y))w(b)\hat{\Gamma}$. Denote this map $y\mapsto b$ by $\sigma: Y\to \bT^d$. Then we have a parametrization $y=\iota(\pi(y))w(\sigma(y))\hat{\Gamma}$. We remark that the map $\sigma$ is not defined canonically unlike $\pi: Y\to X$ and depends on the choice of the fundamental domain $\cF$.

\subsection{A partition of unity of $X$}
In this subsection, we will partition $X$ into small boxes, and construct a partition of unity of $X$. For $\mathbf{r}=(r_+,r_0,r_-)$ with $0<r_+,r_0,r_-<\frac{1}{2}$ we define
$$\sB_{\mathbf{r}}:=B^H(\operatorname{id},r_+)B^{H^0}(\operatorname{id},r_0)B^{H^-}(\operatorname{id},r_-)$$
and also write $\sB_{\mathbf{r}}(x):=B^H(\operatorname{id},r_+)B^{H^0}(\operatorname{id},r_0)B^{H^-}(\operatorname{id},r_-)x$ for $x\in X$. Using \eqref{HaarmeasureChange} we estimate the volume of the box $\sB_{\mathbf{r}}$ as follows: \eq{\begin{aligned}m_G(\mathsf{B}_{\mathbf{r}})&\asymp m_H(B^H(\operatorname{id},r_+))m_{H^0}(B^{H^0}(\operatorname{id},r_0))m_{H^-}(B^{H^-}(\operatorname{id},r_-))\\&\asymp r_+^{\dim H}r_0^{\dim H^0}r_-^{\dim H^-}.\end{aligned}} 

We may choose nonnegative smooth approximating functions $\psi_{r_+}^{H}\in C_c^{\infty}(H)$, $\psi_{r_0}^{H^0}\in C_c^{\infty}(H^0)$, and $\psi_{r_-}^{H^-}\in C_c^{\infty}(H^-)$ such that 
$$\Supp \psi_{r_+}^{H}\subseteq B^H(\operatorname{id},r_+), \; \Supp \psi_{r_0}^{H^0}\subseteq B^{H^0}(\operatorname{id},r_0), \; \Supp \psi_{r_-}^{H^-}\subseteq B^{H^-}(\operatorname{id},r_-), $$
$$\int_H \psi_{r_+}^{H}(h_+)dm_H(h_+)=\int_{H^0} \psi_{r_0}^{H^0}(h_0)dm_{H^0}(h_0)=\int_{H^-} \psi_{r_-}^{H^-}(h_-)dm_{H^-}(h_-)=1.$$ Moreover, it is also possible to take such $\psi_{r_+}^{H}$, $\psi_{r_0}^{H^0}$, and $\psi_{r_-}^{H^-}$ to satisfy 
$$\|\psi_{r_+}^{H}\|_{L^\infty(H)}\ll r_+^{-\dim H}, \qquad \|\triangledown\psi_{r_+}^{H}\|_{L^\infty(H)}\ll r_+^{-(\dim H+1)},$$
$$\|\psi_{r_0}^{H^0}\|_{L^\infty(H^0)}\ll r_0^{-\dim H^0},\qquad \|\triangledown\psi_{r_0}^{H^0}\|_{L^\infty(H^0)}\ll r^{-(\dim H^0+1)},$$
$$\|\psi_{r_-}^{H^-}\|_{L^\infty(H^-)}\ll r_-^{-\dim H^-},\qquad \|\triangledown\psi_{r_-}^{H^-}\|_{L^\infty(H^-)}\ll r^{-(\dim H^-+1)},$$ where the implied constants do not depend on $r_+$, $r_0$, and $r_-$. 

We now define a nonnegative smooth approximating function on $G$ by
$$\psi_{\mathbf{r}}(h_+h_0h_-):=\psi_{r_+}^{H}(h_+)\psi_{r_0}^{H^0}(h_0)\psi_{r_-}^{H^-}(h_-)\Delta_{\widetilde{H}}(h_0)^{-1},$$
where $h_+\in H$, $h_0\in H^0$, and $h_-\in H^-$. This function is clearly supported on $\mathsf{B}_{\mathbf{r}}$, and $\int \psi_{\mathbf{r}}dm_G=1$ by \eqref{HaarmeasureChange}. Moreover,
$\|\psi_{\mathbf{r}}\|_{L^{\infty}(G)}\ll m_G(\sB_{\mathbf{r}})^{-1}$ and $\|\triangledown\psi_{\mathbf{r}}\|_{L^{\infty}(G)}\ll \mathbf{r}_{\operatorname{min}}^{-d^2}$, where $\mathbf{r}_{\operatorname{min}}:=\min\set{r_+,r_0,r_-}$.

We remark that with this approximating function, one can estimate the measure of a $\mathsf{B}_{\mathbf{r}}$-box in $X$ with respect to $\pi_*\mu_{y,t}$, as in \eqref{ballest}. For any $e^{-\kappa_2 t}<\mathbf{r}_{\operatorname{min}}<\frac{1}{2}$ and $x\in X(C_4\mathbf{r}_{\operatorname{min}}^{\frac{1}{d}})$,
\eqlabel{boxest}{\pi_*\mu_{y,t}(\mathsf{B}_{\mathbf{r}}x)\asymp m_G(\mathsf{B}_{\mathbf{r}})\asymp r_+^{\dim H}r_0^{\dim H^0}r_-^{\dim H^-}}
if $\textrm{ht}(\pi(y))<e^{\kappa_2 t}$.

\begin{prop}\label{partition}
There exist constants $C_9, C_{10}, C_{11}>1$ such that the following holds. Let $\mathbf{r}=(r_+,r_0,r_-)$ for $0<r_-\leq r_0=r_+<\frac{1}{2C_9^3}$. Then there exist a set $\set{x_1,\cdots,x_{N_{\mathbf{r}}}}\subset K(C_{10} r_0^{\frac{1}{d}})$ with $N_{\mathbf{r}}\asymp r_+^{-\dim H}r_0^{-\dim H^0}r_-^{-\dim H^-}$ and a partition of unity $\set{\psi_{\mathbf{r},i}}_{i\in\cI_{\mathbf{r}}}$ of $X$ with $\cI_{\mathbf{r}}:=\set{1,\cdots,N_{\mathbf{r}}}\cup\set{\infty}$ satisfying the following properties:
\begin{itemize}
    \item $0\leq \psi_{\mathbf{r},i}\leq 1$ $\mbox{  for all } i\in\cI_{\mathbf{r}}$,
    \item $\mathds{1}_{\sB_{\mathbf{r}}(x_i)}\leq\psi_{\mathbf{r},i}\leq\mathds{1}_{\sB_{C_9^3\mathbf{r}}(x_i)}$ $\mbox{  for all } i\in\cI_{\mathbf{r}}\setminus\{\infty\}$,
    \item $\Supp\psi_{\mathbf{r},\infty}\subseteq X\setminus K(2C_{10}r_0^{\frac{1}{d}})$,
    \item $\displaystyle\sum_{i\in\cI_{\mathbf{r}}}\psi_{\mathbf{r},i}=\mathds{1}_X$,
    \item $\|\triangledown\psi_{\mathbf{r},i}\|_{L^\infty(X)}\leq C_{11}r_-^{-(d^2-mn)}r_0^{d^2-mn-1}$ $\mbox{  for all } i\in\cI_{\mathbf{r}}\setminus\{\infty\}$.
\end{itemize}
\end{prop}
\begin{proof}
We first claim that $\sB_{\mathbf{r}}^{-1}\sB_{\mathbf{r}}\subseteq \sB_{C_9\mathbf{r}}$ and $\sB_{\mathbf{r}}\sB_{\mathbf{r}}\subseteq \sB_{C_9\mathbf{r}}$ for $C_9:=6C_5^2$. Here we will only prove $\sB_{\mathbf{r}}^{-1}\sB_{\mathbf{r}}\subseteq \sB_{C_9\mathbf{r}}$ but the latter one is also obtained by a similar calculation. Since $B^H(\operatorname{id},r_+)B^{H^0}(\operatorname{id},r_0)$ is contained in $B^{\widetilde{H}}(\operatorname{id},2r_0)$, it is enough to show that $(\widetilde{h}h_{-})^{-1}\widetilde{h}'h_{-}'\in\sB_{C_9\mathbf{r}}$ for any $\widetilde{h},\widetilde{h}'\in B^{\widetilde{H}}(\operatorname{id},2r_0)$ and $h_{-},h_{-}'\in B^{H^-}(\operatorname{id},r_-)$. Indeed, 
\eq{\begin{aligned}
    (\widetilde{h}h_{-})^{-1}\widetilde{h}'h_{-}'&=(\widetilde{h}^{-1}\widetilde{h}')\big((\widetilde{h}^{-1}\widetilde{h}')^{-1}h_{-}^{-1}(\widetilde{h}^{-1}\widetilde{h}')\big)h_{-}'\\&\in B^{\widetilde{H}}(\operatorname{id},4r_0)B^G(\operatorname{id},2r_-)B^{H^-}(\operatorname{id},r_-)
\end{aligned}} and
\eq{\begin{aligned}B^{\widetilde{H}}(\operatorname{id},4r_0)&B^G(\operatorname{id},2r_-)B^{H^-}(\operatorname{id},r_-)\subseteq B^{\widetilde{H}}(\operatorname{id},(2C_5+4)r_0)B^{H^-}(\operatorname{id},(2C_5+1)r_-)\\&\subseteq B^{H^0}(\operatorname{id},6C_5^2r_0)B^{H^0}(\operatorname{id},6C_5^2r_0)B^{H^-}(\operatorname{id},3C_5r_-)\end{aligned}}
by \eqref{Htildeball} and \eqref{Gball}, hence the claim follows.

Take $\set{x_1,\cdots,x_{N_{\mathbf{r}}}}$ to be a maximal $\sB_{C_9\mathbf{r}}$-separated subset of $K(C_{10}r_0^{\frac{1}{d}})$, where $C_{10}:=C_4(3C_9)^{\frac{1}{d}}$. Then for any $1\leq i\leq N_{\mathbf{r}}$, the map $g\mapsto gx_i$ is injective on the box $\sB_{C_9\mathbf{r}}$. We get a partition $\cP_{\mathbf{r}}=\set{P_i}_{i\in\cI_{\mathbf{r}}}$ of $X$ by letting
$$P_i=\sB_{C_9^2\mathbf{r}}(x_i)\setminus\big(\displaystyle\bigcup_{j=1}^{i-1}P_j\cup\displaystyle\bigcup_{j=i+1}^{N_{\mathbf{r}}}\sB_{C_9\mathbf{r}}(x_j)\big)$$
for $1\leq i\leq N_{\mathbf{r}}$ and $P_\infty=X\setminus\displaystyle\bigcup_{j=1}^{N_{\mathbf{r}}}P_j$ inductively. Then for $1\leq i\leq N_{\mathbf{r}}$, $\sB_{C_9\mathbf{r}}(x_i)\subseteq P_i\subseteq \sB_{C_9^2\mathbf{r}}(x_i)$ by the construction. Since $m_X(P_i)\asymp m_G(\mathsf{B}_{\mathbf{r}})$ for each $1\leq i\leq N_{\mathbf{r}}$, we have $N_{\mathbf{r}}\asymp r_+^{-\dim H}r_0^{-\dim H^0}r_-^{-\dim H^-}$. To construct a partition of unity, we consider a non-negative smooth approximating function $\psi_{\mathbf{r}}\in C_c^{\infty}(G)$ as constructed above. Then $\psi_{\mathbf{r},i}:=\psi_{\mathbf{r}}*\mathds{1}_{P_i}$'s satisfy the desired properties. In particular, for any monomial $\cD\in\cB$ of degree $1$, $1\leq i\leq N_{\mathbf{r}}$, and $x\in X$ we have
\eq{\begin{aligned}
    |\cD\psi_{\mathbf{r},i}(x)|&=\left|\int_G \cD\psi_{\mathbf{r}}(g)\mathds{1}_{P_i}(g^{-1}x)dm_G(g)\right|\\
    &\leq \|\triangledown\psi_{\mathbf{r}}\|_{L^{\infty}(G)}\int_G \mathds{1}_{P_i}(g^{-1}x)dm_G(g)\\&\leq r_-^{-d^2}m_G(\sB_{C_9^2\mathbf{r}})\ll r_-^{-d^2}(r_+^{\dim H}r_0^{\dim H^0}r_-^{\dim H^-}),
\end{aligned}}
hence $\|\triangledown\psi_{\mathbf{r},i}\|_{L^\infty(X)}\ll r_-^{-(d^2-mn)}r_0^{d^2-mn-1}$ for all $i\in\cI_{\mathbf{r}}\setminus\{\infty\}$.
\end{proof}

In case of $r=r_+=r_0=r_-$, we will denote $\sB_r:=\sB_{r,r,r}$, $\psi_{r,i}:=\psi_{r,r,r,i}$, $N_{r}:=N_{r,r,r}$, and $\cI_{r}:=\cI_{r,r,r}$ for simplicity.

\subsection{Dual affine lattices: reduction to the case $m\leq n$}\label{dual}
In the later sections, we will need an assumption $m\leq n$. More precisely, we will need the property that the expanding rate $e^{nt}$ (when the diagonal flow $a_t$ acts on the unipotent radical $\bR^d$) is not slower than the contracting rate $e^{-mt}$. In this subsection, we claim that we may assume $m\leq n$.

Recall that an element $y$ in $Y=\operatorname{ASL}_d(\bR)/\operatorname{ASL}_d(\bZ)$ is parametrized by $y=gw(b)\hat{\Gamma}$, where $g\in G$ and $b\in\bT^d$. Let us consider the dual map $\varsigma:Y\to Y$ defined by $\varsigma(gw(b)\hat{\Gamma})=(g^{tr})^{-1}w(b)\hat{\Gamma}$. Since $\Gamma=\operatorname{SL}_d(\bZ)$ is preserved by the dual map $g\mapsto (g^{tr})^{-1}$, $\varsigma$ is well-defined. Note that the Haar measure $m_Y$ is also preserved by $\varsigma$, $\varsigma^2=\operatorname{id}_Y$, and $\varsigma(gy)=(g^{-1})^{tr}\varsigma(y)$ for any $g\in G$ and $y\in Y$.

Suppose that we already established Theorem \ref{mainthm} with $m\leq n$. If $n<m$, for $f\in C_c^\infty(Y)$ let us denote $f^*:=f\circ \varsigma$. Then we have
\eq{\begin{aligned}
    \frac{1}{m_H(V)}\int_{V}f(a_tuy)dm_H(u)&=\frac{1}{m_H(V)}\int_{V}f^*\big(\varsigma(a_tuy)\big)dm_H(u)\\&=\frac{1}{m_H(V)}\int_{V}f^*\big((a_t^{tr})^{-1}(u^{tr})^{-1}y\big)dm_H(u).
\end{aligned}}
Note that the push-forward of $m_H$ by the map $H\in u\mapsto (u^{tr})^{-1}\in H^{-}$ is $m_{H^{-}}$. Denoting by $V^{*}$ the image of $V$ by the map $u\mapsto (u^{tr})^{-1}$,
\eqlabel{dualintegral}{\frac{1}{m_H(V)}\int_{V}f(a_tuy)dm_H(u)= \frac{1}{m_{H^{-}}(V^*)}\int_{V^*}f^*(a_{-t}uy)dm_{H^{-}}(u).}

Let us consider a permutation matrix $M=\left(\begin{matrix}
    0_{m,n} & \operatorname{Id}_n \\ \operatorname{Id}_m & 0_{n,m}
\end{matrix}\right)\in G$ and define $f^*_M\in C_c^\infty(Y)$ by $f^*_M(y)=f^*(M^{-1}y)$. Then the conjugation of $M$ acts on $a_{-t}$ and $H^-$ as follows:
$$\operatorname{Ad}_M(a_{-t})=\diag{e^{mt}\Id_n,e^{-nt}\Id_m},$$
$$\operatorname{Ad}_M\left(\begin{matrix} \Id_m& 0\\
A&  \Id_n
\end{matrix}\right)=\left(\begin{matrix} \Id_n& A\\
0&  \Id_m
\end{matrix}\right)$$
for any $A\in \operatorname{Mat}_{n,m}(\bR)$. In particular, $\operatorname{Ad}_M(H^-)$ is the expanding horospherical group of $\operatorname{Ad}_M(a_{-1})$, and $(\operatorname{Ad}_M)_*(m_{H^-})$ is the Haar measure of the expanding horospherical group of $\operatorname{Ad}_M(a_{-1})$.

Observe that the role of $m$ and $n$ are now swapped. Writing 
$$f^*(a_{-t}uy)=f^*_M\big(\operatorname{Ad}_M(a_{-t})\operatorname{Ad}_M(u)(My)\big),$$ we may apply the assumed Theorem \ref{mainthm} with $m\leq n$:
$$\frac{1}{m_{H^{-}}(V^*)}\int_{V^*}f^*(a_{-t}uy)dm_{H^{-}}(u)=\int_Y f_M^*dm_Y+O\left(\cS(f_M^*)\zeta(b,e^{\frac{nt}{2}})^{-\del'}\right).$$
Note that $\int f^*_Mdm_Y=\int f^*dm_y=\int fdm_Y$, $My=Mgw(b)\hat{\Gamma}$, $\|Mg\|\ll\|g\|$, and $\cS(f^*_M)\ll \cS(f\circ\varsigma)$. To sum up, we obtain
\eqlabel{dualequidistribution}{\frac{1}{m_H(V)}\int_{V}f(a_tuy)dm_H(u)=\int_Y fdm_Y+O\left(\cS(f\circ \varsigma)\zeta(b,e^{\frac{nt}{2}})^{-\del'}\right).}
Regarding the Sobolev norm, suppose that the Sobolev norm in \eqref{dualequidistribution} is of degree $l$, i.e. $\cS=\cS^Y_l$. Since $\varsigma$ is an isometry on $X$, it holds that $\cS^X_l(f\circ \varsigma)\ll\cS^X_l(f)$. On each fiber, we have an estimate
$$\|gb\|\leq \| g\|_{\operatorname{op}}^2\|g^{*}b\|\leq C_3^2C_6^2\operatorname{ht}(g\Gamma)^{2d-2}\|g^{*}b\|$$
for any $g\in G$ and $b\in\bR^d$, using \eqref{opnorm} and \eqref{iotaht}. It follows that
$$ \cD(f\circ\varsigma)(y)\ll \operatorname{ht}(\pi(y))^{(2d-2)l}\cD f(y)$$
for any $y\in Y$ and $\cD$ of degree $\leq l$. Let $\cS'$ be the Sobolev norm of degree $(2d-1)l$, i.e. $\cS'=\cS^Y_{(2d-1)l}$. Then
$$\cS(f\circ \varsigma)^2=\displaystyle\sum_{\cD}\|\textrm{ht}(\pi(y))^l\cD (f\circ\varsigma)\|^2_{L^2}\ll \displaystyle\sum_{\cD}\|\textrm{ht}(\pi(y))^{(2d-1)l}\cD f\|^2_{L^2}=\cS'(f)^2,$$
hence the error term in \eqref{dualequidistribution} is $O(\cS'(f)\zeta(b,e^{\frac{nt}{2}})^{-\del'})$. Therefore, Theorem 1.3 with $m>n$ follows from the case $m\leq n$.

In the rest of this paper, we assume $m\leq n$.

\section{Inductive structures of the measures of expanding translates}
\subsection{Well-separateness of $\Gamma^{tr}$-orbit in $\bZ^d$}
For fixed $s>0$,
we define two maps $\xi=\xi_s: X\times V\to\cF$ and $\gamma=\gamma_s: X\times V\to\Gamma$ as follows: For any $x\in X$ and $u\in V$, there uniquely exist $\xi(x,u)\in\cF$ and $\gamma(x,u)\in\Gamma$ such that
\eqlabel{defgam'}{a_su\iota(x)=\xi(x,u)\gamma(x,u)}
by the definition of the fundamental domain $\cF$. The map $\xi(x,u)$ encodes the projected orbit of $a_su$-action onto $X$ as $a_sux=\xi(x,u)\Gamma$. Also, the map $\gamma(x,u)$ describes the orbit of $a_su$-action on the fiber tori as in the following lemma.

\begin{lem}\label{siggam}
For any $y\in Y$ and $u\in V$, we have $\sigma(a_suy)=\gamma(x,u)\sigma(y)$, where $x=\pi(y)$.
\end{lem}
\begin{proof}
It is a direct consequence of the definitions of $\sigma$ and $\gamma$:
\begin{align*}
    \sigma(a_suy)&=\sigma(a_su\iota(x)w(\sigma(y))\hat{\Gamma})\\
    &=\sigma(\xi(x,u)\gamma(x,u)w(\sigma(y))\hat{\Gamma})\\
    &=\sigma(\xi(x,u)w(\gamma(x,u)\sigma(y))\hat{\Gamma})=\gamma(x,u)\sigma(y).
\end{align*}
\end{proof}

For fixed $\bm_0\in\bZ^d\setminus\set{0}$, we define two maps $\bx:X\times V\to\bR^m$ and $\by:X\times V\to\bR^n$ satisfying the following:
\eqlabel{defxy}{(\xi(x,u)^{tr})^{-1}\bm_0=\left(\begin{matrix} \bx(x,u)\\
\by(x,u)
\end{matrix}\right)\in\bR^m\times\bR^n.}

\begin{lem}\label{Vxe}
For $0<\eps\leq\frac{1}{2}$ and $x\in X$ denote by $V_{x,\eps}$ the set of elements $u\in V$ satisfying:
\begin{enumerate}
    \item $\|\xi(x,u)\|<\eps^{-1},$
    \item $\|\bx(x,u)\|> \eps^2\|\bm_0\|.$
\end{enumerate}
If $\eps>e^{-\kappa_2 s}$ and $\textrm{ht}(x)<e^{\kappa_2 s}$, then $m_H(V\setminus V_{x,\eps})\leq C_{13}\eps^{\frac{\kappa_3}{2}}$ for some $C_{13}>1$.
\end{lem}
\begin{proof}
We first estimate the measure of the subset violating (1). If $\|\xi(x,u)\|\ge\eps^{-1}$, then $\textrm{ht}(a_sux)=\textrm{ht}(\xi(x,u)\Gamma)\ge C_6^{-\frac{1}{d-1}}\eps^{-\frac{1}{d-1}}$ by \eqref{iotaht}, hence $a_sux$ is not in $K(C_6^{\frac{1}{d-1}}\eps^{\frac{1}{d-1}})$. As we obtained the estimate \eqref{cuspest}, we have
\eqlabel{cuspest'}{m_H\left(\set{u\in V: a_sux\notin K(C_6^{\frac{1}{d-1}}\eps^{\frac{1}{d-1}})}\right)\asymp \eps^{\frac{d}{d-1}}}
if $\eps>e^{-\kappa_2 s}$, $\textrm{ht}(x)<e^{\kappa_2 s}$. Thus, $m_H(\set{u\in V: \|\xi(x,u)\|\ge\eps^{-1}})\ll \eps^{\frac{d}{d-1}}$.

Now we estimate the measure of the subset satisfying (1) but violating (2). If $u\in V$ satisfies (1), then $\|(\xi(x,u)^{tr})^{-1}\bm_0\|>C_3^{-1}\eps\|\bm_0\|$. Assume from now on $\|\bx(x,u)\|\leq \eps^2\|\bm_0\|$. Therefore $\|\by(x,u)\|\gg\eps\|\bm_0\|$. Moreover we can find some $h\in B^G(\operatorname{id},C_{12}\eps)$ so that $h(\xi(x,u)^{tr})^{-1}\bm_0=h\left(\begin{matrix} \bx(x,u)\\
\by(x,u)
\end{matrix}\right)\in \set{\mb{0}_m}\times\bR^{n}$, for some constant $C_{12}>1$.

Let $\cM:=\set{g\in G: (g^{tr})^{-1}\bm_0\in\set{\mb{0}_m}\times\bR^n}$, then $\xi(x,u)$ is in the $C_{12}\eps$-neighborhood of $\cM$. Denote by $\cC(\eps)\subset\cF$ the intersection of this neighborhood and $\cF$. It follows that $\xi(x,u)\Gamma\in\cC(\eps)\Gamma$. 

We now need to get an upper bound of $m_X(\cC(\eps)\Gamma)$. Note that the $\eps^{\frac{1}{2}}$-neighborhood of $\cF(\eps^{\frac{1}{2}}, C_4\eps^{\frac{1}{2d}})$ is still contained in $\cF$. Since $\cM$ is a $(d^2-m-1)$-dimensional analytic closed submanifold of $G$, we have $$m_G(\cC(\eps)\cap \cF(\eps^{\frac{1}{2}}, C_4\eps^{\frac{1}{2d}}))\ll (\eps^{\frac{1}{2}})^m m_G(\cC(\eps^{\frac{1}{2}})\cap \cF(\eps^{\frac{1}{2}}, C_4\eps^{\frac{1}{2d}}))\ll \eps^{\frac{m}{2}}.$$
On the other hand, by \eqref{Fest} we have $m_G(\cF\setminus \cF(\eps^{\frac{1}{2}}, C_4\eps^{\frac{1}{2d}})))\ll \eps^{\frac{\kappa_3}{2}}$. Therefore, we obtain $m_X(\cC(\eps)\Gamma)=m_G(\cC(\eps))\ll \eps^{\frac{m}{2}}+\eps^{\frac{\kappa_3}{2}}\ll \eps^{\frac{\kappa_3}{2}}$.
We remark that the implied constant is uniform on the choice of $\bm_0\in\bZ^d\setminus\set{0}$ since $\cM$ is determined by $\frac{\bm_0}{|\bm_0|}\in\bS^{d-1}$ and $\bS^{d-1}$ is compact. 

 On the other hand, if $u\in V$ satisfies (1) and violates (2), then $a_sux=\xi(x,u)\Gamma\in \cC(\eps)\Gamma\cap K(C_6^{\frac{1}{d-1}}\eps^{\frac{1}{d-1}})$. Let $f\in L^\infty(X)$ be the characteristic function on the set $\cC(\eps)\Gamma\cap K(C_6^{\frac{1}{d-1}}\eps^{\frac{1}{d-1}})$ and $\psi_{\frac{\eps}{10}}\in C^\infty_c(X)$ be a non-negative smooth approximating function on $B^G(\operatorname{id},\frac{\eps}{10})$ as defined in the proof of \ref{partition}. Applying \eqref{equX'} with $f*\psi_{\frac{\eps}{10}}$ as we obtained \eqref{ballest}, we have
\begin{align*}
    &m_H\left(\set{u\in V: a_sux\in \cC(\eps)\Gamma\cap K(C_6^{\frac{1}{d-1}}\eps^{\frac{1}{d-1}})}\right)\\&\qquad\qquad\qquad \asymp m_X\left(\cC(\eps)\Gamma\cap K(C_6^{\frac{1}{d-1}}\eps^{\frac{1}{d-1}})\right)\ll\eps^{\frac{\kappa_3}{2}}
\end{align*}
if $\eps>e^{-\kappa_2 s}$ and $\textrm{ht}(x)<e^{\kappa_2 s}$.
\end{proof}

The following density property of $\gamma(x,u)^{tr}$-orbit of $\bm_0\in\bZ^d$, where $x$ and $\bm_0$ are fixed, is a key ingredient to obtain an effective equidistribution on fiber tori.

\begin{prop}\label{key}
Let $s>0$ and $\bm_0\in\bZ^d\setminus\set{0}$. For any $e^{-\kappa_2 s}\leq\eps\leq\frac{1}{2}$ and $x\in X$ with $\textrm{ht}(x)<e^{\kappa_2 s}$, let $V_{x,\eps}$ be the set as in Lemma \ref{Vxe}. Then
\begin{enumerate}
    \item $\set{\gamma(x,u)^{tr}\bm_0\in \bZ^d: u\in V_{x,\eps}}\subseteq B^{\bZ^d}(0,R)$,\\
     where $R=C_{15}\eps^{-1}\textrm{ht}(x)^{d-1}\|\bm_0\|e^{ns}$,
    \item For any $\bm\in\bZ^d$,
    $$m_H(\set{u\in V_{x,\eps}: \gamma(x,u)^{tr}\bm_0=\bm})\leq C_{16}\eps^{-3n}e^{-dns}$$
\end{enumerate}
for some constants $C_{15}, C_{16}>1$.
\end{prop}
We remark that the estimate in this proposition is \textit{tight} in the following sense: If we consider $\eps$ and $\textrm{ht}(x)$ as constants, as the points $\gamma(x,u)^{tr}\bm_0$ belong to $B^{\bZ^d}(0,R)$ and the latter contains $\asymp R^d$ points, the estimate in (2) is compatible with each $\bm\in B^{\bZ^d}(0,R)$ obtaining the same weight $R^{-d}\asymp e^{-dns}$. In other words, roughly speaking, Proposition \ref{key} asserts that the orbit $\gamma(x,u)^{tr}\bm_0$ is well-distributed on $B^{\bZ^d}(0,R)$ if $e^{ns}$ is much bigger than $\eps^{-1}$ and $\textrm{ht}(x)$.
\begin{proof}
Recall that there is the canonical measure-preserving bijection $\varphi$ between $M_{m,n}(\bR)$ and $H$ defined by $\varphi(A)=\left(\begin{matrix} \Id_m & A\\
0 & \Id_n\end{matrix}\right)$. Since $V$ is a neighborhood of identity with compact closure, there exists a constant $C_{14}>1$ such that $\varphi(B^{\bR^{mn}}(0,C_{14}^{-1}))\subseteq V\subseteq \varphi(B^{\bR^{mn}}(0,C_{14}))$. Note also that $\|u\|\leq C_{14}$ for any $u\in V$.

(1) Recall that we have $\|\iota(x)\|\leq C_6\textrm{ht}(x)^{d-1}$, $\|\xi(x,u)\|\leq \eps^{-1}$, $\|u\|\leq C_{14}$, and $\|a_s\|_{\textrm{op}}=e^{ns}$. By the definition of $\xi$ and $\gamma$ we have a relation $\xi(x,u)^{-1}a_su\iota(x)=\gamma(x,u)$. Hence,
\begin{align*}
  \|\gamma(x,u)^{tr}\bm_0\|&=\|\iota(x)^{tr}u^{tr}a_s^{tr}(\xi(x,u)^{tr})^{-1}\bm_0\|  \\
  &=\|\iota(x)^{tr}\|_{\textrm{op}}\|u^{tr}\|_{\textrm{op}}\|a_s^{tr}\|_{\textrm{op}}\|(\xi(x,u)^{tr})^{-1}\|_{\textrm{op}}\|\bm_0\|\\
  &\leq C_3^3\|\iota(x)\|\|u\|\|a_s\|_{\textrm{op}}\|\xi(x,u)\|\|\bm_0\|\\
  &\leq C_3^3C_6C_{14}\eps^{-1}\textrm{ht}(x)^{d-1}\|\bm_0\|e^{ns}.
\end{align*}

(2) The equation $\gamma(x,u)^{tr}\bm_0=\bm$ can be written
\eqlabel{meqn}{a_s^{tr}(\xi(x,u)^{tr})^{-1}\bm_0=(u^{tr})^{-1}(\iota(x)^{tr})^{-1}\bm.} Let $v_{+}\in\bR^m$ and $v_{-}\in\bR^n$ be the vectors such that $\left(\begin{matrix} v_{+}\\
v_{-}\end{matrix}\right)=(\iota(x)^{tr})^{-1}\bm$. Here, $v_{+}$ and $v_{-}$ only depend on $\bm$, not on $u$.
For any $A\in \varphi^{-1}(V)$, by a straightforward computation we have 
\eqlabel{Aeqn}{
\begin{aligned}
    (u^{tr})^{-1}(\iota(x)^{tr})^{-1}\bm&=\left(\begin{matrix} v_{+}\\
-A^{tr}v_{+}+v_{-}\end{matrix}\right),
\end{aligned}} where $u=\varphi(A)$. On the other hand, the left hand side of \eqref{meqn} is equal to $\left(\begin{matrix} e^{ns}\bx(x,u)\\
e^{-ms}\by(x,u)\end{matrix}\right)$. Combining this with \eqref{Aeqn}, if $u=\varphi(A)$ is a solution of \eqref{meqn}, then
\eqlabel{strip}{
\begin{aligned}
    \|A^{tr}v_{+}-v_{-}\|\leq e^{-ms}\|\by(x,u)\|.
\end{aligned}
}
If $u\in V_{x,\eps}$, we have $\|\bx(x,u)\|> \eps^{2}\|\bm_0\|$ and 
$$\|\by(x,u)\|\leq \|(\xi(x,u)^{tr})^{-1}\bm_0\|\leq C_3\|\xi(x,u)\|\|\bm_0\|\leq C_3\eps^{-1}\|\bm_0\|$$ by \eqref{defxy} and the definition of $V_{x,\eps}$ in Lemma \ref{Vxe}. Moreover, $\|v_{+}\|>\eps^{2}\|\bm_0\|e^{ns}$ since $v_{+}=e^{ns}\bx(x,u)$. To sum up, a solution $u=\varphi(A)$ must be in a thin tube $\set{A\in M_{m,n}(\bR): \|A^{tr}v_{+}-v_{-}\|\leq C_3\eps^{-1}\|\bm_0\|e^{-ms}}$, where $\|v_{+}\|>\eps^{2}\|\bm_0\|e^{ns}$. We can consider $A$ as in the Euclidean space $\bR^{mn}$, then the volume of the intersection of this tube and the ball $B^{\bR^{mn}}(0,C_{14})$ is \\$\ll \left(C_{14}^{-1}\|v_{+}\|^{-1}(C_3\eps^{-1}\|\bm_0\|e^{-ms})\right)^n\ll \eps^{-3n}e^{-dns}$. Therefore, for any $\bm\in\bZ^{d}$, the $m_H$-measure of the set of $u\in V_{x,\eps}$ satisfying $\gamma_x(u)^{tr}\bm_0=\bm$ is less than $C_{16}\eps^{-3n}e^{-dns}$ for some constant $C_{16}>1$.
\end{proof}

\subsection{A relation between Fourier coefficients of measures}
In this subsection, we explore a relation between $\mu_{y_0,t}$ and $\mu_{y_0,t-s}$, which are two probability measures of expanding translates of $y_0$, and their Fourier coefficients via Proposition \ref{inductive}, where $y_0\in Y$ and $0<s<t$. 

For each $\widetilde{h}=h_+h_0\in \widetilde{H}$ with $h_+\in H$ and $h_0\in H^0$, we define a smooth diffeomorphism $\Psi_{\widetilde{h}}$ of $H$ onto itself by $$\Psi_{\widetilde{h}}(u)=\operatorname{Ad}(h_0)(u)h_+^{-1}=h_0uh_0^{-1}h_+^{-1}.$$ In particular, $u\widetilde{h}^{-1}=h_0^{-1}\Psi_{\widetilde{h}}(u)$. 

Let $c_{1}:=\inf\set{\bd^G(\operatorname{id},\gamma): \gamma\in \Gamma\setminus\set{\operatorname{id}}}$. From the discreteness of $\Gamma$, we have $c_{1}>0$. We start with the following lemma.
\begin{lem}\label{stab}
There exists a constant $C_{17}>1$ such that the following holds. Let $s>0$, $0<r<c_{2}$, $0<\eps<r^2$ and $\mathbf{r}=(r\eps,r\eps,r\eps e^{-ds})$ for some constant $0<c_{2}<(2C_9^4C_{17})^{-1}$. For $x\in X$ and $u\in V$, suppose that $\iota(x)\in\cF(C_9^4r\eps,C_4\eps^{\frac{1}{100d^2}})$ and $\xi(x,u)\in\cF(r,r^{\frac{1}{d-1}})$. If $x'\in \mathsf{B}_{C_9^3\mathbf{r}}(x)$, then there uniquely exist $\widetilde{h}\in B^{H}(\operatorname{id},C_9^3r\eps)B^{H^0}(\operatorname{id},C_9^3r\eps)$ and $h_-\in B^{H^-}(\operatorname{id},C_9^3r\eps e^{-ds})$ such that $x'=\widetilde{h}h_-x$, $\bd^X(a_sux,a_s\Psi_{\widetilde{h}}(u)x')\leq C_{17}r\eps$ and $\gamma(x,u)=\gamma(x',\Psi_{\widetilde{h}}(u))$.
\end{lem}
\begin{proof}
Recall that
\begin{align*}
    &a_su\iota(x)=\xi(x,u)\gamma(x,u),\\
    &a_s\Psi_{\widetilde{h}}(u)\iota(x')=\xi(x',\Psi_{\widetilde{h}}(u))\gamma(x',\Psi_{\widetilde{h}}(u)).
\end{align*}
Since $\iota(x)$ is in $\cF(C_9^4r\eps,C_4\eps^{\frac{1}{100d^2}})$, $\iota$ is an isometry on $\sB_{C_9^3\mathbf{r}}(x)$, so $\iota(x')$ is in $\sB_{C_9^3\mathbf{r}}\iota(x)$. It follows that there uniquely exist 
$$h_+\in B^{H}(\operatorname{id},C_9^3r\eps),\; h_0\in B^{H^0}(\operatorname{id},C_9^3r\eps), \; h_-\in B^{H^-}(\operatorname{id},C_9^3r\eps e^{-ds})$$ such that $\iota(x')=\widetilde{h}h_-\iota(x)$, where $\widetilde{h}=h_+h_0$. Observe that
\begin{align*}a_su\iota(x)&=a_suh_-^{-1}\widetilde{h}^{-1}\iota(x')=(a_suh_-^{-1}u^{-1}a_{-s})a_su\widetilde{h}^{-1}\iota(x')\\&=(a_suh_-^{-1}u^{-1}a_{-s})a_sh_0^{-1}\Psi_{\widetilde{h}}(u)\iota(x')\\&=(a_suh_-^{-1}u^{-1}a_{-s})h_0^{-1}a_s\Psi_{\widetilde{h}}(u)\iota(x').\end{align*}
Note that $\bd^G(\operatorname{id},h_0^{-1})<C_9^3r\eps$ and
\begin{align*}
    \bd^G(\operatorname{id},a_suh_-^{-1}u^{-1}a_{-s})&\asymp \|\operatorname{Ad}_{a_su}(-\log h_-)\|_{\mathfrak{g}} \\&\ll e^{ds}\|\operatorname{Ad}_{u}(-\log h_-)\|_{\mathfrak{g}}\ll e^{ds}\bd^G(\operatorname{id},h_-)\ll r\eps.
\end{align*} Hence, for some constant $C_{17}>1$ we have
\eqlabel{xidist'}{\bd^G(a_su\iota(x),a_s\Psi_{\widetilde{h}}(u)\iota(x'))\leq C_{17}r\eps<r.}

Since $a_su\iota(x)=\xi(x,u)\gamma(x,u)$ and $\xi(x,u)\in\cF(r,r^{\frac{1}{d-1}})$, \eqlabel{xidist''}{\begin{aligned}
    \bd^G(a_su\iota(x),a_s\Psi_{\widetilde{h}}(u)\iota(x'))&\ge \bd^X(\xi(x,u)\Gamma, \xi(x',\Psi_{\widetilde{h}}(u))\Gamma)\\&=\bd^G(\xi(x,u),\xi(x',\Psi_{\widetilde{h}}(u))).
\end{aligned}}
Combining \eqref{xidist'} and \eqref{xidist''}, we get 
\eqlabel{xidist'''}{\bd^G(\xi(x,u),\xi(x',\Psi_{\widetilde{h}}(u)))\leq C_{17}r\eps.}
Moreover, $\bd^X(a_sux,a_s\Psi_{\widetilde{h}}(u)x')\leq C_{17}r\eps$ from \eqref{xidist'}, and
\eqlabel{xidist4}{
\begin{aligned}
\bd^G&(\gamma(x,u),\gamma(x',\Psi_{\widetilde{h}}(u)))=\bd^G(\xi(x,u)^{-1}a_su\iota(x),\xi(x',\Psi_{\widetilde{h}}(u))^{-1}a_s\Psi_{\widetilde{h}}(u)\iota(x'))\\
&\leq \bd^G(\xi(x,u)^{-1}a_su\iota(x),\xi(x,u)^{-1}a_s\Psi_{\widetilde{h}}(u)\iota(x'))\\
&\qquad\qquad+\bd^G(\xi(x,u)^{-1}a_s\Psi_{\widetilde{h}}(u)\iota(x'),\xi(x',\Psi_{\widetilde{h}}(u))^{-1}a_s\Psi_{\widetilde{h}}(u)\iota(x'))\\&
\leq C_2\|\xi(x,u)\|^2\bd^G(a_su\iota(x),a_s\Psi_{\widetilde{h}}(u)\iota(x'))+\bd^G(\xi(x,u)^{-1},\xi(x',\Psi_{\widetilde{h}}(u))^{-1})\\&\leq C_2\|\xi(x,u)\|^2\big(\bd^G(a_su\iota(x),a_s\Psi_{\widetilde{h}}(u)\iota(x'))+\bd^G(\xi(x,u),\xi(x',\Psi_{\widetilde{h}}(u))\big)
\end{aligned}}
using \eqref{leftmetric} and the right-invariance of $\bd^G$. From \eqref{iotaht} and $\xi(x,u)\in\cF(r,r^{\frac{1}{d-1}})$ we have $\|\xi(x,u)\|\leq C_6r^{-1}$. By \eqref{xidist'} and \eqref{xidist'''}, the last line of \eqref{xidist4} is bounded above by $2C_2C_6^2C_{17}r^{-1}\eps< 2C_2C_6^2C_{17}r$. By choosing $c_{2}:=\min\set{(2C_9^4)^{-1}, \frac{c_{1}}{2C_2C_6^2C_{17}}}$, we obtain $\bd^G(\gamma(x,u),\gamma(x',u))<c_{1}$ and deduce $\gamma(x,u)=\gamma(x',u)$.
\end{proof}

Let $t>0$, $0<s\leq\kappa_4 t$, $e^{-\kappa_4 s}\leq \eps<c_{2}^{\frac{1}{\kappa_4}}$, and $\eps^{\kappa_4}<r\leq c_{2}$, where $\kappa_4:=\frac{\min\set{\kappa_2,\kappa_3}}{2000d^3}$, so that $r\eps e^{-ds}>e^{-\frac{\kappa_2}{2}t}$. In the rest of this paper, we fix $y_0\in Y$ with $\textrm{ht}(\pi(y_0))<e^{\frac{\kappa_2}{2}t}$ and denote $\mu_t:=\mu_{y_0,t}$ and $\mu_{t-s}:=\mu_{y_0,t-s}$ for simplicity. We also fix a partition of unity $\set{\psi_j}_{j\in\cJ}$ using Proposition \ref{partition} with $\mathbf{r}=(r\eps,r\eps,r\eps e^{-ds})$, where $\cJ=\set{1,\cdots,N_{\mathbf{r}}}\cup\set{\infty}$ and $N_{\mathbf{r}}\asymp (r\eps)^{-(d^2-1)}e^{dmns}$. Let $\set{x_1,\cdots,x_{N_{\mathbf{r}}}}$ be the corresponding $\mathsf{B}_{C_9\mathbf{r}}$-separated set as in Proposition \ref{partition}.

We also denote $\mathbf{r}'=(r\eps e^{-ds},r\eps,r\eps)$. Since $\operatorname{Ad}_{a_{-s}}$ expands $H^{-}$ with rate $e^{ds}$ and contracts $H$ with rate $e^{-ds}$, we have
\eqlabel{boxconversion}{a_{-s}\mathsf{B}_{\mathbf{r}}a_{s}=\mathsf{B}_{\mathbf{r}'}.}

Recall that we are denoting $\sB_r=\sB_{r,r,r}$. For $z\in X$ with $\textrm{ht}(z)<e^{\frac{\kappa_2}{2} t}$ and $\om\in C_c^{\infty}(X)$ such that $\mathds{1}_{\sB_r(z)}\leq\om\leq\mathds{1}_{\sB_{C_9^3r}(z)}$ and $\|\triangledown\om\|_{L^{\infty}(X)}\leq C_{11}r^{-1}$, let $\mu_{t,\om}$ be the probability measure on $Y$
defined by
\eqlabel{omdef}{\mu_{t,\om}(f):=\pi_*\mu_t(\om)^{-1}\int_Y \om\circ\pi(y)f(y)d\mu_t(y).}
Note that $\pi_*\mu_t(\sB_r(z))\leq\pi_*\mu_t(\om)\leq\pi_*\mu_t(\sB_{C_9^3r}(z))$ and $\kappa_2>0$ was taken to be sufficiently small so that we can apply \eqref{ballest}, hence we have
\eqlabel{ommsr}{C_{18}^{-1}r^{d^2-1}\leq\pi_*\mu_t(\om)\leq C_{18}r^{d^2-1}}
for some $C_{18}>1$. For $1\leq j\leq N_{\mathbf{r}}$, let $\mu_{t-s,j}$ be the probability measure on $Y$ defined by
\eqlabel{jdef}{\mu_{t-s,j}(f):=\alpha_j^{-1}\int_Y \psi_j\circ\pi(y)f(y)d\mu_{t-s}(y),}
where $\alpha_j:=\mu_{t-s}(\psi_j\circ\pi)$. Since $r\eps e^{-ds}>e^{-\frac{\kappa_2}{2}t}>e^{-\kappa_2(t-s)}$ we can also apply \eqref{boxest}. Hence, we have 
\eqlabel{jmsr}{C_{19}^{-1}(r\eps)^{d^2-1}e^{-dmns}\leq\alpha_j\leq C_{19}(r\eps)^{d^2-1}e^{-dmns},\quad \forall 1\leq j\leq N_{\mathbf{r}}}
for some $C_{19}>1$.

Taking push-forward of these measures under $\sigma: Y\to\bT^d$, we define probability measures on the torus as follows:
\eqlabel{nudef}{
\begin{aligned}
    &\nu_t:=\sigma_*\mu_t, \quad \nu_{t,\om}:=\sigma_*\mu_{t,\om}\\
    &\nu_{t-s}:=\sigma_*\mu_{t-s}, \quad \nu_{t-s,j}:=\sigma_*\mu_{t-s,j}, \quad 1\leq j\leq N_{\mathbf{r}}.
\end{aligned}}
We now observe the relation between $\widehat{\nu_{t,\om}}$ and $\widehat{\nu_{t-s,j}}$'s. The following proposition is the main purpose of this section. 

\begin{prop}\label{mainprop}
Let $t>0$, $0<s\leq\kappa_4 t$, $e^{-\kappa_4 s}\leq \eps<c_{2}^{\frac{1}{\kappa_4}}$, $\eps^{\kappa_4}<r\leq c_{2}$ and $\bm_0\in\bZ^d\setminus\set{0}$. For $\iota(z)\in \cF(C_9^4r,2C_{10}r^{\frac{1}{d}})$, let $\om\in C^\infty_c(X)$ be a function satisfying $\mathds{1}_{\sB_r(z)}\leq\om\leq\mathds{1}_{\sB_{C_9^3r}(z)}$ and $\|\triangledown\om\|_{L^{\infty}(X)}\leq C_{11}r^{-1}$ as above. For some constants $c_{3}, c_{5}>0$, if $|\widehat{\nu_{t,\om}}(\bm_0)|\ge c_{3}r^{-(d^2-1)}\eps^{\frac{\kappa_3}{2}}$, then there exists $\cJ'\subseteq\cJ\setminus\set{\infty}$ with $|\cJ'|\ge\eta|\cJ|$ such that for any $j\in\cJ'$ we have $\iota(x_j),\iota(a_{-s}x_j)\in \cF(C_9^4r\eps,C_4\eps^{\frac{1}{100d^2}})$ and
$$\left|\set{\bm\in B^{\bZ^d}(R): |\widehat{\nu}_{t-s,j}(\bm)|>\eta}\right|>\eta|B^{\bZ^d}(0,R)|,$$
where $\eta:=c_{5}\eps^{d^2+3n+\frac{\kappa_3}{2}}\|\bm_0\|^{-d}$ and $R=C_{15}\eps^{-d}\|\bm_0\|e^{ns}$.
\end{prop}

\begin{proof}
From the definitions \eqref{omdef} and \eqref{nudef},
\eqlabel{est1}{\begin{aligned}
    \widehat{\nu_{t,\om}}(\bm_0)&=\int_{\bT^d}e^{-2\pi i\bm_0\cdot b}d\nu_{t,\om}(b)\\
    &=\int_Y e^{-2\pi i \bm_0\cdot \sigma(y)}d\mu_{t,\om}(y)\\
    &=\pi_*\mu_t(\om)^{-1}\int_Y \om\circ\pi(y)e^{-2\pi i\bm_0\cdot\sigma(y)}d\mu_t(y).
\end{aligned}}
Recall that for any $f\in L^{\infty}(Y)$,
$$\mu_t(f)=\frac{1}{m_H(V)}\int_V (a_su)_*\mu_{t-s}(f)dm_H(u)+O(\|f\|_{L^\infty(Y)}e^{-(m+n)(t-s)})$$
from Proposition \ref{inductive}. Applying this to \eqref{est1},
\eqlabel{est2}{\begin{aligned}
    &\widehat{\nu_{t,\om}}(\bm_0)=\frac{\pi_*\mu_t(\om)^{-1}}{m_H(V)}\int_V\int_Y \om(a_su\pi(y))e^{-2\pi i\bm_0\cdot \sigma(a_suy)}d\mu_{t-s}(y)dm_H(u)\\&\qquad\qquad+O(\pi_*\mu_t(\om)^{-1}e^{-(m+n)(t-s)})\\
    &=\frac{\pi_*\mu_t(\om)^{-1}}{m_H(V)}\displaystyle\sum_{j\in\cJ}\int_V\int_Y \om(a_su\pi(y))\psi_j(\pi(y))e^{-2\pi i\bm_0\cdot \sigma(a_suy)}d\mu_{t-s}(y)dm_H(u)\\&\qquad\qquad+O(\pi_*\mu_t(\om)^{-1}e^{-(m+n)(t-s)}).
\end{aligned}}
Let $\cJ_{int}\subset\cJ$ be the set of $j\in\cJ$ such that $\iota(x_j),\iota(a_{-s}x_j)\in \cF(C_9^4r\eps,C_4\eps^{\frac{1}{100d^2}})$. By $\eqref{Fest}$ we have $m_G(\cF\setminus\cF(C_9^5r\eps,\frac{1}{2}C_4\eps^{\frac{1}{100d^2}}))\ll (r\eps)^{\kappa_3}$, since $\kappa_3=\frac{1}{100d^3}$. For any $j\in\cJ\setminus \cJ_{int}$, either $\iota(\mathsf{B}_{\mathbf{r}}x_j)$ or $\iota(\mathsf{B}_{\mathbf{r}'}a_{-s}x_j)$ is in $\cF\setminus\cF(C_9^5r\eps,\frac{1}{2}C_4\eps^{\frac{1}{10d^2}})$, and $m_G(\iota(\mathsf{B}_{\mathbf{r}}x_j))=m_G(\iota(\mathsf{B}_{\mathbf{r}'}a_{-s}x_j))\asymp N_{\mathbf{r}}^{-1}$. Note that $\iota(\mathsf{B}_{\mathbf{r}}x_j)$'s are disjoint, hence using \eqref{boxconversion} $\iota(\mathsf{B}_{\mathbf{r}'}a_{-s}x_j)=\iota(a_{-s}\mathsf{B}_{\mathbf{r}}x_j)$'s are also disjoint. Thus, $|\cJ\setminus\cJ_{int}|\ll (r\eps)^{\kappa_3}N_{\mathbf{r}}$. It follows that
\eqlabel{est3}{\begin{aligned}
    &\frac{\pi_*\mu_t(\om)^{-1}}{m_H(V)}\left|\displaystyle\sum_{j\in\cJ\setminus\cJ_{int}}\int_V\int_Y \om(a_su\pi(y))\psi_j(\pi(y))e^{-2\pi i\bm_0\cdot \sigma(a_suy)}d\mu_{t-s}(y)dm_H(u)\right|\\
    &\leq\frac{\pi_*\mu_t(\om)^{-1}}{m_H(V)}\displaystyle\sum_{j\in\cJ\setminus\cJ_{int}}\int_V\int_Y \om(a_su\pi(y))\psi_j(\pi(y))d\mu_{t-s}(y)dm_H(u)\\
    &\leq \pi_*\mu_t(\om)^{-1}\displaystyle\sum_{j\in\cJ\setminus\cJ_{int}}\pi_*\mu_{t-s}(\psi_j)\\
    &\asymp \pi_*\mu_t(\om)^{-1}|\cJ\setminus\cJ_{int}|N_{\mathbf{r}}^{-1}=O( \pi_*\mu_t(\om)^{-1}(r\eps)^{\kappa_3}).
\end{aligned}}
Here, we are using \eqref{jmsr} from the second last line to the last line. 

We now estimate the summation of \eqref{est2} over $j\in\cJ_{int}$. For each $j\in\cJ_{int}$ and $y$ as in the integral with respect to $\mu_{t-s}$, we may assume $a_su\pi(y)\in\supp\om$ and $\pi(y)\in \supp\psi_j$. It implies that $\xi(\pi(y),u)\in\supp\om\subseteq\cF(r,r^{\frac{1}{d-1}})$ and $\pi(y)\in\mathsf{B}_{C_9^3\mathbf{r}}(x_j)$, so the assumptions in Lemma \ref{stab} are satisfied for $x_j$ and $\pi(y)\in\supp\psi_j$. By Lemma \ref{stab}, for any $y\in\supp(\psi_j\circ\pi)$ there uniquely exist $\widetilde{h}(y)\in B^{H^0}(\operatorname{id},C_9^3r\eps)B^{H}(\operatorname{id},C_9^3r\eps)$ and $h_-(y)\in B^{H^-}(\operatorname{id},C_9^3r\eps e^{-ds})$ such that $\pi(y)=\widetilde{h}(y)h_-(y)x$,
\eqlabel{xidist}{\bd^X(a_sux_j,a_s\Psi_{\widetilde{h}(y)}(u)\pi(y))\leq C_{17}r\eps,}
\eqlabel{gammastab}{\gamma(x_j,u)=\gamma(\pi(y),\Psi_{\widetilde{h}(y)}(u))}
for any $u\in V$ and $j\in\cJ_{int}$. Since we are assuming $\|\triangledown\om\|_{L^{\infty}}\leq C_{11}r^{-1}$, by \eqref{xidist} we have
\eqlabel{omapprox}{|\om(a_sux_j)-\om(a_s\Psi_{\widetilde{h}(y)}(u)\pi(y))|\leq C_{11}C_{17}\eps.}
By Lemma \ref{siggam} and \eqref{gammastab}, we also have
\eqlabel{sigma}{\sigma(a_s\Psi_{\widetilde{h}(y)}(u)y)=\gamma(\pi(y),\Psi_{\widetilde{h}(y)}(u))\sigma(y)=\gamma(x_j,u)\sigma(y).}

Recall that $\Psi_{\widetilde{h}(y)}$ is a diffeomorphism on $H$ with \eqlabel{automorphismestimate}{m_H(\Psi_{\widetilde{h}(y)}^{-1}(V)\triangle V)\ll \|\Psi_{\widetilde{h}(y)}^{-1}-\operatorname{Id}_H\|_{\operatorname{op}}\ll \|\widetilde{h}(y)\|\ll r\eps,}
where $\|\Psi_{\widetilde{h}(y)}^{-1}-\operatorname{Id}_H\|_{\operatorname{op}}$ denotes the operator norm of $H\ni h\mapsto \Psi_{\widetilde{h}(y)}^{-1}(h)-h$.
Note that the Jacobian determinant of the diffeomorphism $u\mapsto \Psi_{\widetilde{h}(y)}(u)$ is a constant $D(y)$ independent of $u$. Indeed, the constant is explicitly $D(y)=\operatorname{det}(Z_1)^d$ where $\widetilde{h}(y)=h_+h_0$ and $h_0=\left(\begin{matrix}
    Z_1 & \\ & Z_2
\end{matrix}\right)\in H^0$. Moreover, $|D(y)-1|\ll |\operatorname{det}(Z_1)-1|\ll \bd^{\widetilde{H}}(\operatorname{id},\widetilde{h}(y))\ll r\eps$ for any $y\in \supp(\psi_j\circ \pi)$. Using \eqref{automorphismestimate} and change of variables we have
\eqlabel{changeofvariable}{\begin{aligned}
    &\displaystyle\sum_{j\in\cJ_{int}}\int_V \om(a_su\pi(y))\psi_j(\pi(y))e^{-2\pi i\bm_0\cdot \sigma(a_suy)}dm_H(u)\\&=\displaystyle\sum_{j\in\cJ_{int}}\int_{\Psi_{\widetilde{h}(y)}^{-1}(V)}\om(a_s\Psi_{\widetilde{h}(y)}(u)\pi(y))\psi_j(\pi(y))e^{-2\pi i\bm_0\cdot \sigma(a_s\Psi_{\widetilde{h}(y)}(u)y)}dm_H(u)\\&\qquad\qquad+O(|D(y)-1|)\\&=\displaystyle\sum_{j\in\cJ_{int}}\int_{V}\om(a_s\Psi_{\widetilde{h}(y)}(u)\pi(y))\psi_j(\pi(y))e^{-2\pi i\bm_0\cdot \sigma(a_s\Psi_{\widetilde{h}(y)}(u)y)}dm_H(u)+O(r\eps).
\end{aligned}}
It follows from \eqref{omapprox}, \eqref{sigma}, and \eqref{changeofvariable} that
\eqlabel{est4}{\begin{aligned}
    &\frac{\pi_*\mu_t(\om)^{-1}}{m_H(V)}\displaystyle\sum_{j\in\cJ_{int}}\int_V\int_Y \om(a_su\pi(y))\psi_j(\pi(y))e^{-2\pi i\bm_0\cdot \sigma(a_suy)}d\mu_{t-s}(y)dm_H(u)\\
    &=\frac{\pi_*\mu_t(\om)^{-1}}{m_H(V)}\displaystyle\sum_{j\in\cJ_{int}}\int_V\int_Y \om(a_sux_j)\psi_j(\pi(y))e^{-2\pi i\gamma(x_j,u)^{tr}\bm_0\cdot\sigma(y)}d\mu_{t-s}(y)dm_H(u)\\&\qquad\qquad+O(\pi_*\mu_t(\om)^{-1}r\eps)+O(\pi_*\mu_t(\om)^{-1}\eps)\\
    &=\frac{\pi_*\mu_t(\om)^{-1}}{m_H(V)}\displaystyle\sum_{j\in\cJ_{int}}\int_V\int_Y \om(a_sux_j)\psi_j(\pi(y))e^{-2\pi i\gamma(x_j,u)^{tr}\bm_0\cdot\sigma(y)}d\mu_{t-s}(y)dm_H(u)\\&\qquad\qquad+O(\pi_*\mu_t(\om)^{-1}\eps).
\end{aligned}}
 For $j\in\cJ_{int}$ and $\bm\in\bZ^d$, let $V_{x_j,\eps}$ be the set as in Lemma \ref{Vxe} and $E_{j,\bm}\subseteq V$ be the set $\set{u\in V_{x_j,\eps}: \gamma(x_j,u)^{tr}\bm_0=\bm}$ which was the set of (2) in Proposition \ref{key}. For each $j\in\cJ_{int}$, $\textrm{ht}(x_j)\leq\eps^{-1}$ by the definition of $\cJ_{int}$, so we have a disjoint partition $V_{x_j,\eps}=\displaystyle\bigcup_{\bm\in B^{\bZ^d}(0,R)}E_{j,\bm}$ by Proposition \ref{key}, where $R=C_{15}\eps^{-d}\|\bm_0\|e^{ns}$. By Lemma \ref{Vxe} and Proposition \ref{key}, we have the following measure estimates:
 \eqlabel{Vepsest}{m_H(V\setminus V_{x_j,\eps})\leq C_{13}\eps^{\frac{\kappa_3}{2}},}
 \eqlabel{Eest}{m_H(E_{j,\bm})\leq C_{16}\eps^{-3n}e^{-dns}}
for any $j\in\cJ_{int}$ and $\bm\in\bZ^d$. Denoting \eqlabel{betadef}{\beta_{j,\bm}:=\frac{1}{m_H(V)}\int_{E_{j,\bm}}\om(a_sux_j)dm_H(u),} 
each summand of the last line in \eqref{est4} can be expressed by
\eqlabel{est5}{\begin{aligned}
    &\int_V\int_Y \om(a_sux_j)\psi_j(\pi(y))e^{-2\pi i\gamma(x_j,u)^{tr}\bm_0\cdot\sigma(y)}d\mu_{t-s}(y)dm_H(u)\\
    &=\displaystyle\sum_{\bm\in B^{\bZ^d}(0,R)}\int_{E_{j,\bm}}\int_Y\om(a_sux_j)\psi_j(\pi(y))e^{-2\pi i\gamma(x_j,u)^{tr}\bm_0\cdot\sigma(y)}d\mu_{t-s}(y)dm_H(u)\\
    &\qquad\qquad +O\left(\int_{V\setminus V_{x_j,\eps}}m_X(\psi_j)dm_H(u)\right)\\
    &=\displaystyle\sum_{\bm\in B^{\bZ^d}(0,R)}\left(\int_{E_{j,\bm}}\om(a_sux_j)dm_H(u)\right)\left(\int_Y\psi_j(\pi(y))e^{-2\pi i\bm\cdot\sigma(y)}d\mu_{t-s}(y)\right)\\
    &\qquad\qquad +O(\eps^{\frac{\kappa_3}{2}} m_X(\psi_j))\\
    &=m_H(V)\left(\displaystyle\sum_{\bm\in B^{\bZ^d}(0,R)}\alpha_j\beta_{j,\bm}\widehat{\nu_{t-s,j}}(\bm)+O(\eps^{\frac{\kappa_3}{2}} m_X(\psi_j))\right)
\end{aligned}}
for any $j\in\cJ_{int}$. Combining \eqref{est2},\eqref{est3}, \eqref{est4} and \eqref{est5} all together,
\eqlabel{est6}{\begin{aligned}
    \pi_*\mu_t(\om)\widehat{\nu_{t,\om}}(\bm_0)&=\displaystyle\sum_{j\in\cJ_{int}}\displaystyle\sum_{\bm\in B^{\bZ^d}(0,R)}\alpha_j\beta_{j,\bm}\widehat{\nu_{t-s,j}}(\bm)\\
    &\qquad+O(\max(\eps^{\frac{\kappa_3}{2}},\eps,e^{-d(t-s))}))\\
    &=\displaystyle\sum_{j\in\cJ_{int}}\displaystyle\sum_{\bm\in B^{\bZ^d}(0,R)}\alpha_j\beta_{j,\bm}\widehat{\nu_{t-s,j}}(\bm)+O(\eps^{\frac{\kappa_3}{2}})
\end{aligned}}
if $s\leq\kappa_4 t$, $\eps\ge e^{-\kappa_4 s}$, and $r>\eps^{\kappa_4}$. Recall that $\pi_*\mu_t(\om)\asymp r^{d^2-1}$ from \eqref{ommsr}, hence we can find constants $c_{3}, c_{4}>0$ so that if $|\widehat{\nu_{t,\om}}(\bm_0)|\ge c_{3}r^{-(d^2-1)}\eps^{\frac{\kappa_3}{2}}$, then
\eqlabel{est7}{\left|\displaystyle\sum_{j\in\cJ_{int}}\displaystyle\sum_{\bm\in B^{\bZ^d}(0,R)}\alpha_j\beta_{j,\bm}\widehat{\nu_{t-s,j}}(\bm)\right|>c_{4}r^{(d^2-1)}|\widehat{\nu_{t,\om}}(\bm_0)|\ge c_{3}c_{4}\eps^{\frac{\kappa_3}{2}}.}

By \eqref{jmsr}, \eqref{Eest} and \eqref{betadef}, we have
\eqlabel{acond}{\alpha_j\beta_{j,\bm}\leq \alpha_j(C_{16}\eps^{-3n}e^{-dns})\leq \frac{C_{20}\eps^{-(d^2+3n)}\|\bm_0\|^{d}}{|B^{\bZ^d}(0,R)||\cJ_{int}|}}
for some constant $C_{20}>1$ since $|B^{\bZ^d}(0,R)|\asymp R^d$ and $|\cJ_{int}|\asymp N_{\mathbf{r}}$. Set
$$\cJ':=\set{j\in \cJ_{int}: \left|\displaystyle\sum_{\bm\in B^{\bZ^d}(0,R)}\alpha_j\beta_{j,\bm}\widehat{\nu_{t-s,j}}(\bm)\right|> \frac{c_{3}c_{4}\eps^{\frac{\kappa_3}{2}}}{2|\cJ_{int}|}}.$$
By definition of $\cJ'$ we have
$$ \left|\displaystyle\sum_{j\in\cJ_{int}\setminus \cJ'}\displaystyle\sum_{\bm\in B^{\bZ^d}(0,R)}\alpha_j\beta_{j,\bm}\widehat{\nu_{t-s,j}}(\bm)\right|<\frac{1}{2}c_{3}c_{4}\eps^{\frac{\kappa_3}{2}},$$
hence
\eqlabel{pigeonhole1}{\left|\displaystyle\sum_{j\in \cJ'}\displaystyle\sum_{\bm\in B^{\bZ^d}(0,R)}\alpha_j\beta_{j,\bm}\widehat{\nu_{t-s,j}}(\bm)\right|>\frac{1}{2}c_{3}c_{4}\eps^{\frac{\kappa_3}{2}}.}
On the other hand, \eqref{acond} implies that
\eqlabel{pigeonhole2}{\left|\displaystyle\sum_{\bm\in B^{\bZ^d}(0,R)}\alpha_j\beta_{j,\bm}\widehat{\nu_{t-s,j}}(\bm)\right|\leq\frac{C_{20}\eps^{-(d^2+3n)}\|\bm_0\|^{d}}{|\cJ_{int}|}}
for any $j\in\cJ_{int}$. Combining \eqref{pigeonhole1} and \eqref{pigeonhole2} we obtain $|\cJ'|\geq 2\eta|\cJ_{int}|$, where $c_{5}:=\frac{c_{3}c_{4}}{4C_{20}}$ and $\eta=c_{5}\eps^{d^2+3n+\frac{\kappa_3}{2}}\|\bm_0\|^{-d}$. Since $|\cJ_{int}|>\frac{|\cJ|}{2}$, it holds that $|\cJ'|\geq \eta|\cJ|$.

For $j\in \cJ'$ we define
$$I(j):=\set{\bm\in B^{\bZ^d}(0,R): \left|\widehat{\nu_{t-s,j}}(\bm)\right|> \eta}.$$
Then for any $j\in \cJ'$ we have 
$$\left|\displaystyle\sum_{\bm\in B^{\bZ^d}(0,R)\setminus I(j)}\alpha_j\beta_{j,\bm}\widehat{\nu_{t-s,j}}(\bm)\right|<\frac{c_{3}c_{4}\eps^{\frac{\kappa_3}{2}}}{4|\cJ_{int}|},$$
hence
$$\left|\displaystyle\sum_{\bm\in I(j)}\alpha_j\beta_{j,\bm}\widehat{\nu_{t-s,j}}(\bm)\right|>\frac{c_{3}c_{4}\eps^{\frac{\kappa_3}{2}}}{4|\cJ_{int}|}.$$
Using \eqref{acond} again, it follows that $|I(j)|>\eta|B^{\bZ^d}(0,R)|$. This concludes the proof.
\end{proof}

\section{Proof of Theorem \ref{mainthm}}
\subsection{High concentration of $\nu_{t-2s}$ from a large Fourier coefficient}
The aim of this subsection is to show that a large nontrivial Fourier coefficient $\widehat{\nu_{t,\om}}(\bm_0)$ implies the existence of a highly concentrated ball in $\bT^d$ with respect to $\nu_{t-2s}$. We make use of the following purely harmonic analytic lemma proved in \cite{BFLM11}, which is a quantitative analog of Wiener's lemma.
\begin{prop}\cite[Proposition 7.5.]{BFLM11}\label{BFLM}
Let $\cN(E;S)$ denote the covering number of $E\subset\bZ^d$ by $S$-balls. There exists $c_{6}>0$ so that if a probability measure $\nu$ on $\bT^d$ satisfies
\eqlabel{BFLMcond}{\cN\left(\set{\bm\in B^{\bZ^d}(0,R): |\widehat{\nu}(\bm)|>\eta}; S\right)>\lambda\left(\frac{R}{S}\right)^d}
with $S<const_d\cdot R$, then there exists an $S^{-1}$-separated set $A\subset \bT^d$ with
$$\nu\left(\displaystyle\bigcup_{p\in A}B^{\bT^d}(p,R^{-1})\right)>c_{6}(\eta\lambda)^3.$$
In particular, there exists $p_0\in\bT^d$ with $\nu(B^{\bT^d}(p_0,R^{-1}))>c_{6}(\eta\lambda)^3S^{-d}.$
\end{prop}

Combining Propositions \ref{mainprop} and \ref{BFLM}, we deduce the following proposition about the concentration of $\nu_{t-s,j}$ for each $j\in\cJ'$.

\begin{prop}\label{highconcentrationj}
    Let $t>0$, $0<s\leq\kappa_4 t$, $e^{-\kappa_4 s}\leq \eps<c_{2}^{\frac{1}{\kappa_4}}$, $\eps^{\kappa_4}\leq r< c_{2}$ and $\bm_0\in\bZ^d\setminus\set{0}$. Let $z$, $\om$, and $\cJ'$ be as in Proposition \ref{mainprop}. Then there exists a constant $c_{7}>0$ such that if $|\widehat{\nu_{t,\om}}(\bm_0)|\ge c_{3}r^{-(d^2-1)}\eps^{\frac{\kappa_3}{2}}$, then for any $j\in\cJ'$ there exists $p_j\in\bT^d$ with
    $$\nu_{t-s,j}\big(B^{\bT^d}(p_j,R^{-1})\big)>c_{7}\eta^6,$$
    where $\eta=c_{5}\eps^{d^2+3n+\frac{\kappa_3}{2}}\|\bm_0\|^{-d}$ and $R=C_{15}\eps^{-d}\|\bm_0\|e^{ns}$.
\end{prop}
\begin{proof}
    It directly follows from Proposition \ref{BFLMcond} with $\nu_{t-s,j}$, $R$, $\eta$ as in Proposition \ref{mainprop}, $S=1$, and $\lambda=3^d\eta$.
\end{proof}

With this high-concentration of $\nu_{t-s,j}$'s for each $j\in\cJ'$, in the rest of this subsection, we show that $\nu_{t-2s}$ is highly concentrated.

By definition of $\mu_{t-s}$ and $\mu_{t-2s}$ we have
\eqlabel{2stranslate}{\mu_{t-2s}=(a_{-s})_{*}\mu_{t-s}.}
For $j\in\cJ'$ we define $\widetilde{\psi}_{j}\in C_c^{\infty}(X)$ by $\widetilde{\psi}_j(x):=\psi_j(a_sx)$ and let $\mu_{t-2s,j}$ be the probability measure on $Y$ defined by
\eqlabel{jdef'}{\mu_{t-2s,j}(f):=\widetilde{\alpha}_j^{-1}\int_Y \widetilde{\psi}_j\circ\pi(y)f(y)d\mu_{t-2s}(y),}
where $\widetilde{\alpha}_j:=\mu_{t-2s}(\widetilde{\psi}_j\circ\pi)$.
By \eqref{jmsr} and \eqref{2stranslate} we have
\eqlabel{alphatildecounting}{\widetilde{\alpha}_j=\mu_{t-2s}(\widetilde{\psi}_j\circ\pi)=\mu_{t-s}(\psi_j\circ\pi)= \alpha_j\asymp (r\eps)^{d^2-1}e^{-dmns}.}
One can deduce from \eqref{2stranslate} and \eqref{jdef'} that
\eqlabel{2stranslate'}{\begin{aligned}\mu_{t-2s,j}(f)&=\widetilde{\alpha}_j^{-1}\int_Y \widetilde{\psi}_j\circ\pi(y)f(y)d\mu_{t-2s}(y)\\&=\widetilde{\alpha}_j^{-1}\int_Y \widetilde{\psi}_j\circ\pi(y)f(y)d(a_{-s})_*\mu_{t-s}(y)\\&=\widetilde{\alpha}_j^{-1}\int_Y \psi_j\circ\pi(y)f(a_{-s}y)d\mu_{t-s}(y)\\&=(a_{-s})_{*}\mu_{t-s,j}(f).\end{aligned}}

On the other hand, $\mu_{t-2s,j}$ is supported on $\pi^{-1}(\supp\widetilde{\psi}_j)=\pi^{-1}(a_{-s}\supp\psi_j).$ Recall that $\mathds{1}_{\sB_{\mathbf{r}}(x_i)}\leq\psi_{\mathbf{r},i}\leq\mathds{1}_{\sB_{C_9^3\mathbf{r}}(x_i)}$, $\mathbf{r}'=(r\eps e^{-ds},r\eps,r\eps)$, and the property \eqref{boxconversion}. It follows that $\sB_{\mathbf{r}'}a_{-s}x_j\subseteq \supp\widetilde{\psi}_j\subseteq \sB_{C_9^3\mathbf{r}'}a_{-s}x_j$ and
\eqlabel{psitildesupport}{\supp\mu_{t-2s,j}\subseteq\pi^{-1}(\supp\widetilde{\psi}_j)\subseteq \pi^{-1}(\sB_{C_9^3\mathbf{r}'}a_{-s}x_j).}

Let $\nu_{t-2s}:=\sigma_{*}\mu_{t-2s}$ and $\nu_{t-2s,j}:=\sigma_{*}\mu_{t-2s,j}$ for $j\in \cJ'$. We investigate the relation between $\nu_{t-s,j}$ and $\nu_{t-2s,j}$. In view of the notations of $\xi(x,u)$ and $\gamma(x,u)$, for any $x\in X$ we have
$$a_{s}\iota(a_{-s}x)=\xi(a_{-s}x,\operatorname{id})\gamma(a_{-s}x,\operatorname{id}).$$
For each $j\in\cJ'$ it follows that $\xi(a_{-s}x_j,\operatorname{id})=\iota(x_j)$. For simplicity let us write $\gamma_j:=\gamma(a_{-s}x_j,\operatorname{id})$ so that
$\iota(a_{-s}x_j)\gamma_j^{-1}=a_{-s}\iota(x_j)$. 

\begin{lem}\label{tubeconcentration}
    For any $j\in\cJ'$ and $f\in C(\bT^d)$ we have $\nu_{t-2s,j}=(\gamma_j^{-1})_*\nu_{t-s,j}$.
    In particular, under the assumptions in Proposition \ref{highconcentrationj}, 
    $$\nu_{t-2s,j}\big(\gamma_j^{-1}B^{\bT^d}(p_j,R^{-1})\big)>c_{7}\eta^6.$$
\end{lem}

\begin{proof}
    We first claim that $\gamma(a_{-s}x',\operatorname{id})=\gamma_j$ for any $x'\in\supp\psi_{j}$. Recall that $\iota(x_j),\iota(a_{-s}x_j)\in\cF(C_9^4r\eps,C_4\eps^{\frac{1}{100d^2}})$ from the construction of $\cJ'$ in Proposition \ref{mainprop}. It follows that $\iota(x')\in \sB_{C_9^3\mathbf{r}}\iota(x_j)$ and $\iota(a_{-s}x')\in\sB_{C_9^3\mathbf{r}'}\iota(a_{-s}x_j)$, hence $\bd^G(\iota(x'),\iota(x_j))<3C_9^3r\eps$ and $\bd^G(\iota(a_{-s}x'),\iota(a_{-s}x_j))<3C_9^3r\eps$. For simplicity let us write $\gamma=\gamma(a_{-s}x',\operatorname{id})$. From the relations 
    $$\iota(a_{-s}x_j)\gamma_j^{-1}=a_{-s}\iota(x_j), \qquad \iota(a_{-s}x')\gamma^{-1}=a_{-s}\iota(x'),$$ we have
    $\bd^G(\gamma_j\iota(a_{-s}x_j)^{-1},\gamma\iota(a_{-s}x')^{-1})=\bd^G(\iota(x_j)^{-1},\iota(x')^{-1})<3C_9^3r\eps$, hence
    $$3C_9^3r\eps>\bd^G(\gamma^{-1}\gamma_j,\iota(a_{-s}x')^{-1}\iota(a_{-s}x_j))\geq \bd^G(\gamma^{-1}\gamma_j,\operatorname{id})-3C_9^3r\eps.$$
    The claim then follows from the discreteness of $\Gamma$.

    By Lemma \ref{siggam}, $\sigma(a_{-s}y)=\gamma_j^{-1}\sigma(y)$ for any $y\in\supp\psi_j\circ\pi$. It follows that $\sigma_{*}\big((a_{-s})_*\mu_{t-s,j}\big)=(\gamma_j^{-1})_*\big(\sigma_*(\mu_{t-s,j})\big)=(\gamma_j^{-1})_*\nu_{t-s,j}$. Applying \eqref{2stranslate'}, we get $\nu_{t-2s,j}=\sigma_{*}\big((a_{-s})_*\mu_{t-s,j}\big)=(\gamma_j^{-1})_*\nu_{t-s,j}$. Moreover, combining with Proposition \ref{highconcentrationj} we get
    \eq{\nu_{t-2s,j}\big(\gamma_j^{-1}B^{\bT^d}(p_j,R^{-1})\big)=\nu_{t-s,j}\big(B^{\bT^d}(p_j,R^{-1})\big)>c_{7}\eta^6.}
\end{proof}

Recall that $\cJ'$ was constructed so that $\iota(a_{-s}x_j)\in\cF(C_9^4r\eps,C_4\eps^{\frac{1}{100d^2}})$. This implies that $\iota(a_{-s}x_j)^{-1}$ does not belong to the locus $\cE$, hence for $j\in\cJ'$ we may define $h_{+,j}\in H$, $h_{0,j}\in H^0$, and $h_{-,j}\in H^{-}$ by
\eqlabel{xjdecomposition}{\iota(a_{-s}x_j)^{-1}=h_{+,j}h_{0,j}h_{-,j}.}
Moreover, by definition of $\cF(C_9^4r\eps,C_4\eps^{\frac{1}{100d^2}})$ we have $$\|\iota(a_{-s}x_j)\|\ll \operatorname{ht}(\iota(a_{-s}x_j))^{d-1}\ll \eps^{-\frac{d-1}{100d^2}}$$ and 
$$\bd^G(\iota(a_{-s}x_j)^{-1},\cE)\geq \|\iota(a_{-s}x_j)\|^{-2}\bd^G(\iota(a_{-s}x_j),\cE^{-1}) \geq (C_9^4r\eps)^{\frac{1}{10d}}$$ for all $j\in\cJ'$. It follows from \eqref{hdecompositionbound} that
\eqlabel{hjbound}{\|h_{+,j}\|,\|h_{0,j}\|,\|h_{-,j}\|\ll (r\eps)^{-\frac{d-1}{10d}}\|\iota(a_{-s}x_j)\|^{d-1}\ll \eps^{-\frac{1}{5}}.}

Let us write $\gamma_j^{-1}B^{\bT^d}(p_j,R^{-1})=\widetilde{p}_j+\gamma_j^{-1}B^{\bR^d}(0,R^{-1})$ with $\widetilde{p}_j:=\gamma_j^{-1}p_j$. We now describe the shape of $\gamma_j^{-1}B^{\bR^d}(0,R^{-1})$. Recall that $\gamma_j^{-1}=\iota(a_{-s}x_j)^{-1}a_{-s}\iota(x_j)$, $\|\iota(x_j)\|_{
\operatorname{op}
}\ll \|\iota(x_j)\|\ll \operatorname{ht}(\iota(x_j))^{d-1}\ll \eps^{-\frac{d-1}{100d^2}}$, and $R=C_{15}\eps^{-d}\|\bm_0\|e^{ns}$. Then
\eq{\begin{aligned}
    \gamma_j^{-1}B^{\bR^d}(0,R^{-1})&=\iota(a_{-s}x_j)^{-1}a_{-s}\iota(x_j)B^{\bR^d}(0,R^{-1})\\&\subseteq\iota(a_{-s}x_j)^{-1}a_{-s}B^{\bR^d}(0,\eps^{d-\frac{1}{10}}e^{-ns})\\&\subseteq h_{+,j}h_{0,j}h_{-,j}\big(B^{\bR^{m}}(0,\eps^{d-\frac{1}{10}}e^{-2ns})\times B^{\bR^{n}}(0,\eps^{d-\frac{1}{10}}e^{-(n-m)s})\big)
\end{aligned}}
for sufficiently small $\eps$. Using \eqref{hjbound} and the property that $h_{0,j}$ and $h_{-,j}$ stabilizes the subspace $\set{\mathbf{0}_m}\times\bR^{n}$, we obtain
\eqlabel{tubeshape}{\gamma_j^{-1}B^{\bR^d}(0,R^{-1})\subseteq h_{+,j}\cT,}
where $\cT:=B^{\bR^{m}}(0,\eps^{d-1}e^{-2ns})\times B^{\bR^{n}}(0,\eps^{d-1}e^{-(n-m)s})$. Recall that we are assuming $m\leq n$ as discussed in Subsection \ref{dual}. Thus the tube $\cT$ is contained in $B^{\bR^d}(0,\eps^{d-1})$ and projected onto $\bT^d$ injectively.

\begin{lem}\label{concentrationlemma}
    There exist constants $c_{8}>0$ and $0<c_{9}<c_{2}^{\frac{1}{\kappa_4}}$ such that the following holds. Let $t>0$, $0<s\leq\kappa_4 t$, $e^{-\kappa_4 s}\leq \eps<c_{9}$, $\eps^{\kappa_4}\leq r< c_{9}^{\kappa_4}$ and $0<\|\bm_0\|<\eps^{-1}$. Let $z$, $\om$, and $\cJ'$ be as in Proposition \ref{mainprop} and suppose $|\widehat{\nu_{t,\om}}(\bm_0)|\ge c_{3}r^{-(d^2-1)}\eps^{\frac{\kappa_3}{2}}$. Then there exists $\widetilde{x}\in X$ with $\iota(\widetilde{x})\in\cF(C_9^4r\eps,C_4\eps^{\frac{1}{100d^2}})$ and $\cJ''\subseteq \cJ'$ with $|\cJ''|\geq c_{8}\eps^{5d^2}e^{dmns}$  such that the followings hold for any $j\in \cJ''$:
    \begin{enumerate}
        \item $\iota(a_{-s}x_j)\in B^G(\iota(\widetilde{x}),r\eps)$,
        \item $\set{h_{+,j}}_{j\in \cJ''}$ is a $r\eps^2 e^{-ds}$-separated set in $H$,
        \item $\nu_{t-2s}(\widetilde{p}_j+h_{+,j}\cT)>\eps^{50d^2}$.
    \end{enumerate}
\end{lem}
\begin{proof}
    Since $\set{\iota(a_{-s}x_j)}_{j\in\cJ'}\subset \cF(C_9^4r\eps,C_4\eps^{\frac{1}{100d^2}})$ and the set $\cF(C_9^4r\eps,C_4\eps^{\frac{1}{100d^2}})$ is covered by $\ll (r\eps)^{-(d^2-1)}$ balls of radius $r\eps$, we can find $\widetilde{x}\in X$ and $\cJ''\subset\cJ'$ such that $\iota(a_{-s}x_j)\in B^G(\iota(\widetilde{x}),r\eps)$ for $j\in\cJ''$ and $|\cJ''|\gg (r\eps)^{d^2-1}|\cJ'|$. Since $$|\cJ'|\geq \eta|\cJ|\gg (\eps^{d^2+3n+\frac{\kappa_3}{2}}\|\bm_0\|^{-d})(r\eps)^{-(d^2-1)}e^{dmns},$$ $|\cJ''|\geq c_{8}\eps^{5d^2}e^{dmns}$ for some constant $c_{8}>0$.

    To show (2) suppose that there exist $j\neq j'\in\cJ''$ such that $h_{+,j}h_{+,j'}^{-1}\in B^H(\operatorname{id},r\eps^2 e^{-ds})$. Let us write $\iota(a_{-s}x_j)^{-1}=h_{+,j}\widetilde{h}_j'$ and $\iota(a_{-s}x_{j'})^{-1}=h_{+,j'}\widetilde{h}_{j'}'$, where $\widetilde{h}_j'=h_{0,j}h_{-,j}$ and $\widetilde{h}_{j'}'=h_{0,j'}h_{-,j'}$. 
    
    Since both $\iota(a_{-s}x_j)$ and $\iota(a_{-s}x_{j'})$ are contained in a ball of radius $r\eps$, there uniquely exist $h_+\in B^H(\operatorname{id},2C_5r\eps)$ and $\widetilde{h}'\in B^{\widetilde{H}'}(\operatorname{id},2C_5r\eps)$ such that $\iota(a_{-s}x_j)^{-1}=h_+\widetilde{h}'\iota(a_{-s}x_{j'})^{-1}$, i.e. $(h_{+,j}\widetilde{h}_j')^{-1}=h_+\widetilde{h}'(h_{+,j'}\widetilde{h}_{j'}')^{-1}$. We can write this again by $h_+^{-1}(\widetilde{h}_j')^{-1}=\widetilde{g}'(h_{+,j'}^{-1}h_{+,j})$, where $\widetilde{g}'=\widetilde{h}'(\widetilde{h}_{j'}')^{-1}\in \widetilde{H}'$. By \eqref{norm} and \eqref{hjbound} we have $\|\widetilde{g}'\|\ll \|\widetilde{h}_{j'}'\|\ll\|h_{0,j'}\|\|h_{-,j'}\|\ll\eps^{-\frac{2}{5}}$. Hence, using \eqref{leftmetric}, $$\bd^G(\widetilde{g}'(h_{+,j'}^{-1}h_{+,j}),\widetilde{g}')\ll \|\widetilde{g}'\|^2\bd^G(h_{+,j'}^{-1}h_{+,j},\operatorname{id})\ll r\eps^{\frac{6}{5}}e^{-ds}.$$
    It follows that $h_+^{-1}(\widetilde{h}_j')^{-1}\in B^G(\operatorname{id}, C_5^{-1}r\eps e^{-ds})\widetilde{H}'\subseteq B^H(\operatorname{id}, r\eps e^{-ds})\widetilde{H}'$. Thus, we obtain $h_+\in B^H(\operatorname{id}, r\eps e^{-ds})$. However, this means that $\iota(a_{-s}x_j)\in \sB_{2C_5^2 \mathbf{r}'}\iota(a_{-s}x_{j'})$ and contradicts the fact that $\set{\iota(a_{-s}x_j)}_{j\in\cJ'}$ is a $\sB_{C_9\mathbf{r}'}$-separated set, as we set $C_9=6C_5^2$. Thus (2) is proved.

    We next prove (3). Let us choose a maximal $C_9^5r\eps e^{-ds}$-separated set $\set{u_1,\cdots,u_M}$ in $B^H(\operatorname{id},r\eps)$, where $M\asymp e^{dmns}$. Since $\iota(a_{-s}x_j)$ is contained in $\cF(C_9^4r\eps,C_4\eps^{\frac{1}{100d^2}})$, for any $1\leq k\leq M$ and $x\in\supp\widetilde{\psi}_j$ we have $\iota(u_kx)=u_k\iota(x)$ by \eqref{psitildesupport}, hence $\sigma(u_ky)=\sigma(y)$ for any $1\leq k\leq M$ and $y\in\supp\widetilde{\psi}_j\circ\pi$. It follows that
    $\sigma_*\big((u_k)_*\mu_{t-2s,j}\big)=\nu_{t-2s,j}$, hence combining with Lemma \ref{tubeconcentration} and \eqref{tubeshape} we obtain
\eqlabel{perturbationcon}{\sigma_*\big((u_k)_*\mu_{t-2s,j}\big)(\widetilde{p}_j+h_{+,j}\cT)\geq c_{7}\eta^6\gg (\eps^{d^2+3n+\frac{\kappa}{3}}\|\bm_0\|^{-d})^6\gg \eps^{30d^2}}
for any $1\leq k\leq M$, if $\eps$ is sufficiently small. 

On the other hand, we may express $(u_k)_*\mu_{t-2s,j}$ by
\eq{\begin{aligned}
    (u_k)_*\mu_{t-2s,j}(f)&=\widetilde{\alpha}_j^{-1}m_H(V)^{-1}\int_V \widetilde{\psi}_j(\pi(a_{t-2s}uy_0))f(u_ka_{t-2s}uy_0)dm_H(u)\\&=\widetilde{\alpha}_j^{-1}m_H(V)^{-1}\int_{u_k'V} \widetilde{\psi}_j^k(\pi(a_{t-2s}uy_0))f(a_{t-2s}uy_0)dm_H(u),
\end{aligned}}
where $\widetilde{\psi}_j^k(x):=\widetilde{\psi}_j(u_k^{-1}x)$ and $u_k'=a_{-(t-2s)}u_ka_{t-2s}$.  It follows that
$$(u_k)_*\mu_{t-2s,j}(f)= \widetilde{\alpha}_j^{-1}\int_Y \big(\widetilde{\psi}_j^k\circ\pi\big)(y)f(y)d\mu_{t-2s}(y)+O(\widetilde{\alpha}_j^{-1}\|f\|_\infty e^{-d(t-2s)})$$
for any $f\in L^\infty(Y)$, since $m_H(V\triangle u_k'V)\ll e^{-d(t-2s)}$. We also note that
$\supp \widetilde{\psi}_j^k=u_k\cdot\supp \widetilde{\psi}_j$ is contained in $$u_k\sB_{C_9^3\mathbf{r}'}=B^{H}(u_k,C_9^3r\eps e^{-ds})B^{H^0}(\operatorname{id},C_9^3r\eps)B^{H^-}(\operatorname{id},C_9^3r\eps),$$ hence the sets $\supp \widetilde{\psi}_j^k$ for $1\leq k\leq M$ are all disjoint, which implies that $\displaystyle\sum_{k=1}^{M} \widetilde{\psi}_j^k\leq 1_{X}$. It follows that for any $f\in L^\infty(Y)$
\eq{\begin{aligned}\displaystyle\sum_{k=1}^{M}&(u_k)_*\mu_{t-2s,j}(f)\\&=\widetilde{\alpha}_j^{-1}\int_Y \left(\displaystyle\sum_{k=1}^{M}\big(\widetilde{\psi}_j^k\circ\pi\big)(y)\right)f(y)d\mu_{t-2s}(y)+O(M\widetilde{\alpha}_j^{-1}\|f\|_\infty e^{-d(t-2s)})\\&\leq \widetilde{\alpha}_j^{-1}\int_Y f(y)d\mu_{t-2s}(y)+O(M\widetilde{\alpha}_j^{-1}\|f\|_\infty e^{-d(t-2s)})\\&= \widetilde{\alpha}_j^{-1}\mu_{t-2s}(f)+O(M\widetilde{\alpha}_j^{-1}\|f\|_\infty e^{-d(t-2s)}).\end{aligned}}
Combining with \eqref{perturbationcon} we get
\eq{\begin{aligned}\nu_{t-2s}(\widetilde{p}_j+h_{+,j}\cT)&=\sigma_*\mu_{t-2s}(\widetilde{p}_j+h_{+,j}\cT)\\&\geq \widetilde{\alpha}_j \displaystyle\sum_{j=1}^{M}\sigma_*\big((u_k)_*\mu_{t-2s,j}\big)(\widetilde{p}_j+h_{+,j}\cT)+O(Me^{-d(t-2s)})\\&\gg  \widetilde{\alpha}_j M \eps^{30d^2}+O(Me^{-d(t-2s)})\gg \widetilde{\alpha}_j M\eps^{30d^2}. \end{aligned}}
Recall that $\widetilde{\alpha}_j\asymp (r\eps)^{d^2-1}e^{-dmns}$ and $M\asymp e^{dmns}$. Thus, we conclude
$$\nu_{t-2s}(\widetilde{p}_j+h_{+,j}\cT)\gg (r\eps)^{d^2-1}\eps^{30d^2}\gg \eps^{40d^2}.$$
Taking $c_{9}>0$ sufficiently small, we get $ \nu_{t-2s}(\widetilde{p}_j+h_{+,j}\cT)\geq \eps^{50d^2}$ as desired.
    
\end{proof}

\begin{prop}\label{highcon}
Let $t>0$, $-\frac{1}{\kappa_4}\log c_{9}<s\leq\kappa_4 t$, $e^{-\kappa_{4}^2 s}\leq r<c_{9}^{\kappa_4}$, and $z$ and $\om$ be as in Proposition \ref{mainprop}. If $|\widehat{\nu_{t,\om}}(\bm_0)|\ge c_{3}e^{-\kappa_{4}^2 s}$ for some $0<\|\bm_0\|<e^{\kappa_4 s}$, then there exists $p\in\bT^d$ with
$$\nu_{t-2s}(B^{\bT^d}(p,e^{-s}))> c_{10}e^{-\frac{\kappa_3}{20d}s}$$
for some $c_{10}>0$.
\end{prop}
\begin{proof}
Recall that $\kappa_4$ was chosen by $\kappa_4=\frac{\min\set{\kappa_2,\kappa_3}}{2000d^3}$. Let $\eps=e^{-\kappa_4 s}$ so that $r\geq \eps^{\kappa_4}$. Since $\kappa_4\leq \frac{\kappa_3}{2000d^3}$, we have
$$r^{-(d^2-1)}\eps^{\frac{\kappa_3}{2}}\leq \eps^{\frac{\kappa_3}{2}-(d^2-1)\kappa_4}<\eps^{\kappa_4}=e^{-\kappa_4^2s},$$ hence the assumption $|\widehat{\nu_{t,\om}}(\bm_0)|\ge c_{3}r^{-(d^2-1)}\eps^{\frac{\kappa_3}{2}}$ of Lemma \ref{concentrationlemma} is satisfied. It follows that there exists $\widetilde{x}\in X$ with $\iota(\widetilde{x})\in\cF(C_9^4r\eps,C_4\eps^{\frac{1}{100d^2}})$ and $\cJ''\subseteq \cJ'$ with $|\cJ''|\geq c_{8}\eps^{5d^2}e^{dmns}$ such that the properties (1),(2),(3) in Lemma \ref{concentrationlemma} hold. We consider the cases $m<n$ and $m=n$ separately.

\textbf{Case 1: $m<n$.}\\
In this case we pick any element $j\in\cJ''$ and take $p=\widetilde{p}_j$. By (3) of Lemma \ref{concentrationlemma} we have $\nu_{t-2s}(p+h_{+,j}\cT)>\eps^{50d^2}$. Since $n-m\ge 1$, the tube $\cT$ is contained in $B^{\bR^d}(0,\eps^{d-1}e^{-s})$. It follows from \eqref{hjbound} that the tube $h_{+,j}\cT$ is contained in $B^{\bR^d}(0,e^{-s})$. Thus,
$$\nu_{t-2s}(B^{\bT^d}(p,e^{-s}))\geq \nu_{t-2s}(p+h_{+,j}\cT)>\eps^{50d^2}=e^{-50d^2\kappa_4 s}\geq e^{-\frac{\kappa_3}{20d}s}.$$

\textbf{Case 2: $m=n$.}\\
As the tube $\cT=B^{\bR^m}(0,\eps^{d-1}e^{-2ns})\times B^{\bR^n}(0,\eps^{d-1})$ is not contained in the ball $B^{\bR^d}(0,e^{-s})$, in this case it is not sufficient to use a single element $j\in\cJ''$. Indeed, we need to investigate the intersection of two tubes in the collection $\set{\widetilde{p}_j+h_{+,j}\cT}_{j\in\cJ''}$. Denote by $A_j\in \operatorname{Mat}_{n,n}(\bR)$ the matrix corresponding to $h_{+,j}\in H$, i.e. $h_{+,j}=\varphi(A_j)=\left(\begin{matrix} \Id_m & A_j\\
0 & \Id_n\end{matrix}\right)$. Note that $\|A_j\|\ll \eps^{-\frac{1}{5}}$ by \eqref{hjbound}. Here, for $A_j\in \operatorname{Mat}_{n,n}(\bR)$ the norm $\|A_j\|$ stands for just the maximum of the absolute values of the entries $A_j$, in contrast to the $\|g\|$ defined in \S \ref{subsec:Metricnorm} which also involves $g^{-1}$.
\begin{lem}\label{IntersectionLemma}
    For any $j,j'\in\cJ''$ with $\operatorname{det}(A_{j}-A_{j'})\neq 0$, the intersection of two tubes $\widetilde{p}_{j}+h_{+,j'}\cT$ and $\widetilde{p}_{j'}+h_{+,j'}\cT$ in $\bT^d$ is contained in a ball of radius $O(|\operatorname{det}(A_{j}-A_{j'})|^{-1}e^{-2ns})$.
\end{lem}
\begin{proof}
    Since $h_{+,j}\cT, h_{+,j'}\cT\subset \bR^d$ are projected onto $\bT^d$ injectively, we may consider $\widetilde{p}_{j}+h_{+,j'}\cT$ and $\widetilde{p}_{j'}+h_{+,j'}\cT$ as subsets of $\bR^d$. One can write $h_{+,j'}\cT$ and $h_{+,j'}\cT$ explicitly as follows:
    $$h_{+,j}\cT=\set{(v_+,v_-)\in \bR^m\times\bR^n: \|v_+-A_jv_-\|< \eps^{d-1}e^{-2ns}, \; \|v_-\|<\eps^{d-1}},$$
    $$h_{+,j'}\cT=\set{(v_+,v_-)\in \bR^m\times\bR^n: \|v_+-A_{j'}v_-\|< \eps^{d-1}e^{-2ns}, \; \|v_-\|<\eps^{d-1}}.$$
    Since $h_{+,j}\cT$ and $h_{+,j'}\cT$ are neighborhoods of $n(=m)$-dimensional subspaces, by translation, it suffices to consider the case $\widetilde{p}_{j}=\widetilde{p}_{j'}=0$.

    Suppose that $(v_+,v_-)\in h_{+,j'}\cT\cap h_{+,j'}\cT$. Then $(v_+,v_-)$ satisfies 
    $$\|(A_j-A_{j'})v_-\|\leq \|v_+-A_jv_-\|+\|v_+-A_{j'}v_-\|<2\eps^{d-1}e^{-2ns}.$$
    Note that $(A_j-A_{j'})^{-1}$ can be written $\operatorname{det}(A_j-A_{j'})^{-1}\operatorname{adj}(A_j-A_{j'})$, where $\operatorname{adj}(A_j-A_{j'})$ is the adjugate matrix of $A_j-A_{j'}$, whose norm is bounded above by $\ll \|A_j-A_{j'}\|^{d-1}\ll \eps^{-\frac{d-1}{5}}$. It follows that
    \eq{\begin{aligned}
        \|v_-\|&\ll \|(A_j-A_{j'})^{-1}\|\|(A_j-A_{j'})v_-\|\\&\ll |\operatorname{det}(A_j-A_{j'})|^{-1}\eps^{-\frac{d-1}{5}}(\eps^{d-1}e^{-2ns}),
    \end{aligned}}
    hence $\|v_-\|\ll |\operatorname{det}(A_j-A_{j'})|^{-1}\eps^{\frac{4(d-1)}{5}}e^{-2ns}$. It also implies that $$\|v_+\|\leq \|A_j\|\|v_-\|+\eps^{d-1}e^{-2ns}\ll |\operatorname{det}(A_j-A_{j'})|^{-1}e^{-2ns}.$$ It completes the proof.
\end{proof}

We also need the following counting lemma in order to control the size of $|\operatorname{det}(A_{j_1}-A_{j_2})|^{-1}$:
\begin{lem}\label{detcount}
    For any $r\eps^{-n}e^{-ds}<\tau <1$ and $j_0\in \cJ''$,
    $$\#\set{j\in\cJ'': |\operatorname{det}(A_{j}-A_{j_0})|<\tau}\ll \tau^{\frac{1}{n}}\eps^{-3d^2}e^{dmns}.$$
\end{lem}
\begin{proof}
    Let $\cZ\subset \operatorname{Mat}_{n,n}(\bR)$ be the zero set of the map $X\mapsto\operatorname{det}(X)$, and
    $$\cL_{j_0,\tau}:=\set{A\in \operatorname{Mat}_{n,n}(\bR): |\operatorname{det}(A-A_{j_0})|<2\tau,\; \|A\|\ll \eps^{-\frac{1}{5}}}.$$
    Since $X\mapsto \operatorname{det}(X)$ is a homogeneous polynomial of degree $n$, the set $\cL_{j_0,\tau}$ is contained in $O(\tau^\frac{1}{n})$-neighborhood of the hypersurface $A_{j_0}+\cZ$. By \cite[Theorem 1.1]{Lot15} the volume of the $O(\tau^\frac{1}{n})$-neighborhood of the hypersurface $A_{j_0}+\cZ$ is $\ll \tau^{\frac{1}{n}}\eps^{-\frac{1}{5}}$. 

    We next claim that for $j\in\cJ''$ with $|\operatorname{det}(A_j-A_{j_0})|<\tau$ the ball of radius $\frac{1}{2}r\eps e^{-ds}$ centered at $A_j$ is contained in $\cL_{j_0,\tau}$. Indeed, for every $A$ contained in the ball we have $\|A\|\ll \eps^{-\frac{1}{5}}$, hence all the first derivative of the polynomial $X\mapsto \operatorname{det}(X-A_{j_0})$ are $\ll \eps^{-\frac{n-1}{5}}$. It follows that
    $$|\operatorname{det}(A-A_{j_0})|<|\operatorname{det}(A_j-A_{j_0})|+O(\eps^{-\frac{n-1}{5}}\cdot r\eps e^{-ds})<2\tau,$$
    thus the claim is proved.
    
    Since $A_j$'s are $r\eps^2 e^{-ds}$-separated in $\operatorname{Mat}_{n,n}(\bR)$ by (2) of Lemma \ref{concentrationlemma}, the cardinality of the set $\set{j\in\cJ'': |\operatorname{det}(A_{j}-A_{j_0})|<\tau}$ is bounded by
    $$ \ll \tau^{\frac{1}{n}}\eps^{-\frac{1}{5}}(r\eps^2e^{-ds})^{-mn}\leq \tau^{\frac{1}{n}}(\eps^3e^{-ds})^{-mn}=\tau^{\frac{1}{n}}\eps^{-3d^2}e^{dmns}.$$
\end{proof}

Recall that $|\cJ''|\geq c_{8}\eps^{5d^2}e^{dmns}$. Applying Lemma \ref{detcount} with $\tau=\eps^{200d^3}$, one can find $\set{j_1,\cdots,j_M}\subset \cJ''$ with $M\asymp \eps^{-100d^2}$ such that $$|\operatorname{det}(A_{j_k}-A_{j_l})|\geq \tau=\eps^{200d^3}$$ for any $1\leq k<l\leq M$, since
$$\frac{|\cJ''|}{\tau^{\frac{1}{d}}\eps^{-3d^2}e^{dmns}}\gg \frac{\eps^{5d^2}e^{dmns}}{\eps^{200d^2}\eps^{-3d^2}e^{dmns}}\geq \eps^{-100d^2}.$$
Denote by $\varrho_k$ the characteristic function of the set $\widetilde{p}_{j_k}+h_{+,j_k}\cT$ for $1\leq k\leq M$ and let $\varrho:=\displaystyle\sum_{k=1}^M \varrho_k$. Then by (3) of Lemma \ref{concentrationlemma} we have
\eqlabel{Qestimate0}{\int \varrho d\nu_{t-2s}=\sum_{k=1}^M \nu_{t-2s}(\widetilde{p}_{j_k}+h_{+,j_k}\cT) \gg M\eps^{50d^2}\gg \eps^{-50d^2}.}
Let $Q:=\int \varrho d\nu_{t-2s}$. By Cauchy-Schwarz inequality, we get
\eqlabel{Qestimate1}{\int \varrho(b)^2 d\nu_{t-2s}(b)\geq \left(\int\varrho d\nu_{t-2s}\right)^2= Q^2.}
Since $\varrho_k$'s are characteristic functions, we also have
\eqlabel{Qestimate2}{\displaystyle\sum_{k=1}^M \int \varrho_k(b)^2d\nu_{t-2s}(b)=\displaystyle\sum_{k=1}^M \int \varrho_k(b)d\nu_{t-2s}(b)=Q.}
Combining \eqref{Qestimate0}, \eqref{Qestimate1}, and \eqref{Qestimate2},
\eq{\begin{aligned}
    \displaystyle\sum_{k\neq l}\int \varrho_k\varrho_l d\nu_{t-2s}&=\int \varrho(b)^2 d\nu_{t-2s}(b)-\displaystyle\sum_{k=1}^M \int \varrho_k(b)^2d\nu_{t-2s}(b)\\&\geq Q^2-Q\gg \eps^{-100d^2}.
\end{aligned}}
The number of the pairs in $\set{1,\cdots,M}^2$ is $M^2\ll \eps^{-200d^2}$, hence we can find a pair $(k,l)$ with $k\neq l$ such that
$\int \varrho_k\varrho_l d\nu_{t-2s}\gg \eps^{100d^2}$, i.e.
$$ \nu_{t-2s}\left( (\widetilde{p}_{j_k}+h_{+,j_k}\cT)\cap (\widetilde{p}_{j_l}+h_{+,j_l}\cT) \right)\gg \eps^{100d^2}.$$
By Lemma \ref{IntersectionLemma}, the intersection is contained in a ball of radius $\ll \tau^{-1}e^{-2ns}=\eps^{-200d^3}e^{-2ns}\leq e^{-s}$. Thus, there exists a point $p\in\bT^d$ such that
$$\nu_{t-2s}\big(B^{\bT^d}(p,e^{-s})\big)\gg \eps^{100d^2}=e^{-100d^2\kappa_4 s}\geq e^{-\frac{\kappa_3}{20d}s}.$$

\end{proof}

\subsection{Low concentration of $\nu_{t-2s}$ from a Diophantine condition}
The goal of this subsection is to show that if $b_0\in\bT^d$ is not approximated by a rational number with a small denominator, then the concentration of any ball in $\bT^d$ is low with respect to $\nu_{t-2s}=\nu_{y_0,t-2s}$, where $y_0=g_0w(b_0)\hat{\Gamma}$.
Recall the definition
$$\zeta(b,T):=\min\set{N\in\bN: \displaystyle\min_{1\leq q\leq N}\|qb\|_\bZ\leq \frac{N^2}{T}}$$
for $b\in\bR^d$ and $T>0$.

\begin{lem}\label{count}
Let $\Xi$ be an interval in $\bT$ of length $2\rho$. For $b\in\bT^d$, if $\zeta(b,R)\ge \rho^{-2d}$, then
$$\left|\set{\bm\in B^{\bZ^d}(0,R): b\cdot\bm\in \Xi}\right|\leq C_{21}\rho|B^{\bZ^d}(0,R)|$$
for some $C_{21}>1$.
\end{lem}
\begin{proof}
Let $b=(b_1,\cdots,b_d)$ and write
$$\zeta(b_i,R):=\min\set{N\in\bN: \displaystyle\min_{1\leq q\leq N}\|qb_i\|_\bZ\leq \frac{N^2}{R}}$$
for each $1\leq i\leq d$.
We first claim $\zeta(b,R)\leq \displaystyle\prod_{j=1}^d\zeta(b_j,R)$. By definition, for each $i$ there exists $1\leq q_i\leq \zeta(b_i,R)$ such that $\|q_ib_i\|_\bZ\leq\frac{\zeta(b_i,R)^2}{R}$. It follows that for any $1\leq i\leq d$, we have
$$\|q_1\cdots q_db_i\|_\bZ\leq \|q_ib_i\|_\bZ\displaystyle\prod_{j\neq i}q_j\leq\frac{\zeta(b_i,R)^2}{R}\displaystyle\prod_{j\neq i}\zeta(b_j,R)\leq\frac{1}{R}\displaystyle\prod_{j=1}^d\zeta(b_j,R)^2,$$
hence $\|q_1\cdots q_db\|_\bZ\leq \frac{1}{R}\displaystyle\prod_{j=1}^d\zeta(b_j,R)^2$. Since $1\leq q_1\cdots q_d\leq \displaystyle\prod_{j=1}^d\zeta(b_j,R)$, the claim follows.

By the claim and the assumption $\zeta(b,R)\geq \rho^{-2d}$, there exists $1\leq i\leq d$ such that $\zeta(b_i,R)\geq \rho^{-2}$. For the convenience of notations, assume that such $i$ is $1$ without loss of generality. Write $b=(b_1,b')\in\bT\times\bT^{d-1}$ and $\bm=(k,\bm')\in\bZ\times\bZ^{d-1}$. The condition $b\cdot\bm\in\Xi$ is equivalent to $kb_1\in\Xi'$, where $\Xi'=\Xi-b'\cdot\bm'$. Thus, it suffices to prove the following statement: for any interval $\Xi'\subset\bT$ of length $2\rho$ and $\zeta(b_1,R)\ge\rho^{-2}$,
\eqlabel{circrot}{|k\in B^{\bZ}(0,R): kb_1\in\Xi'|\ll\rho R.}

One can modify Weyl's famous equidistribution criterion to be an effective version, and obtain the following upper bound of the density of the irrational rotational orbit on $\bT$ (See Lemma \ref{Weyl''} for a detailed computation):
\eqlabel{Weyl}{\frac{1}{2R}\left|\left\{k\in B^{\bZ}(0,R): kb_1\in\Xi'\right\}\right|\ll |\Xi'|+|\Xi'|^{-1}\zeta(b_1,R)^{-1}.}
Under the assumptions $|\Xi'|=|\Xi|=2\rho$ and $\zeta(b_1,R)\ge\rho^{-2}$ it yields \eqref{circrot}, so concludes the proof.
\end{proof}

In the rest of this subsection, the maps $\xi=\xi_{t-2s}:X\times V\to\cF$ and $\gamma=\gamma_{t-2s}:X\times V\to\Gamma$ denote the same maps as in Section 3, but for the parameter $t-2s$ instead of $s$, i.e.
$$a_{t-2s}u\iota(x)=\xi(x,u)\gamma(x,u).$$
The following proposition relates the concentration of the measure $\nu_{t-2s}=\nu_{y_0,t-2s}$ and the Diophantine condition of $y_0$.
\begin{prop}\label{lowcon}
Let $0<s<\frac{t}{3}$ and $\rho\ge e^{-\kappa_4 t}$. For $g_0\in G$ and $b_0\in\bT^d$, let $y_0=g_0w(b_0)\hat{\Gamma}\in Y$. Let $x_0=\pi(y_0)\in X$ and assume $\textrm{ht}(x_0)\leq \rho^{-\frac{1}{10d^2(d-1)}}$. For some constants $c_{11},C_{22}>0$, if $\zeta(b_0,\|g_0\|^{-1}e^{n(t-2s)})\ge c_{11}\rho^{-2d}$, then
$$\nu_{t-2s}(B^{\bT^d}(p,\rho))\leq C_{22}\rho^{\frac{\kappa_3}{10d}}$$
for any $p\in\bT^d$.
\end{prop}

\begin{proof}
Let $\gamma_0$ be the element of $\Gamma$ such that $g_0=\iota(x_0)\gamma_0$. Then $y_0$ can be expressed by $y_0=g_0w(b_0)\hat{\Gamma}=\iota(x_0)w(\gamma_0b_0)\hat{\Gamma}$. Let $b=\gamma_0b_0$, then $\sigma(y_0)=\gamma_0b_0=b$. Denote by $\Theta:\bT^d\to\bT$ the projection onto the first coordinate. Let $\Xi=\Theta(B^{\bT^d}(p,\rho))$, then $\Xi\subseteq\bT$ is an interval of length $2\rho$. By the definitions of $\nu_{t-2s}$ and $\mu_{t-2s}$, we have
\eqlabel{xiset}{\nu_{t-2s}(\Theta^{-1}(\Xi))=\frac{1}{m_H(V)}m_H(\set{u\in V: \sigma(a_{t-2s}uy_0)\in\Theta^{-1}(\Xi)}).}
By Lemma \ref{siggam} with $\xi=\xi_{t-2s}$ and 
$\gamma=\gamma_{t-2s}$, $$\sigma(a_{t-2s}uy_0)=\gamma(x_0,u)\sigma(y_0)=\gamma(x_0,u)b$$
holds. Hence, the set in \eqref{xiset} can be expressed as
\eqlabel{xiset2}{\begin{aligned}
    \set{u\in V: \sigma(a_{t-2s}uy_0)\in\Theta^{-1}(\Xi)}
    &=\set{u\in V: e_1\cdot\sigma(a_{t-2s}uy_0)\in\Xi}\\
    &=\set{u\in V: \gamma(x_0,u)^{tr}e_1\cdot b\in\Xi}.
\end{aligned}}
Let $\eps=\rho^{\frac{1}{5d}}$ and $V_{x_0,\eps}$ be the set as in Lemma \ref{Vxe}. Note that the conditions $\eps=\rho^{\frac{1}{5d}}>e^{-\kappa_2(t-2s)}$ and $\textrm{ht}(x_0)<e^{\kappa_2(t-2s)}$ of Lemma \ref{Vxe} and Proposition \ref{key} for $\xi_{t-2s}$ and $\gamma_{t-2s}$ are satisfied since $\kappa_4\leq\frac{\kappa_2}{2000d^3}$. Then 
\eqlabel{estVxe}{m_H(V\setminus V_{x_0,\eps})\leq C_{13}\eps m_H(V)} by Lemma \ref{Vxe}. We now apply Proposition \ref{key} for $t-2s$ and $\bm_0=e_1$. Let $R=C_{15}\eps^{-1}\textrm{ht}(x_0)^{d-1}e^{n(t-2s)}$ as in Proposition \ref{key}. Denoting $\Omega:=\set{\bm\in B^{\bZ^d}(0,R): b\cdot\bm\in\Xi}$ and applying Proposition \ref{key}, we have
\eqlabel{xiset3}{\set{u\in V: \gamma(x_0,u)^{tr}e_1\cdot b\in\Xi}\subseteq (V\setminus V_{x,\eps})\cup\displaystyle\bigcup_{\bm\in\Omega}\set{u\in V_{x_0,\eps}:\gamma(x_0,u)e_1=\bm},}
\eqlabel{xiest}{m_H(\set{u\in V_{x_0,\eps}:\gamma(x_0,u)e_1=\bm})\leq C_{16}\eps^{-3n}e^{-dn(t-2s)}}
for any $\bm\in\bZ^d$. We also have
\eqlabel{Omega}{|\Omega|\leq C_{21}\rho|B^{\bZ^d}(0,R)|\ll \rho\eps^{-d}\textrm{ht}(x_0)^{d(d-1)}e^{dn(t-2s)}}
by Lemma \ref{count}. Let $c_{11}:=\max(C_1^{\frac{1}{2}}C_6^{\frac{1}{2}}C_{15}^{-\frac{1}{2}},1)$. Using \eqref{rescaling},\eqref{gammab},\eqref{norm}, and \eqref{opnorm}, the assumption $\zeta(b_0,\|g_0\|^{-1}e^{n(t-2s)})\ge c_{11}\rho^{-2d}$ implies
\eqlabel{zetacond}{\begin{aligned}
    \zeta(b,R)&\ge\zeta(b_0,\|\gamma_0\|^{-1}R)\ge\zeta(b_0,C_1^{-1}\|g_0\|^{-1}\|\iota(x_0)\|^{-1}R)\\
    &\ge\zeta(b_0, C_1^{-1}C_6^{-1}\|g_0\|^{-1}\textrm{ht}(x_0)^{-(d-1)}R)\\
    &\ge \min(C_1^{-\frac{1}{2}}C_6^{-\frac{1}{2}}C_{15}^{\frac{1}{2}},1)\zeta(b_0, \|g_0\|^{-1}e^{n(t-2s)})\ge \rho^{-2d},
\end{aligned}}
 so the condition of Lemma \ref{count} is satisfied.
Combining \eqref{xiset}-\eqref{Omega},
\eqlabel{xiset4}{
\begin{aligned}
    \nu_{t-2s}(\Theta^{-1}(\Xi))&=\frac{1}{m_H(V)}m_H(\set{u\in V: \gamma(x_0,u)^{tr}e_1\cdot b\in\Xi})\\
    &\leq\frac{m_H(V\setminus V_{x,\eps})}{m_H(V)}+\displaystyle\sum_{\bm\in\Omega}\frac{m_H(\set{u\in V_{x_0,\eps}:\gamma(x_0,u)e_1=\bm})}{m_H(V)}\\
    &\leq C_{13}\eps+C_{16}|\Omega|\eps^{-3n}e^{-dn(t-2s)}\\
    &\ll \eps+\rho\eps^{-(d+3n)}\textrm{ht}(x_0)^{d(d-1)}\ll \rho^{\frac{\kappa_3}{10d}}.
\end{aligned}}
It completes the proof since $B^{\bT^d}(p,\rho)\subseteq \Theta^{-1}(\Xi)$.
\end{proof}

\subsection{Proof of Theorem \ref{mainthm}}
Combining Proposition \ref{highcon} and Proposition \ref{lowcon}, we deduce the following Fourier decay estimate.

\begin{prop}\label{Fourdecay}
Let $t>0$, $\|g_0\|\leq e^{\frac{nt}{4}}$, $b\in\bT^d$, and $y_0=g_0w(b_0)\hat{\Gamma}$. Let $\rho=\min\left(e^{-\kappa_{4} t},c_{11}^{-\frac{1}{2d}}\zeta(b_0,e^{\frac{nt}{2}})^{-\frac{1}{2d}}\right)$ and $r=\rho^{\kappa_4^2}$. Assume $\rho>c_{12}$ and $\textrm{ht}(\pi(y_0))\leq\rho^{-\frac{1}{10d^2(d-1)}}$, where $c_{12}:=c_{10}^{10}C_{22}^{-10}$. Let $z$ and $\om$ be as in Proposition \ref{mainprop}. Then
$$|\widehat{\nu_{t,\om}}(\bm_0)|<c_{3}\rho^{\kappa_4^2}$$
for any $0<\|\bm_0\|<\rho^{-\kappa_4}$.
\end{prop}
\begin{proof}
This is a direct consequence of Proposition \ref{highcon} and \ref{lowcon}, so it suffices to check the conditions in the propositions.
Set $s=-\log\rho$ so that $\rho=e^{-s}$, then $s\leq \kappa_4 t$. Since $\|g_0\|\leq e^{\frac{nt}{4}}$, $\|g_0\|^{-1}e^{n(t-2s)}\ge e^{\frac{nt}{2}}$, hence the condition $\zeta(b_0,\|g_0\|^{-1}e^{n(t-2s)})\ge c_{11}\rho^{-2d}$ is satisfied. It follows from Proposition \ref{lowcon} that
\eqlabel{high}{\nu_{t-2s}(B^{\bT^d}(p,\rho))\leq C_{22}\rho^{\frac{\kappa_3}{10d}}}
for any $p\in\bT^d$. Suppose that there exists some $0<\|\bm_0\|<\rho^{-\kappa_4}=e^{\kappa_4s}$ with $|\widehat{\nu_{t,\om}}(\bm_0)|\ge c_{3}\rho^{\kappa_4^2}=c_{3}e^{-\kappa_4^2 s}$ for contradiction. For sufficiently small $\rho$, it follows from Proposition \ref{highcon} that
\eqlabel{low}{\nu_{t-2s}(B^{\bT^d}(p,\rho))> c_{10}\rho^{\frac{\kappa_3}{20d}}\ge C_{22}\rho^{\frac{\kappa_3}{10d}},}
which contradicts \eqref{high}.
\end{proof}

We now prove Theorem \ref{mainthm}.

\begin{proof}[Proof of Theorem \ref{mainthm}]
We fix $y_0=g_0w(b_0)\hat{\Gamma}\in Y$ with $b_0\in\bT^d\setminus\bQ^d$ and denote $\mu_t=\mu_{y_0,t}$ the same as we did. To prove Theorem \ref{mainthm}, it suffices to find a constant $\del'>0$ such that if $\|g_0\|\leq\zeta(b_0,e^{\frac{nt}{2}})^{-\del'}$, then
$$\mu_t(f)=m_Y(f)+O(\cS(f)\zeta(b_0,e^{\frac{nt}{2}})^{-\del'})$$
for any $t\ge 0$, $f\in C^\infty_c(Y)$.
Recall the notations that $l$ is the degree of the Sobolev norm $\cS$, $\del_0$ is the constant of Theorem \ref{equX}, and $\kappa_1,\cdots,\kappa_4$ are the constants we have chosen throughout the present paper. Given $t>0$, let $$\rho=\min\left(e^{-\kappa_{4} t},c_{11}^{-\frac{1}{2d}}\zeta(b_0,e^{\frac{nt}{2}})^{-\frac{1}{2d}}\right)$$ as in Proposition \ref{Fourdecay}, $\kappa_{5}:=\frac{\kappa_4^2}{3d}$ and $\eps=\rho^{\frac{\kappa_5}{2l(d-1)}}$. We can also take $\del_1>0$ to satisfy $ \zeta(b_0,e^{\frac{nt}{2}})^{\del_1}\leq\min(\rho^{-\frac{1}{10d^2(d-1)}},e^{\frac{\del_0}{2\kappa_1}t})$. The desired constant $\del'$ will be taken to be $0<\del'<\del_1$ so that $\|g_0\|\leq\zeta(b_0,e^{\frac{nt}{2}})^{\del'}$ implies $\textrm{ht}(x_0)\leq \rho^{-\frac{1}{10d^2(d-1)}}$ and $\|g_0\|\leq e^{\frac{nt}{4}}$, which are the conditions in Proposition \ref{Fourdecay}.

We fix a partition of unity $\set{\om_i}_{i\in\cI}$ using Proposition \ref{partition} with a radius $r=\rho^{\kappa_4}$, where $\cI=\set{1,\cdots,N_r}\cup\set{\infty}$. Let $\set{z_1,\cdots,z_{N_r}}\subset K(C_{10}r^{\frac{1}{d}})$ be the corresponding set as in Proposition \ref{partition}. Let $\cI_{int}\subset\cI$ be the set of $i\in\cI$ such that $\iota(z_i)\in\cF(C_9^4r,\eps)$. Since $r^{\kappa_3}\leq \rho^{\kappa_3\kappa_4}\leq \rho^{\kappa_5}\leq\eps^d$, $m_G(\cF(C_9^5r,\frac{\eps}{2}))\ll \eps^d$ by \eqref{Fest}. For any $i\in\cI\setminus\cI_{int}$, each $\iota(\sB_r(z_i))$ is in $\cF\setminus\cF(C_9^5r,\frac{\eps}{2})$ and has $m_G$-measure $\asymp r^{d^2-1}$. Thus, $|\cI\setminus\cI_{int}|\ll \eps^dr^{-(d^2-1)}$ since the $\iota(\sB_r(z_i))$'s are disjoint.

For any $i\in\cI_{int}$, $z_i$ and $\om_i$ satisfy the conditions $\iota(z_i)\in\cF(C_9^4r,C_{10}r^{\frac{1}{d}})$, $\mathds{1}_{\sB_r(z_i)}\leq\om_i\leq\mathds{1}_{\sB_{C_9^3r}(z_i)}$, and $\|\triangledown\om_i\|_{L^\infty(X)}\leq C_{11}r^{-1}$, which we have assumed in the previous propositions. Write $\mu_{t,i}:=\mu_{t,\om_i}$ and $\nu_{t,i}:=\nu_{t,\om_i}$ for simplicity, where $\mu_{t,\om_i}$ and $\nu_{t,\om_i}$ is defined as \eqref{omdef} and \eqref{nudef}. 
 
 Recall that we defined $\overline{f}\in C^\infty_c(X)$ by $\overline{f}(x)=\int_{\pi^{-1}(x)}f(y)dm_{\pi^{-1}(x)}(y)$ for $x\in X$, $f\in C^\infty_c(Y)$. Let us define  $h\in C^\infty_{c}(Y)$ by $h(y)=f(y)-\overline{f}(\pi(y))$ for $y\in Y$. Note that $$\cS(h)\ll\cS(f),\quad\|h\|_{L^{\infty}(Y)}\ll\|f\|_{L^{\infty}(Y)},\quad|h(gy)-h(y)|\ll r\cS(f)$$
for $g\in B^G(\operatorname{id},r)$ by definition. We also have 
\eqlabel{intzero}{\int_{\pi^{-1}(x)}h(y)dm_{\pi^{-1}(x)}(y)=0}
for any $x\in X$. We first decompose $\mu_t(f)$ as
\eqlabel{decom1}{\mu_t(f)=\mu_t(\overline{f}\circ\pi)+\mu_t(h)=\pi_*\mu_t(\overline{f})+\mu_t(h).}
By \eqref{equX'} and $\textrm{ht}(x_0)\ll\|g_0\|\leq\zeta(b_0,e^{\frac{nt}{2}})^{\del'}\leq e^{\frac{\del_0}{2\kappa_1}t}$, we have
\eqlabel{equX''}{
\begin{aligned}
    \pi_*\mu_t(\overline{f})&=m_X(\overline{f})+O(\textrm{ht}(x_0)^{\kappa_1}\cS^X(\overline{f})e^{-\del_0t})\\
    &=m_Y(f)+O(\textrm{ht}(x_0)^{\kappa_1}\cS(f)e^{-\del_0t})\\
    &=m_Y(f)+O(\cS(f)e^{-\frac{\del_0t}{2}}),
\end{aligned}}
Split the second term of \eqref{decom1} by 
\eqlabel{decom2}{
\begin{aligned}
    \mu_t(h)&=\displaystyle\sum_{i\in\cI}\pi_*\mu_t(\om_i)\mu_{t,i}(h)\\
    &=\displaystyle\sum_{i\in\cI\setminus\cI_{int}}\pi_*\mu_t(\om_i)\mu_{t,i}(h)+\displaystyle\sum_{i\in\cI_{int}}\pi_*\mu_t(\om_i)\mu_{t,i}(h).
\end{aligned}}
By \eqref{ommsr} and a trivial bound $\mu_{t,i}(h)\leq\|h\|_{L^{\infty}(Y)}\ll\cS(f)$,
\eqlabel{bdd1}{\displaystyle\sum_{i\in\cI\setminus\cI_{int}}\pi_*\mu_t(\om_i)\mu_{t,i}(h)\ll |\cI\setminus\cI_{int}|\cS(f)r^{(d^2-1)}\ll \cS(f)\eps^d=\cS(f)\rho^{\frac{d\kappa_5}{2l(d-1)}}.}

For each $i\in\cI_{int}$, denote by $\vartheta_i:\bT^d\to \pi^{-1}(z_i)$ the bijective Lipschitz function defined by $\vartheta_i(b):=\iota(z_i)w(b)\hat{\Gamma}$. Note that $\vartheta_i\circ\sigma=\operatorname{id}_{\pi^{-1}(z_i)}$ and the Lipschitz constant of $\vartheta_i$ is bounded above by $\|\iota(z_i)\|\leq C_6\textrm{ht}(z_i)^{d-1}$ and below by $\|\iota(z_i)\|^{-1}\ge C_6^{-1}\textrm{ht}(z_i)^{-(d-1)}$. Let $h_i:=h\circ\vartheta_i\in C^\infty_c(\bT^d)$. 

For any $i\in\cI_{int}$, $\supp \mu_{t,i}\subseteq \pi^{-1}\big(\sB_{C_9^3r}(z_i)\big)$ and $\iota\big(\sB_{C_9^3r}(z_i)\big)\subseteq \cF$. It follows that $\bd^G(g,\iota(z_i))<C_9^3r$ for any $g\in (\supp \om_i\circ\phi)$, so
\eqlabel{distance}{|h(gw(b)\hat{\Gamma})-h(\iota(z_i)w(b)\hat{\Gamma})|\ll r\cS(f),}
i.e. $|h(y)-h_i(\sigma(y))|\ll r\cS(f)$ for any $y\in\supp\mu_{t,i}$. Hence,
\eqlabel{happrox}{\mu_{t,i}(h)=\mu_{t,i}(h_i\circ\sigma)+O(r\cS(f))=\nu_{t,i}(h_i)+O(r\cS(f)).}

For each $h_i\in C^\infty_c(\bT^d)$ and $b\in\bT^d$, we consider the Fourier expansion
\eqlabel{Fourierex}{h_i(b)=\displaystyle\sum_{\bm\in\bZ^d}\widehat{h_i}(\bm)e^{-2\pi i\bm\cdot b}.}
By \eqref{intzero}, the Fourier coefficient at zero vanishes as
\eqlabel{Fzero}{\widehat{h_i}(0)=\int_{\bT^d}h(\iota(z_i)w(b)\hat{\Gamma})dm_{\bT^d}(b)=0.}

The smoothness of $h$ provides a decay of nonzero Fourier coefficients. Let $\cS^{\bT^d}$ be a $l$-th degree Sobolev norm on $\bT^d$ as $\cS^Y$. Then $\cS^{\bT^d}(h_i)$ is bounded above by $\cS^{\bT^d}(h\circ\vartheta_i)\ll \textrm{ht}(z_i)^{l(d-1)}\cS(h)$. Hence, for any $\bm\in\bZ^d\setminus\set{0}$,
\eqlabel{Fnonzero}{|\widehat{h_i}(\bm)|\leq \cS^{\bT^d}(h_i)\|\bm\|^{-(d+2)}\ll \textrm{ht}(z_i)^{l(d-1)}\cS(f)\|\bm\|^{-(d+2)}}
since $\cS$ is a $l$-th degree Sobolev norm and it controls the $L^\infty$ norm of $(d+2)$-th derivatives. 

On the other hand, we obtained an upper bound of $|\widehat{\nu_{t,i}}(\bm)|$ in Proposition \ref{Fourdecay}. Since $\kappa_{5}=\frac{\kappa_4^2}{3d}$, $\rho^{\kappa_4^2}=\rho^{3d\kappa_5}$ and $\rho^{-2\kappa_5}\leq \rho^{-\kappa_4}$. By Proposition \ref{Fourdecay},
\eqlabel{Fourdecay'}{|\widehat{\nu_{t,i}}(\bm)|<c_{3}\rho^{3d\kappa_5}}
for any $0<\|\bm\|<\rho^{-2\kappa_5}$. Expand $\nu_{t,i}(h)$ as
\eqlabel{decomp3}{
\begin{aligned}
    |\nu_{t,i}(h)|&=\left|\displaystyle\sum_{\bm\in\bZ^d\setminus\set{0}}\widehat{h_i}(\bm)\widehat{\nu_{t,i}}(\bm)\right|\\
    &\leq\displaystyle\sum_{0<\|\bm\|<\rho^{-2\kappa_5}}|\widehat{h_i}(\bm)||\widehat{\nu_{t,i}}(\bm)|+\displaystyle\sum_{\|\bm\|\ge\rho^{-2\kappa_5}}|\widehat{h_i}(\bm)||\widehat{\nu_{t,i}}(\bm)|.
\end{aligned}}

Then the first term in the last line is $\ll (\rho^{-2\kappa_5})^d\cS(f)\rho^{3d\kappa_5}\ll \rho^{d\kappa_5}\cS(f)$, using \eqref{Fourdecay'} and a trivial bound $|\widehat{h_i}(\bm)|\leq\|h\|_{L^{\infty}(Y)}\ll\cS(f)$. Using \eqref{Fnonzero} and a trivial bound $|\widehat{\nu_{t,k}}(\bm)|\leq 1$, the second term in the last line is bounded by $$\textrm{ht}(z_i)^{l(d-1)}\cS(f)\displaystyle\sum_{\|\bm\|\ge\rho^{-2\kappa_5}}\|\bm\|^{-(d+2)}\ll\textrm{ht}(z_i)^{l(d-1)}\cS(f)\rho^{2\kappa_5}.$$ Hence, we have
\eqlabel{bound1}{|\nu_{t,i}(h)|\ll \textrm{ht}(z_i)^{l(d-1)}\cS(f)\rho^{2\kappa_5}\leq \cS(f)\rho^{\kappa_5}.}
since $\textrm{ht}(z_i)\leq\eps^{-1}=\rho^{-\frac{\kappa_5}{2l(d-1)}}$. Taking summation over $i\in\cI_{int}$ and using \eqref{happrox} and \eqref{bound1},
\eqlabel{bdd2}{
\begin{aligned}
\left|\displaystyle\sum_{i\in\cI_{int}}\pi_*\mu_t(\om_i)\mu_{t,i}(h)\right|&= \left|\displaystyle\sum_{i\in\cI_{int}}\pi_*\mu_t(\om_i)\nu_{t,i}(h_i)\right|+O(r\cS(f))\\
&\leq\displaystyle\sum_{i\in\cI_{int}}\pi_*\mu_t(\om_i)\cS(f)\rho^{\kappa_5}+O(r\cS(f))\\
&=O((\rho^{\kappa_5}+\rho^{\kappa_4})\cS(f))=O(\rho^{\kappa_5}\cS(f)).
\end{aligned}
}

Combining \eqref{decom1}-\eqref{bdd1} and \eqref{bdd2}, we obtain the estimate
$$\mu_t(f)=m_Y(f)+O(\cS(f)(e^{-\frac{\del_0t}{2}}+\rho^{\frac{d\kappa_5}{2l(d-1)}})).$$
We note that $\zeta(b_0,e^{\frac{nt}{2}})\leq (e^{\frac{nt}{2}})^{\frac{d}{3d+1}}$ by \eqref{Dirichlet} and $\zeta(b_0,e^{\frac{nt}{2}})\ll \rho^2$. Hence we can find $\del'>0$ such that $e^{-\frac{\del_0t}{2}}+\rho^{\frac{d\kappa_5}{2l(d-1)}}\ll\zeta(b_0,e^{\frac{nt}{2}})^{-\del'}$ and $\del'<\del_1$. This completes the proof of the main theorem.
\end{proof}

\section{General diagonal subgroups}\label{gendiag}

In this section, we extend Theorem \ref{mainthm} for general $1$-dimensional diagonal subgroups and give a sketch of the modification. Following the notations of \cite{KW08} and \cite{KM12}, let us denote by $\mathfrak{a}^+$ the set of $d$-tuples $\bt=(t_1,\cdots,t_d)\in\bR^d$ such that
\eq{t_1\ge\cdots\ge t_m>0,\quad t_{d}\ge\cdots\ge t_{m+1}>0, \quad\textrm{and}\quad \displaystyle\sum_{i=1}^m t_i=\displaystyle\sum_{j=1}^n t_{m+j},}
and for $\bt\in\mathfrak{a}^+$ define
\eqlabel{btdef}{a_{\bt}:=\diag{e^{t_1},\cdots,e^{t_m},e^{-t_{m+1}},\cdots,e^{-t_d}}\in G}
and
\eq{\lfloor\bt\rfloor:=\displaystyle\min_{1\leq i\leq d}t_i.}
In general the horospherical subgroup for $a_\bt$ may be bigger than $$H=\set{\left(\begin{matrix} \Id_m& A\\
0&  \Id_n
\end{matrix}\right)\in G: A\in \operatorname{Mat}_{m,n}(\bR)},$$ but we still have an effective equidistribution result for $a_\bt$-translates of $H$-orbits. The following theorem is a reformulation of \cite[Theorem 1.3]{KM12} which is a generalization of Theorem \ref{equX}.
\begin{thm}\label{equXgen}
Let $V\subset H$ be a fixed neighborhood of the identity in $H$ with smooth boundary and compact closure. Then there exist constants $\del_0,\kappa_1>0$ only depending on $d$ so that
\eqlabel{equXgen'}{\frac{1}{m_H(V)}\int_V f(a_\bt ux)dm_H(u)=\int_X dm_X+O\left(\textrm{ht}(x)^{\kappa_1}\cS^X(f)e^{-\del_0\lfloor\bt\rfloor}\right)}
for any $\bt\in\mathfrak{a}^+$, $x\in X$ and $f\in C_c^\infty(X)$. Here, the implied constant depends only on $d$ and $V$.
\end{thm}

The following is a generalization of the main result of this paper, Theorem \ref{mainthm}, for general diagonal elements.

\begin{thm}\label{mainthm'}
Let $V$ be as in Theorem \ref{equXgen}. Then there exists a constant $\del'>0$ only depending on $m$ and $n$ such that
\eqlabel{equY'}{\frac{1}{m_H(V)}\int_V f(a_\bt uy)dm_H(u)=\int_X dm_Y+O\left(\cS(f)\zeta(b,e^{\frac{\lfloor\bt\rfloor}{2}})^{\del'}\right)}
for any $t\ge 0$, $f\in C_c^\infty(Y)$, and $y=gw(b)\hat{\Gamma}$ with $\|g\|\leq \zeta(b,e^{\frac{\lfloor\bt\rfloor}{2}})^{\del'}$ and $b\in\bT^d$. Here, the implied constant depends only on $m,n$, and $V$.
\end{thm}

We now give a sketch of the proof of Theorem \ref{mainthm'}. The proof is almost the same as the proof of Theorem \ref{mainthm}. We denote by $\mu_\bt=\mu_{y,\bt}\in\crly{P}(Y)$ the normalized probability measure on the orbit $a_\bt Vy$ and decompose $\mu_\bt$ by $\sum_i \pi_*\mu_{\bt,\om_i}(\om_i)\mu_{\bt,\om_i}$ with respect to a partition of unity $\set{\om_i}_{i\in I}$ of $X$, where $\mu_{\bt,\om_i}(f):=\pi_*\mu_\bt(\om_i)^{-1}\int f\om_i\circ\pi d\mu_\bt$. We also write $\nu_\bt:=\sigma_*\mu_\bt\in\crly{P}(\bT^d)$ and $\nu_{\bt,\om_i}:=\sigma_*\mu_{\bt,\om_i}\in\crly{P}(\bT^d)$. As in the proof of Theorem \ref{mainthm}, it suffices to prove that for each $\om_i$, any nontrivial Fourier coefficient $\widehat{\nu_{\bt,\om_i}}(\bm_0)$ is small for $\bm_0\in\bZ^d\setminus\set{0}$ (Proposition \ref{Fourdecay}).

Recall that we combined Proposition \ref{highcon} and Proposition \ref{lowcon} to prove Proposition \ref{Fourdecay}. Proposition \ref{highcon} claimed that if $\widehat{\nu_{t,\om_i}}(\bm_0)$ is large for some small $\bm_0$, then $\nu_{t-s}=\sigma_*\mu_{t-s}=\sigma_*(a_{-s}\mu_t)$ is highly concentrated for any $0<s\leq\kappa_5 t$. For general diagonal elements, we consider $\mu_\bt$ and $\mu_{\bt'}=a_{-s}\mu_\bt$ with $0<s<\frac{\lfloor\bt\rfloor}{d}$, instead of $\mu_t$ and $\mu_{t-s}=a_{-s}\mu_t$. Here we denote by $\bt'\in\mathfrak{a}^+$ the vector $(t_1-ns,\cdots,t_m-ns,t_{m+1}-ms,\cdots,t_d-ms)\in\bR^d$.

We remark that the preliminaries in Section 2 still hold for a general diagonal subgroup. Here we only state the following analog of Proposition \ref{inductive} which holds without any change of the proof. The rest of Section 2 holds by replacing the use of Theorem \ref{equX} with Theorem \ref{equXgen}.
\begin{prop}
For any $0\leq s<\frac{\lfloor\bt\rfloor}{d}$, $u_0\in V$, $y\in Y$, and $f\in L^\infty(Y)$, we have
$$\mu_{y,\bt}(f)=(a_su_0)_*\mu_{y,\bt'}(f)+O(\|f\|_{L^\infty}e^{\lfloor\bt\rfloor-ds}).$$
\end{prop}

Throughout Section 3 and Subsection 4.1, we investigated a relation between $\mu_{y,t}$ and $\mu_{y,t-s}$. We did not use any particular property of the measure $\mu_{y,t-s}$ and only used the fact that $\mu_{y,t-s}$ is approximately the $a_{-s}$-translate of $\mu_{y,t}$. Thus, the arguments in Section 3 and Subsection 4.1 still work for the measures $\mu_{y,\bt}$ and $\mu_{y,\bt'}$ without any modification. Hence, we get the following proposition as we obtained Proposition \ref{highcon}.

\begin{prop}\label{highcon'}
Let $\bt\in\mathfrak{a}^+$, $-\frac{1}{\kappa_4}\log c_{9}<s\leq\kappa_4\frac{\lfloor\bt\rfloor}{d}$, $e^{-\kappa_{4}^2 s}\leq r<c_{9}^{\kappa_4}$, and let $z$ and $\om$ be as in Proposition \ref{mainprop}. If $|\widehat{\nu_{\bt,\om}}(\bm_0)|\ge c_{3}e^{-\kappa_4^2s}$ for some $0<\|\bm_0\|<e^{\kappa_4s}$, then there exists $p\in\bT^d$ with
$$\nu_{\bt'}(B^{\bT^d}(p,e^{-ns}))>c_{10}e^{-\frac{\kappa_3}{20d}ns},$$
where $\nu_{\bt'}=\sigma_*\mu_{\bt'}=\sigma_*(a_{-s}\mu_{\bt})$.
\end{prop}

Now it remains to prove an analogue of Proposition \ref{lowcon}. We made use of the results of Subsection 3.1 for the proof of Proposition \ref{lowcon}, so we first need some modifications of Subsection 3.1. For 
\eq{\begin{aligned}
    \bt'&=(t_1-2ns,\cdots,t_m-2ns,t_{m+1}-2ms,\cdots,t_d-2ms)\\
    &=(t_1',\cdots,t_m',t_{m+1}',\cdots,t_d')\in\mathfrak{a}^+,
\end{aligned}}
we define two maps $\xi=\xi_\bt': X\times V\to\cF$ and $\gamma=\gamma_\bt': X\times V\to\Gamma$ as follows: For any $x\in X$ and $u\in V$, there uniquely exist $\xi(x,u)\in\cF$ and $\gamma(x,u)\in\Gamma$ such that
\eqlabel{defgam}{a_\bt' u\iota(x)=\xi(x,u)\gamma(x,u)}
by the definition of the fundamental domain $\cF$. 

For fixed $\bm_0\in\bZ^d\setminus\set{0}$, we define two maps $\bx:X\times V\to\bR^m$ and $\by:X\times V\to\bR^n$ satisfying the following:
\eqlabel{defxy'}{(\xi(x,u)^{tr})^{-1}\bm_0=\left(\begin{matrix} \bx(x,u)\\
\by(x,u)
\end{matrix}\right)\in\bR^m\times\bR^n.}
Let $\bx=(\bx_1,\cdots,\bx_m)$ and $\by=(\by_1,\cdots,\by_n)$.

Lemma \ref{siggam} and Lemma \ref{Vxe} still hold in this setting as follows.

\begin{lem}\label{siggam'}
For any $y\in Y$ and $u\in V$, we have $\sigma(a_{\bt'} uy)=\gamma(x,u)\sigma(y)$, where $x=\pi(y)$.
\end{lem}

\begin{lem}\label{Vxe'}
For $0<\eps\leq\frac{1}{2}$ and $x\in X$ denote by $V_{x,\eps}$ the set of elements $u\in V$ satisfying:
\begin{enumerate}
    \item $\|\xi(x,u)\|<\eps^{-1},$
    \item $|\bx_1(x,u)|> \eps^2\|\bm_0\|.$
\end{enumerate}
If $\eps>e^{-\kappa_2 \frac{\lfloor\bt'\rfloor}{d}}$ and $\textrm{ht}(x)<e^{\kappa_2 \frac{\lfloor\bt'\rfloor}{d}}$, then $m_H(V\setminus V_{x,\eps})\leq C_{13}\eps^{\frac{\kappa_3}{2}}$ for some constants $C_{13}>0$.
\end{lem}
In (2) of Lemma \ref{Vxe'}, we replaced $\|\bx(x,u)\|$ of Lemma \ref{Vxe} by $|\bx_1(x,u)|$. However, the proof of Lemma \ref{Vxe} still works if one replace $\set{\mb{0}_m}\times\bR^n$ in the definition of $\cM$ by $\set{0}\times\bR^{d-1}$.

Denote by 
$\mathfrak{B}(\bt',D)$ the box $$[-De^{t_1'},De^{t_1'}]\times\cdots\times[-De^{t_m'},De^{t_m'}]\times[-De^{t_1'},De^{t_1'}]\times\cdots[-De^{t_1'},De^{t_1'}]\subset\bR^d$$ for $C>0$. The following is a generalization of Proposition \ref{key} and here is the place where a nontrivial modification is needed.

\begin{prop}\label{key'}
Let $s>0$ and $\bm_0\in\bZ^d\setminus\set{0}$. For any $e^{-\kappa_2 s}\leq\eps\leq\frac{1}{2}$ and $x\in X$ with $\textrm{ht}(x)<e^{\kappa_2 s}$, let $V_{x,\eps}$ be the set as in Lemma \ref{Vxe}. Then
\begin{enumerate}
    \item $\set{\gamma(x,u)^{tr}\bm_0\in \bZ^d: u\in V_{x,\eps}}\subseteq\iota(x)^{tr}\mathfrak{B}(\bt',C_{14}\eps^{-1}\|\bm_0\|)$,
    \item For any $\bm\in\bZ^d$,
    $$m_H(\set{u\in V_{x,\eps}: \gamma(x,u)^{tr}\bm_0=\bm})\leq C_{16}\eps^{-3n}e^{-nt_1'-(t_{m+1}'+\cdots+t_d')}.$$
\end{enumerate}
\end{prop}
We remark that the estimate in this proposition is still \textit{tight} as Proposition \ref{key}: If we consider $\eps$ and $\textrm{ht}(x)$ as constants, as the points $\gamma(x,u)^{tr}\bm_0$ belongs to $\iota(x)^{tr}\mathfrak{B}(\bt',C_{14}\eps^{-1}\|\bm_0\|)$ and the latter contains $\asymp e^{nt_1'+(t_1'+\cdots+t_m')}$ points, the estimate in (2) is compatible with each $\bm\in \iota(x)^{tr}\mathfrak{B}(\bt',C_{14}\eps^{-1}\|\bm_0\|)$ obtaining the same weight $e^{-nt_1'-(t_1'+\cdots+t_m')}\asymp e^{-nt_1'-(t_{m+1}'+\cdots+t_d')}$.
\begin{proof}
Recall that there is the canonical measure-preserving bijection $\varphi$ between $M_{m,n}(\bR)$ and $H$ defined by $\varphi(A)=\left(\begin{matrix} \Id_m & A\\
0 & \Id_n\end{matrix}\right)$, there exists a constant $C_{14}>1$ such that $\varphi(B^{\bR^{mn}}(0,C_{14}^{-1}))\subseteq V\subseteq \varphi(B^{\bR^{mn}}(0,C_{14}))$, and $\|u\|\leq C_{14}$ for any $u\in V$.

(1) Recall that we have $\|\xi(x,u)\|\leq \eps^{-1}$, $\|u\|\leq C_{14}$. Hence, $\|\bx(x,u)\|$ and $\|\by(x,u)\|$ are bounded by $\eps^{-1}\|\bm_0\|$. By explicit matrix computation, it follows that
$$u^{tr}a_{\bt'}^{tr}(\xi(x,u)^{tr})^{-1}\bm_0\in \mathfrak{B}(\bt',C_{14}\eps^{-1}\|\bm_0\|).$$
By the definition of $\xi$ and $\gamma$ we have $\xi(x,u)^{-1}a_{\bt'} u\iota(x)=\gamma(x,u)$. Therefore,
$$\gamma(x,u)^{tr}\bm_0=\iota(x)^{tr}u^{tr}a_{\bt'} ^{tr}(\xi(x,u)^{tr})^{-1}\bm_0\in \iota(x)^{tr}\mathfrak{B}(\bt',C_{14}\eps^{-1}\|\bm_0\|).$$

(2) The equation $\gamma(x,u)^{tr}\bm_0=\bm$ can be written
\eqlabel{meqn'}{a_{\bt'}^{tr}(\xi(x,u)^{tr})^{-1}\bm_0=(u^{tr})^{-1}(\iota(x)^{tr})^{-1}\bm.} Let $v_{+}\in\bR^m$ and $v_{-}\in\bR^n$ be the vectors such that $\left(\begin{matrix} v_{+}\\
v_{-}\end{matrix}\right)=(\iota(x)^{tr})^{-1}\bm$. Here, $v_{+}$ and $v_{-}$ only depend on $\bm$, not on $u$.
For any $A\in \varphi^{-1}(V)$, by a straightforward computation we have 
\eqlabel{Aeqn'}{
\begin{aligned}
    (u^{tr})^{-1}(\iota(x)^{tr})^{-1}\bm&=\left(\begin{matrix} v_{+}\\
-A^{tr}v_{+}+v_{-}\end{matrix}\right),
\end{aligned}} where $u=\varphi(A)$. On the other hand, the left hand side of \eqref{meqn'} is equal to $(e^{t_1'}\bx_1(x,u),\cdots,e^{t_m'}\bx_m(x,u),e^{-t_{m+1}'}\by_1(x,u),\cdots,e^{-t_d'}\by_n(x,u))$.
Combining this with \eqref{Aeqn'}, if $u=\varphi(A)$ is a solution of \eqref{meqn'}, then
\eqlabel{strip'}{
\begin{aligned}
    A^{tr}v_{+}-v_{-}\in\mathfrak{T}(\bt',\|\by(x,u)\|),
\end{aligned}
}
denoting by $\mathfrak{T}(\bt',D)$ the small box
$$[-De^{-t_{m+1}'},De^{-t_{m+1}'}]\times\cdots\times[-De^{-t_d'},De^{-t_d'}]\subset\bR^n.$$
If $u\in V_{x,\eps}$, we have $\|\bx_1(x,u)\|> \eps^{2}\|\bm_0\|$ and 
$$\|\by(x,u)\|\leq \|(\xi(x,u)^{tr})^{-1}\bm_0\|\leq C_3\|\xi(x,u)\|\|\bm_0\|\leq C_3\eps^{-1}\|\bm_0\|$$ by \eqref{defxy'} and the definition of $V_{x,\eps}$ in Lemma \ref{Vxe'}. Moreover, $\|v_{+}\|>\eps^{2}\|\bm_0\|e^{t_1'}$ since $\|v_{+}\|\ge e^{t_1'}|\bx_1(x,u)|$. To sum up, a solution $u=\varphi(A)$ must be in a thin tube $\set{A\in M_{m,n}(\bR): A^{tr}v_{+}-v_{-}\in\mathfrak{T}(\bt',C_3\eps^{-1}\|\bm_0\|)}$, where $\|v_{+}\|>\eps^{2}\|\bm_0\|e^{t_1'}$. We can consider $A$ as in the Euclidean space $\bR^{mn}$, then the volume of the intersection of this tube and the ball $B^{\bR^{mn}}(0,C_{14})$ is \\$\ll\displaystyle\prod_{j=1}^n\left(C_{14}^{-1}\|v_{+}\|^{-1}(C_3\eps^{-1}\|\bm_0\|e^{-t_{m+j}'})\right)\ll \eps^{-3n}e^{-nt_1'-(t_{m+1}'+\cdots+t_d')}$. Therefore, for any $\bm\in\bZ^{d}$, the $m_H$-measure of the set of $u\in V_{x,\eps}$ satisfying $\gamma_x(u)^{tr}\bm_0=\bm$ is less than $C_{16}\eps^{-3n}e^{-nt_1'-(t_{m+1}'+\cdots+t_d')}$ for some constant $C_{16}>1$.
\end{proof}

We are now ready to prove an analog of Proposition \ref{lowcon}. The proof is almost same with Proposition \ref{lowcon}, but here we count the number of certain points in $\iota(x)^{tr}\mathfrak{B}(\bt',C_{14}\eps^{-1})$, instead of $B^{\bZ^d}(0,R)$.
\begin{prop}\label{lowcon'}
Let $0<s<\frac{\lfloor\bt\rfloor}{2d}$ and $\rho\ge e^{-\kappa_4\frac{\lfloor\bt\rfloor}{d}}$. For $g_0\in G$ and $b_0\in\bT^d$, let $y_0=g_0w(b_0)\hat{\Gamma}\in Y$. Let $x_0=\pi(y_0)\in X$ and assume $\textrm{ht}(x_0)\leq \rho^{-\frac{1}{10d^2(d-1)}}$. If $\zeta(b_0,\|g_0\|^{-1}e^{n(\frac{\lfloor\bt\rfloor}{d}-s)})\ge c_{11}\rho^{-2d}$, then
$$\nu_{\bt'}(B^{\bT^d}(p,\rho))\leq C_{22}\rho^{\frac{\kappa_3}{10d}}$$
for any $p\in\bT^d$.
\end{prop}

\begin{proof}
We define $\Theta:\bT^d\to\bT$ and $\Xi=\Theta(B^{\bT^d}(p,\rho))$ as in Proposition \ref{lowcon'}, then we also have
$$\nu_{\bt'}(\Theta^{-1}(\Xi))=\set{u\in V: \gamma(x_0,u)^{tr}e_1\cdot b\in\Xi}.$$
Let $\eps=\rho^{\frac{1}{5d}}$ and $V_{x_0,\eps}$ be the set as in Lemma \ref{Vxe'}. Note that the conditions $\eps=\rho^{\frac{1}{5d}}>e^{-\kappa_2\frac{\lfloor\bt'\rfloor}{d}}$ and $\textrm{ht}(x_0)<e^{\kappa_2\frac{\lfloor\bt'\rfloor}{d}}$ of Lemma \ref{Vxe'} and Proposition \ref{count} are satisfied since $\kappa_4\leq\frac{\kappa_2}{2000d^3}$ and $\lfloor\bt'\rfloor>\lfloor\bt\rfloor-2ds$. Then we have 
\eqlabel{estVxe'}{m_H(V\setminus V_{x_0,\eps})\leq C_{13}\eps m_H(V)} by Lemma \ref{Vxe'}. We now apply Proposition \ref{key'} for $\bm_0=e_1$. Denoting $\Omega:=\set{\bm\in \iota(x)^{tr}\mathfrak{B}(\bt',C_{14}\eps^{-1}): b\cdot\bm\in\Xi}$ and applying Proposition \ref{key'}, we have
\eqlabel{xiset3'}{\set{u\in V: \gamma(x_0,u)^{tr}e_1\cdot b\in\Xi}\subseteq (V\setminus V_{x,\eps})\cup\displaystyle\bigcup_{\bm\in\Omega}\set{u\in V_{x_0,\eps}:\gamma(x_0,u)e_1=\bm},}
\eqlabel{xiest'}{m_H(\set{u\in V_{x_0,\eps}:\gamma(x_0,u)e_1=\bm})\leq C_{16}\eps^{-3n}e^{-nt_1'-(t_1'+\cdots+t_m')}}
for any $\bm\in\bZ^d$.

Write $\mathfrak{B}:=\mathfrak{B}(\bt',C_{14}\eps^{-1})$ for simplicity. Since $\mathfrak{B}$ can be covered with $\ll e^{nt_1'+(t_1'+\cdots+t_m')-d\lfloor\bt'\rfloor}$ number of cubes with length $\eps^{-1}e^{\lfloor\bt'\rfloor}$, we can also cover $\iota(x)^{tr}\mathfrak{B}$ with $\iota(x)^{tr}Q_1',\cdots,\iota(x)^{tr}Q_N'$, where $Q_i'$'s are cubes with length $\eps^{-1}e^{\lfloor\bt'\rfloor}$ and $N\ll e^{nt_1'+(t_1'+\cdots+t_m')-d\lfloor\bt'\rfloor}$. Since $\|\iota(x)\|\leq C_{15}\textrm{ht}(x_0)^{d-1}$, the region $\iota(x)^{tr}\mathfrak{B}$ is covered with $Q_1,\cdots,Q_N$, where $Q_i$'s are cubes with length $R:=C_{15}\eps^{-1}\textrm{ht}(x_0)^{d-1}e^{\lfloor\bt'\rfloor}$.

Applying Lemma \ref{count} for each $Q_i$, we obtain
\eqlabel{Omega'}{
\begin{aligned}
|\Omega|\leq C_{21}\rho\displaystyle\sum_{i=1}^N|Q_i|&\ll \rho\eps^{-d}\textrm{ht}(x_0)^{d(d-1)}e^{d\lfloor\bt'\rfloor}N\\&=\rho\eps^{-d}\textrm{ht}(x_0)^{d(d-1)}e^{nt_1'+(t_1'+\cdots+t_m')}.
\end{aligned}}
Then the rest of the proof is the same as in Proposition \ref{lowcon}.
\end{proof}

Combining Proposition \ref{highcon'} and Proposition \ref{lowcon'}, we deduce the following Fourier decay estimate. 

\begin{prop}\label{Fourdecay''}
Let $\bt\in\mathfrak{a}^+$, $\|g_0\|\leq e^{\frac{\lfloor\bt\rfloor}{4}}$, $b\in\bT^d$, and $y_0=g_0w(b_0)\hat{\Gamma}$. Let $\rho=\max\left(e^{-\kappa_4 t},c_{11}^{-\frac{1}{2d}}\zeta(b_0,e^{\frac{\lfloor\bt\rfloor}{2}})^{-\frac{1}{2d}}\right)$ and $r=\rho^{\kappa_4^2}$. Assume $\rho>c_{12}$ and $\textrm{ht}(\pi(y_0))\leq\rho^{-\frac{1}{10d^2(d-1)}}$. Let $z$ and $\om$ be as in Proposition \ref{mainprop}. Then
$$|\widehat{\nu_{t,\om}}(\bm_0)|<c_{3}\rho^{\kappa_4^2}$$
for any $0<\|\bm_0\|<\rho^{-\kappa_4}$.
\end{prop}
Theorem \ref{mainthm'} follows from Proposition \ref{Fourdecay''} by the same procedure as in the proof of Theorem \ref{mainthm}.

\appendix
\section{Proofs of properties in §\ref{HH0H-} and §\ref{subfund}}\label{appendixA}
In this subsection, we verify the properties we stated in §\ref{HH0H-} and §\ref{subfund}.

We first verify \eqref{hdecompositionbound}. For $g=h_+h_0h_-$ let us write
$$h_+= \left(\begin{matrix} \Id_m& A_+\\
0&  \Id_n
\end{matrix}\right), \quad h_0=\left(\begin{matrix} Z_{11}& 0\\
0&  Z_{22}
\end{matrix}\right), \quad h_-=\left(\begin{matrix} \Id_m& 0\\
A_-&  \Id_n
\end{matrix}\right)$$
for $A_+\in \operatorname{Mat}_{m,n}(\bR)$, $A_-\in\operatorname{Mat}_{n,m}(\bR)$, $Z_{11}\in\operatorname{GL}_m(\bR)$, $Z_{22}\in\operatorname{GL}_n(\bR)$, and $\operatorname{det}(Z_{11})\operatorname{det}(Z_{22})=1$. Then by a straightforward calculation, we have
$$g=\left(\begin{matrix} Z_{11}+A_+Z_{22}A_-& A_+Z_{22}\\
Z_{22}A_-&  Z_{22}
\end{matrix}\right),\;\; g^{-1}=\left(\begin{matrix} Z_{11}^{-1}& -Z_{11}^{-1}A_+\\
-A_-Z_{11}^{-1}&  Z_{22}^{-1}+A_-Z_{11}^{-1}A_+
\end{matrix}\right).$$
It follows that all matrix coefficients of $A_+Z_{22}$, $Z_{22}$, $Z_{11}^{-1}$, and $A_-Z_{11}^{-1}$ are bounded $\ll \|g\|$. We now bound the matrix coefficients of $Z_{22}^{-1}$ and $Z_{11}$. We may write $Z_{22}^{-1}=\operatorname{det}(Z_{22})^{-1}\operatorname{adj}(Z_{22})$. Since $\operatorname{det}(Z_{22})$ is a degree $n$ polynomial of variables $(g_{ij})_{m+1\leq i,j\leq n}$ and $\cE$ is the locus of $\operatorname{det}(Z_{22})$, we have $|\operatorname{det}(Z_{22})|\gg \min\set{\operatorname{dist}(g,\cE),1}^n$. It follows that all matrix coefficients of $Z_{22}^{-1}$ are bounded by
$\ll \max\set{\operatorname{dist}(g,\cE)^{-n},1}\|g\|^{n-1}$. Similarly, all coefficients of $Z_{11}$ are bounded by $\ll \max\set{\operatorname{dist}(g,\cE)^{-m},1}\|g\|^{m-1}$, using $\operatorname{det}(Z_{11}^{-1})=\operatorname{det}(Z_{22})$. Hence, we obtain 
$$\|h_0\|\ll \max\set{\operatorname{dist}(g,\cE)^{-(d-1)},1}\|g\|^{d-2}.$$ Since $A_+=(A_+Z_{22})Z_{22}^{-1}$ and $A_-=(A_-Z_{11}^{-1})Z_{11}$, using \eqref{norm} we also estimate $$\|h_+\|, \|h_-\| \ll \max\set{\operatorname{dist}(g,\cE)^{-(d-1)},1}\|g\|^{d-1}.$$ Thus \eqref{hdecompositionbound} is proved.

We now verify the properties in §\ref{subfund}. Recall that a continuous function $F:G\to\bR_{>0}$ is given by
$$F(g)^2=\frac{(\sum_{i,j}|g_{ij}|^2)(\sum_{i,j}|(g^{-1})_{ij}|^2)}{\sum_{i,j}|g_{ij}|^2+\sum_{i,j}|(g^{-1})_{ij}|^2},$$
and satisfies $F(g)\leq\min((\sum_{i,j}|g_{ij}|^2)^{\frac{1}{2}},(\sum_{i,j}|(g^{-1})_{ij}|^2)^{\frac{1}{2}})\leq 2F(g)$ for any $g\in G$. Since the maps $g\mapsto \sum_{i,j}|g_{ij}|^2$ and $g\mapsto \sum_{i,j}|(g^{-1})_{ij}|^2$ are proper, $F$ is also proper. We will construct a fundamental domain $\cF$ of $X=G/\Gamma$ so that it consists of every elements in $\cF^\circ:=\set{g\in G: F(g)< F(g\gamma),\;\forall \gamma\in\Gamma\setminus\set{\operatorname{id}}}$ and properly chosen elements from its boundary. We can observe that the closure of $\cF^\circ$ is contained in the set $\set{g\in G: F(g)\leq F(g\gamma),\;\forall \gamma\in\Gamma\setminus\set{\operatorname{id}}}$ and the boundary $\overline{\cF^\circ}\setminus\cF^\circ$ has measure zero with respect to $m_G$.

(1) Let us first show that for any $x\in X$, there exists $g\in\overline{\cF^\circ}$ such that $x=g\Gamma$. For any $x\in X$, let $\Lambda_x$ be the corresponding lattice in $\bR^d$, and $\Lambda_x^*$ be the dual lattice of $\Lambda_x$, i.e.
\eq{\Lambda_x^*:=\set{v^*\in\bR^d: v^*\cdot v\in\bZ^d, \quad \forall v\in\Lambda_x}.}
 We refer \cite{Cas71} for basic properties of the dual lattice. For any $g\in \phi^{-1}(x)$, the column vectors of $g$ are vectors in $\Lambda_x$ and the row vectors of $g^{-1}$ are vectors in $\Lambda_x^*$, where $\phi:G\to X$ is the canonical projection. 
Observe that for any $R>0$, there are only finitely many vectors in either $\Lambda_x$ or $\Lambda_x^*$ with Euclidean length $\leq R$. It follows that the function $F|_{\phi^{-1}(x)}:\phi^{-1}(x)\to\bR_{>0}$ that is restricted to the subset $\phi^{-1}(x)$ attains its minimum. Hence, for any $x\in X$, there exists $g\in\overline{\cF^\circ}$ such that $x=g\Gamma$. 

For each $\gamma\in\Gamma$, let $\cF_\gamma:=\set{g\in G: F(g)\leq F(g\gamma)}$. Note that $\overline{\cF}=\displaystyle\bigcap_{\gamma\in\Gamma}\cF_{\gamma}$, $\partial\cF_\gamma=\set{g\in G: F(g)=F(g\gamma)}$, and $\partial\cF\subseteq\displaystyle\bigcup_{\gamma\in\Gamma}\partial\cF_\gamma$. If $g$ is in the interior $\cF^\circ$, then $F(g)<F(g\gamma)$ for any $\gamma\in\Gamma$. Hence, for any $x\in \phi(\cF^\circ)$, there uniquely exists $g\in \cF^\circ$ such that $x=g\Gamma$, where $\phi: G\to X$ is the canonical projection map. Let $\iota(x)$ be the unique element in $\cF^\circ$, then $\iota:\phi(\cF^\circ)\to\cF^\circ$ is a continuous measure-preserving map satisfying $\phi\circ\iota=\operatorname{id}_{\phi(\cF^\circ)\to\phi(\cF^\circ)}$.

We can choose elements from the boundary so that for any $x\in X$ there uniquely exists $g\in\cF$ such that $x=g\Gamma$. Since the $m_G$-measure of the boundary $\overline{\cF^\circ}\setminus\cF^\circ$ is zero, the map $\iota$ is extended to be a measure-preserving map from $X$ to $\cF$. It also follows that $m_G(\cF)=m_G(\cF^\circ)=m_X(\phi(\cF^\circ))=m_X(X)=1$.

(2) We next prove the estimate \eqref{iotaht}. Let $\lambda_j(\Lambda)$ denote the $j$-th successive minimum of a lattice $\Lambda\subset\bR^d$ i.e. the infimum of $\lambda$ such that the ball $B^{\bR^d}(0,\lambda)$ contains $j$ independent vectors. By Minkowski's second theorem \cite[Theorem I in Chapter VIII]{Cas71}, $\lambda_1(\Lambda_x)\lambda_2(\Lambda_x)\cdots\lambda_d(\Lambda_x)\asymp 1$ for any $x\in X$. 

Since $\lambda_j(\Lambda_x)\ge\lambda_1(\Lambda_x)\ge\textrm{ht}(x)^{-1}$ for any $1\leq j\leq d-1$, we have $\lambda_d(\Lambda_x)\ll\textrm{ht}(x)^{d-1}$. For $x\in\phi(\cF^\circ)$ let $g=\iota(x)$, i.e. $g\in\cF$ is the element with $x=g\Gamma$. Then $\lambda_d(\Lambda_x)\ll\textrm{ht}(x)^{d-1}$ implies that there exists $\gamma\in\Gamma$ such that the Euclidean length of each column vectors of $g\gamma$ is $\ll\textrm{ht}(x)^{d-1}$. Hence, We have $|g_{ij}|\leq F(g)\leq F(g\gamma)\ll \textrm{ht}(x)^{d-1}$ for any $i,j$. In order to estimate the components of $g^{-1}$, we may consider the dual lattice $\Lambda_x^*$. By Mahler's inequality \cite[Theorem VI in Chapter VIII]{Cas71}, $\lambda_1(\Lambda_x)\lambda_d(\Lambda_x^*)\ll 1$, so we have $\lambda_d(\Lambda_x^*)\ll \lambda_1(\Lambda_x)\leq \textrm{ht}(x)$. In other words, there exists $\gamma\in\Gamma$ such that the Euclidean length of each row vectors of $(g\gamma)^{-1}$ is $\ll\textrm{ht}(x)$. It follows that $|(g^{-1})_{ij}|\leq F(g)\leq F(g\gamma)\ll\textrm{ht}(x)$ for any $i,j$. Thus, $\|\iota(x)\|\ll\textrm{ht}(x)^{d-1}$ for any $x\in \phi(\cF^\circ)$.

(3) We prove the estimate \eqref{fundest}. We first show that 
$$m_G(\set{g\in\cF: \textrm{dist}(g,\partial\cF)<r})\ll r^{\frac{1}{2(d-1)^2+1}}.$$ For each $R>0$, denote by $\Upsilon(R)$ the set of elements $g\in G$ such that $(\sum_{i,j}|g_{ij}|^2)^{\frac{1}{2}}\leq R$. We make use of the following volume estimate of $\Upsilon(R)$ and counting estimate of $\Upsilon(R)\cap\Gamma$ (See \cite[Example 1.6 and Appendix 1]{DRS93} or \cite{EM93}):
\eqlabel{VolCountUp}{m_G(\Upsilon(R))\asymp R^{\frac{d(d-1)}{2}}, \quad |\Upsilon(R)\cap\Gamma|\asymp R^{\frac{d(d-1)}{2}}.}

In (2) we showed that if $\textrm{ht}(x)\leq\eps^{-1}$, then the column vectors of $g=\iota(x)$ has Euclidean length $\ll\eps^{-(d-1)}$, so $(\sum_{i,j}|g_{ij}|^2)^{\frac{1}{2}}\ll\eps^{-(d-1)}$. It means that the set $\cF\setminus\Upsilon(R)$ is contained in $\set{g\in\cF: \textrm{ht}(g\Gamma)\gg R^{\frac{1}{d-1}}}$, hence
\eqlabel{Upvolest}{m_G(\cF\setminus\Upsilon(R))\ll R^{-\frac{d}{d-1}}.}

Observe that if $g\in\Upsilon(R)$ and $\gamma\notin\Upsilon(R^3)$, then $F(g)<F(g\gamma)$, so $g$ is not on the boundary $\partial\cF_\gamma$. It follows that $$\partial\cF\cap\Upsilon(R)\subseteq\displaystyle\bigcup_{\gamma\in\Upsilon(R^3)\cap\Gamma}(\partial\cF_\gamma \cap \Upsilon(R)).$$ Denote by $\cF(r)$ and $\cF_\gamma(r)$ the set of $g$ such that $\textrm{dist}(g,\partial\cF)<r$ and $\textrm{dist}(g,\partial\cF_\gamma)<r$, respectively. Note that each $\partial\cF_\gamma$ is the set of zeros of the equation $S_\gamma(g)=0$, where $S_\gamma(g)=F(g\gamma)-F(g)$, hence a $(\dim G-1)$-dimensional hypersurface of $G$. One can check that the first derivative $\triangledown S_\gamma (g):\mathfrak{g}\to\mathfrak{g}$ is bounded away from zero uniformly on $\gamma\in \Gamma\setminus\set{\operatorname{id}}$. Moreover, if $g\in\Upsilon(R)$ and $\gamma\notin\Upsilon(R^3)$, then the norm of the derivatives of higher degrees are bounded above by $\ll\|g\|^2+\|g\gamma\|^2\ll R^8$. Thus we have $S_\gamma(\operatorname{exp}(v)g)=S_\gamma(g)+\triangledown S_\gamma (g)\cdot v+O(R^8\|v\|^2)$, hence $\partial\cF_\gamma$ is locally approximated in the neighborhood $B^G(g,R^{-10})$ by the zero set of the linear equation $S_\gamma(g)+\triangledown S_\gamma (g)\cdot v=0$.  
It follows that for $r<R^{-10}$, $$m_G(\cF_\gamma(r)\cap\Upsilon(R))\ll \frac{r}{R^{-10}}m_G(\cF_\gamma(R^{-10})\cap\Upsilon(R))\ll  R^{\frac{d(d-1)}{2}+10}r$$ by \eqref{VolCountUp}. Therefore, using \eqref{VolCountUp} and \eqref{Upvolest}, we have
\eq{\begin{aligned}
m_G(\cF(r))&=m_G(\cF(r)\cap\Upsilon(R))+m_G(\cF(r)\setminus\Upsilon(R))\\
&\leq m_G(\displaystyle\bigcup_{\gamma\in\Upsilon(R^3)\cap\Gamma}(\cF_\gamma(r) \cap \Upsilon(R)))+m_G(\cF\setminus\Upsilon(R))\\
&\ll\displaystyle\sum_{\gamma\in\Upsilon(R^3)\cap\Gamma}m_G(\cF_\gamma(r)\cap\Upsilon(R))+R^{-\frac{d}{d-1}}\\
&\ll |\Upsilon(R^3)\cap\Gamma|R^{\frac{d(d-1)}{2}+10}r+R^{-\frac{d}{d-1}}\\
&\ll R^{2d(d-1)+10}r+R^{-\frac{d}{d-1}}
\end{aligned}}
for $r<R^{-10}$. Choosing $R$ to satisfy $r=R^{-10d(d-1)}$, we obtain $$m_G(\cF(r))\ll r^{\frac{1}{10d^2}}.$$ Combining with \eqref{cspmsr}, we have shown that
\eqlabel{appendix1}{m_G(\cF\setminus\cF_0(r,\eps))\ll r^{\frac{1}{10d^2}}+\eps^d,}
where we denote $\cF_0(r,\eps):=\set{g\in\cF: \textrm{ht}(g\Gamma)\leq \eps^{-1}, \textrm{dist}(g,\partial\cF)\ge r}$.

Denote by $\cE(r)$ and $\cE^{-1}(r)$ by the set of $g$ such that $\operatorname{dist}(g,\cE)<r^{\frac{1}{20d}}$ and $\operatorname{dist}(g^{-1},\cE)=\operatorname{dist}(g,\cE^{-1})<r^{\frac{1}{20d}}$, respectively. It is now sufficient to estimate $m_G\big(\cE^{-1}(r)\cap \cF\big)=m_G\big(\cE(r)\cap (\cF)^{-1}\big)$. Note that $r^{\frac{1}{40d}}$-neighborhood of $(\cF_0(r^{\frac{1}{50d}}, C_4r^{\frac{1}{1000d^2}}))^{-1}$ is still contained in $(\cF_0(r^{\frac{1}{4}}, C_4r^{\frac{1}{20d}}))^{-1}$. Since $\cE$ is a $(d^2-2)$-dimensional analytic closed submanifold of $G$, we have \eq{\begin{aligned}
    m_G\big(\cE(r)\cap (\cF_0(r^{\frac{1}{50d}}, &C_4r^{\frac{1}{1000d^2}}))^{-1}\big)\\&\ll r^{\frac{1}{40d}} m_G\big(\cE(r^{\frac{1}{2}})\cap (\cF_0(r^{\frac{1}{4}}, C_4r^{\frac{1}{20d}}))^{-1}\big)\ll r^{\frac{1}{40d}}.
\end{aligned}}
By \eqref{appendix1} we also have $m_G\big(\cF\setminus \cF_0(r^{\frac{1}{50d}},C_4r^{\frac{1}{1000d^2}})\big)\ll r^{\frac{1}{100d^3}}$. It follows that $m_G\big(\cE^{-1}(r)\cap \cF\big)\ll r^{\frac{1}{100d^3}}$. Let $\kappa_3:=\frac{1}{100d^3}$. Again using \eqref{appendix1} we obtain
\eq{\begin{aligned}
    m_G(\cF\setminus\cF(r,\eps))&\leq m_G(\cF\setminus\cF_0(r,\eps))+ m_G(\cE^{-1}(r)\cap \cF)\\&\ll r^{\frac{1}{10d^2}}+\eps^d+r^{\frac{1}{100d^3}}\ll \max\set{r^{\kappa_3},\eps^d}.
\end{aligned}}

\section{Effective Weyl's criterion on $\bT$}
This subsection is devoted to proving the estimate \eqref{Weyl} using an effective version of Weyl's criterion. Neither the result nor the proof is new, but we prove it here to express the estimate in a suitable form using the function $\zeta$. To prove the estimate \eqref{Weyl} it is enough to show the following statement.
\begin{lem}\label{Weyl''}
Let $\Xi=[x_0-\rho,x_0+\rho]$ be an interval in $\bT$, where $x\in\bT$ and $0<\rho<\frac{1}{2}$. For any $\alpha\in\bT\setminus\bQ$ and $T\in\bN$, 
\eqlabel{Weyl'}{\frac{1}{T}|\set{0\leq k\leq T-1: k\alpha\in\Xi}|\ll\rho+\rho^{-1}\zeta(\alpha,T)^{-1},}
where $\zeta(\alpha, T)=\min\set{N\in\bN: \min_{1\leq q\leq N}\|q\alpha\|_{\bZ}\leq \frac{N^2}{T}}$.
\end{lem}
\begin{proof}
For any $0<\eps<\rho$, there exists a non-negative approximation function $\psi_\rho\in C_c^{\infty}(\bT)$ such that $\Supp \psi_\rho\subseteq [-\rho,\rho]$, $\int_{\bT}\psi_\rho(x)dx=1$, and $\|\psi_\rho''\|_{\infty}\ll \rho^{-2}$. Then for any $m\in\bZ\setminus\set{0}$, we have $|\widehat{\psi_\rho}(m)|\ll\min(1,|m|^{-2}\rho^{-2})$. Let $\phi=\psi_\rho*\mathds{1}_{[x_0-2\rho,x_0+2\rho]}$, then $\phi(x)\ge\mathds{1}_{\Xi}(x)$ for any $x\in\bT$. We also have 
$$|\widehat{\phi}(m)|\ll \rho\min(1,|m|^{-2}\rho^{-2})=\min(\rho,|m|^{-2}\rho^{-1})$$ for $m\in\bZ\setminus\set{0}$. In order to get an upper bound of the left-hand side of \eqref{Weyl'}, it suffices to estimate $\frac{1}{T}\displaystyle\sum_{0\leq k\leq T-1}\phi(k\alpha)$. We can expand this summation as follows using the Fourier expansion of $\phi$:
\eq{
\begin{aligned}
\frac{1}{T}\displaystyle\sum_{0\leq k\leq T-1}\phi(k\alpha)&=\frac{1}{T}\displaystyle\sum_{m\in\bZ}\widehat{\phi}(m)\displaystyle\sum_{0\leq k\leq T-1}e^{2\pi imk\alpha}\\
&=\widehat{\phi}(0)+\frac{1}{T}\displaystyle\sum_{0<|m|< \zeta(\alpha,T)}\widehat{\phi}(m)\frac{1-e^{2\pi i mT\alpha}}{1-e^{2\pi i m\alpha}}\\&+\displaystyle\sum_{|m|\geq\zeta(\alpha,T)}\widehat{\phi}(m)\left(\frac{1}{T}\displaystyle\sum_{0\leq k\leq T-1}e^{2\pi imk\alpha}\right).
\end{aligned}}

We now estimate each term as follows:
\eq{\widehat{\phi}(0)=\int_{\bT}\mathds{1}_{[x_0-2\rho,x_0+2\rho]}(x)dx=4\rho,}
\eq{\frac{1}{T}\displaystyle\sum_{0<|m|< \zeta(\alpha,T)}\widehat{\phi}(m)\frac{1-e^{2\pi i mT\alpha}}{1-e^{2\pi i m\alpha}}\ll\rho\zeta(\alpha,T)^{-1},}
\eq{\begin{aligned}
\displaystyle\sum_{|m|\geq\zeta(\alpha,T)}\widehat{\phi}(m)\left(\frac{1}{T}\displaystyle\sum_{0\leq k\leq T-1}e^{2\pi imk\alpha}\right)&\ll\rho^{-1}\displaystyle\sum_{|m|\geq\zeta(\alpha,T)}|m|^{-2}\\&\ll\rho^{-1}\zeta(\alpha,T)^{-1},
\end{aligned}}
using $|\widehat{\phi}(m)|\ll \rho$ and $\left|\frac{1-e^{2\pi i mT\alpha}}{1-e^{2\pi i m\alpha}}\right|\ll \|m\alpha\|_{\bZ}^{-1}\leq T(\zeta(\alpha,T)-1)^{-2}$ for the second estimate, and $|\widehat{\phi}(m)|\ll\rho^{-1}|m|^{-2}$ and $\left|\frac{1}{T}\displaystyle\sum_{0\leq k\leq T-1}e^{2\pi imk\alpha}\right|\leq 1$ for the third estimate. Hence, we obtain
\eq{\begin{aligned}
\frac{1}{T}|\set{0\leq k\leq T-1: k\alpha\in\Xi}|&\leq\frac{1}{T}\displaystyle\sum_{0\leq k\leq T-1}\phi(k\alpha)\ll \rho+\rho^{-1}\zeta(\alpha,T)^{-1}.
\end{aligned}}
\end{proof}
\def\cprime{$'$} \def\cprime{$'$} \def\cprime{$'$}
\providecommand{\bysame}{\leavevmode\hbox to3em{\hrulefill}\thinspace}
\providecommand{\MR}{\relax\ifhmode\unskip\space\fi MR }
\providecommand{\MRhref}[2]{%
  \href{http://www.ams.org/mathscinet-getitem?mr=#1}{#2}
}

\end{document}